\documentclass[11pt, leqno]{article}

\usepackage{amssymb}
\usepackage{amsmath}
\usepackage{amsthm}
\usepackage[scale=0.8]{geometry}

\usepackage{mathrsfs}

\usepackage{mathtools}

\usepackage{hyperref}

\allowdisplaybreaks[1]

\theoremstyle{plain}
\newtheorem{theorem}{Theorem}[section]
\newtheorem{lemma}[theorem]{Lemma}
\newtheorem{corollary}[theorem]{Corollary}
\newtheorem{proposition}[theorem]{Proposition}
\theoremstyle{definition}
\newtheorem{remark}[theorem]{Remark}
\newtheorem{definition}[theorem]{Definition}

\newtheorem{condition}[theorem]{Condition}

\numberwithin{equation}{section}
\makeatletter
\def\address#1#2{\begingroup
\noindent\parbox[t]{8.8cm}{%
\small{\scshape\ignorespaces#1}\par\vskip1ex
\noindent\small{\itshape E-mail address}%
\/: #2\par\vskip4ex}\hfill%
\endgroup}%
\makeatother

\newcommand{\floor}[1]{\lfloor #1 \rfloor}
\newcommand{\ceiling}[1]{\lceil #1 \rceil}

\newcommand{\vep}{\varepsilon}

\newcommand{\RR}{{\mathbb R}}
\newcommand{\BB}{{\mathbb B}}
\newcommand{\ZZ}{\mathbb{Z}}

\newcommand{\tal}{\tilde{\alpha}}
\newcommand{\ta}{\tilde{a}}

\newcommand{\tmkm}{\tau^m_{k-1}}
\newcommand{\tmk}{\tau^m_k}
\newcommand{\tmim}{\tau^m_{i-1}}
\newcommand{\tmi}{\tau^m_i}
\newcommand{\tmjm}{\tau^m_{j-1}}
\newcommand{\tmj}{\tau^m_j}
\newcommand{\tml}{\tau^m_l}
\newcommand{\tmlm}{\tau^m_{l-1}}

\newcommand{\IteratedIntegrals}{\mathcal{I}^\ast}
\newcommand{\spanIteratedIntegrals}{\mathcal{I}}
\newcommand{\goodClass}{\mathcal{J}}

\newcommand{\Dm}{D_m}
\newcommand{\cP}{\mathcal{P}}

\newcommand{\hyp}{\mathchar`-}

\newcommand{\tpsi}{\tilde{\psi}}

\newcommand{\trho}{\tilde{\rho}}

\newcommand{\tf}{\tilde{f}}

\title{H\"older estimates and weak convergences of\\ certain weighted sum processes}
\author{Shigeki Aida and Nobuaki Naganuma}
\date{}
\begin{document}

\maketitle

\begin{abstract}
	We study weighted sum processes associated
	to elements in a Wiener chaos with fixed order.
	More precisely, we show H\"older estimates and a functional limit theorem for them.
	Main tools we use are the integration by parts formula in Malliavin calculus,
	the fourth moment theorem, and estimates in multidimensional Young integrals.
\end{abstract}


\textbf{MSC2020 subject classifications}: 60F05; 60H35; 60G15.

\bigskip

\textbf{Funding}: This work was partially
supported by JSPS KAKENHI Grant Numbers JP20H01804 and JP22K13932.

\textbf{Statements and Declarations}: The authors declare that they have no competing interest.

\tableofcontents

\section{Introduction}
In this paper, we study weighted sum processes,
which arise naturally from study of approximation
of solutions to rough differential equations (RDEs)
driven by fractional Brownian motions (fBms). 
We can determine asymptotic errors of the approximate solutions
by showing a version of the functional central limit theorem (FCLT)
for the weighted sum processes
\cite{aida-naganuma2020, aida-naganuma2023approach, 
GradinaruNourdin2009, hu-liu-nualart2016, hu-liu-nualart2021, liu-tindel, liu-tindel2020, naganuma2015}.
First, let us recall basic known results related to the limit theorem
of weighted sum processes and its applications 
to the study of asymptotic errors of the approximate solutions.

For random variables $\{F^m_i\}_{i=1}^{2^m}$ and $\{Z^m_i\}_{i=1}^{2^m}$,
we call the following process in time parameter $t\geq 0$ a weighted sum process
\begin{align*}
	\frac{1}{\sqrt{2^m}}
	\sum_{i=1}^{\floor{2^mt}}
		F^m_i
		Z^m_i.
\end{align*}
Here $m$ is a positive integer, and $\floor{2^mt}$ is the integer part of $2^mt$.
Also, $F^m_i$ and $Z^m_i$ might depend not only on $i$ but also on $m$.
We refer to $\{Z^m_i\}_{i=1}^{2^m}$ as an integrator and regard $F^m_i$ as a weight for $Z^m_i$.
In the present paper, we investigate the weighted sum process
in the case where $F^m_i$ and $Z^m_i$ are given by functionals of
one/multi-dimensional fBms $B$ with Hurst parameter $0<H<1$.

The simplest case is that in which $F^m_i=1$ (the sum process is ``weight-free'')
and $\{Z^m_i\}$ are independent and identically distributed random variables 
that are independent of $m$.
Then an FCLT of the ``weight-free'' sum process is 
nothing but Donsker's invariance principle.
We next focus on the case where $\{Z^m_i\}$ are not independent.
Breuer-Major~\cite{BreuerMajor1983} showed an FCLT
in the case where $F^m_i=1$ and $Z^m_i=h(X_i)$.
Here, $\{X_i\}_{i=1}^\infty$ is a sequence of stationary Gaussian random variables
with nice covariance structure,
and $h$ is a nice real-valued function satisfying conditions stated in terms of Hermite polynomials.
Their proof of weak convergence of the finite-dimensional distributions
is based on the moment method. They calculated moments of all orders.
After decades, a new method for this proof was given by 
Nualart and Peccati \cite{NualartPeccati2005}:
it is known as the fourth moment theorem or the Nualart-Peccati criterion.
The fourth moment theorem characterizes weak convergence of random variables
by the convergence of their second and fourth moments
if they belong to a Wiener chaos with fixed order.
Using this theorem, we can treat the case where 
$F^m_i=1$ and $Z^m_i=H_q(B_{i-1,i})$ very easily.
Here $H_q$ is the $q$-th Hermite polynomial,
$B$ a one-dimensional fBm, and $B_{s,t}=B_t-B_s$.
In fact, this is a special case of Breuer and Major's result.
A multidimensional version of the fourth moment theorem \cite{PeccatiTudor2005}
provides a criterion for independence of limit random variables.
As an application, 
Nourdin-Nualart-Tudor \cite{NourdinNualartTudor2010} showed central and non-central
limit theorems
for the case where $F^m_i=f(B_{\tmim})$ and $Z^m_i=H_q(2^{mH}B_{\tmim,\tmi})$.
Here $f$ is a real-valued function, and $\tmi=i2^{-m}$ for $0\leq i\leq 2^m$.
It is worth noting that these results provide 
an extension of the result for the case where $F^m_i=1$ and $Z^m_i=H_q(B_{i-1,i})$
because $\{2^{mH}B_{\tmim,\tmi}\}_{i=1}^{2m}$ and $\{B_{i-1,i}\}_{i=1}^{2^m}$
have the same distribution.
More information about the fourth moment theorem is available in the relevant literature \cite{nourdin-peccati}.

We now describe further applications of the fourth moment theorem.
By using the results by Nourdin-Nualart-Tudor and their variants,
we can identify the limit error distribution of an approximate solution
to RDE driven by one-dimensional fBm 
\cite{aida-naganuma2020, GradinaruNourdin2009, naganuma2015}.
This is attributable to the fact that the main terms of errors of approximate solutions
can be expressed as weighted sum processes.
Similarly to the one-dimensional case,
the main terms of errors of 
approximate solution can be expressed as weighted sum processes
in a multidimensional case.
However, the weighted sum processes are more complicated
and contain iterated integrals of multidimensional fBm as $Z^m_i$
(note that iterated integrals with respect to one-dimensional fBm
are nothing but powers of the increments of the fBm).
See \cite{aida-naganuma2023approach, liu-tindel} for example.
Additionally, weights $F^m_i$'s
are expressed as a functional of solution $Y_t$ of RDE,
the Jacobian $J_t$ and its inverse $J_t^{-1}$.
Because the iterated integrals belong to a Wiener chaos with fixed order,
the fourth moment theorem may be applicable to them.
Weak convergence of these weighted sum processes was also studied
in earlier works \cite{hu-liu-nualart2016, hu-liu-nualart2021, liu-tindel2020}
in the context of approximation theory.

We now provide an overview of our main results and compare them 
with the most closely related work \cite{liu-tindel}.
Our main results are summarized as presented below:
\begin{enumerate}
	\item	moment and discrete H\"older estimates of weighted sum processes
of the Wiener chaos of order $2$ (Theorem~\ref{moment estimate}, and
Corollary~\ref{cor to moment estimate}),
	\item	a limit theorem of weighted sum processes of 
the Wiener chaos of order $2$ (Theorem~\ref{limit theorem of weighted Hermite variation processes}).
\end{enumerate}
These subjects have already been studied by Liu-Tindel in \cite{liu-tindel}.
In addition, the method established in \cite{liu-tindel}
was generalized and refined by the same authors in \cite{liu-tindel2020}.
Here we focus on the original method introduced in \cite{liu-tindel}
and compare their results with ours.
Roughly speaking, Theorem~\ref{moment estimate} and Corollary~\ref{cor to moment estimate}
correspond to Corollary 4.9 in \cite{liu-tindel},
while Theorem~\ref{limit theorem of weighted Hermite variation processes} corresponds to 
Proposition~9.5 in \cite{liu-tindel}.
Theorem~\ref{limit theorem of weighted Hermite variation processes}
is proved similarly to Proposition~9.5 in \cite{liu-tindel}
after showing an FCLT of the ``weight-free'' sum process
(Theorem~\ref{Levy area variation}) with the help of the fourth moment theorem.
However assumption and proof of Theorem~\ref{moment estimate} and Corollary~\ref{cor to moment estimate}
are very different from the one of Corollary 4.9 in \cite{liu-tindel}.

We continue to compare Corollary~4.9 in \cite{liu-tindel} with Theorem~\ref{moment estimate} and Corollary~\ref{cor to moment estimate}.
The assumption and proof of Corollary~4.9 in \cite{liu-tindel}
are derived from rough path analysis.
More precisely, the authors assume that the weights are 
``paths controlled by fBm'', whereas the integrators $Z^m_i$ can take a relatively general form.
Under these conditions, they prove the result using the discrete sewing lemma.
In this framework, the weights are required to be H\"older continuous
with respect to the time parameter.
In contrast to them, our assumption and proof are derived from Malliavin calculus.
More precisely, we assume that the weights belong to a good class $\goodClass(\RR)$
in the sense of Malliavin calculus.
It is noteworthy that 
$\goodClass(\RR)$ requires no H\"older continuity of the weights.
Most typical examples are
$F=(F_t=\varphi(Y_t,J_t,J^{-1}_t))_{t\in[0,1]}$
and $G=(F_{\theta(t)})_{t\in[0,1]}$,
where $\varphi$ is a smooth function that satisfies some growth condition,
and where $\theta$ is a Borel measurable mapping from $[0,1]$ to $[0,1]$.
Other examples are given in Section~\ref{weighted hermite variation}.
On the other hand, we can handle only second order iterated integrals as integrators $Z^m_i$.
We use the
integration by parts formula in Malliavin calculus 
to prove moment estimates of weighted sum processes in
Theorems~\ref{moment estimate}.
Although this technique is an extension of that used
in \cite{NourdinNualartTudor2010, naganuma2015, aida-naganuma2020},
it is necessary to estimate complicated (discrete) multidimensional Young integrals.
This difficulty arises from the fact that 
$Z^m_i$ contains iterated integrals of multidimensional processes,
and it is overcome in Lemmas~\ref{expansion formula of product} 
and \ref{estimate of integration by parts}.
If we obtain Lemmas~\ref{expansion formula of product} and 
\ref{estimate of integration by parts} for higher order iterated integrals, then we may
extend Theorem~\ref{moment estimate} to the case of higher order iterated integrals
under the same assumption, which includes $\frac{1}{3}<H\leq \frac{1}{2}$.
However, this expansion is beyond the scope of the present paper because it is necessary to develop a unified
approach to treat higher order iterated integrals.

As stated above, the weights in \cite{liu-tindel} are paths controlled by fBm,
whereas those in our results are elements of $\goodClass(\RR)$.
Of course, certain paths controlled by fBm are elements of $\goodClass(\RR)$,
and conversely, some elements of $\goodClass(\RR)$ are paths controlled by fBm.
Although we cannot say more about their relationship,
$\goodClass(\RR)$ is a natural class in the following sense.
First, $\goodClass(\RR)$ appears in the multidimensional extension of \cite{aida-naganuma2020, naganuma2015}.
Second, $\goodClass(\RR)$ does not require
its elements to be continuous in the time parameter
or adapted to the filtration with respect to fBm.
Recall that, in general theory of It{\^o}'s stochastic integration,
integrands are adapted to the filtration generated by the integrators.
In our case, this corresponds to 
that the weights $\{F^m_i\}_{i=1}^{2^m}$ are adapted to
the filtration generated by the integrators $\{Z^m_i\}_{i=1}^{2^m}$,
namely the filtration generated by the fBm.
However the limit process of the integrators
is a standard Brownian motion independent of fBms.
This independence implies that the stochastic integral appearing as the limit
is well-defined without the adaptedness of the integrands to the fBm.
Furthermore, since the limit is a stochastic integral with respect to
the standard Brownian motion, the integrands need not be continuous.
A typical example of such a process in $\goodClass(\RR)$
is $G=(F_{\theta(t)})_{t\in[0,1]}$.
It is worth mentioning that the assumption imposed on $\goodClass(\RR)$
can be relaxed for our purpose.
For this, see Remark~\ref{rem47312897319} and Remark~\ref{rem3892803123} (1).

Finally, we mention the potential extension to the case  $\frac{1}{4}<H\leq \frac{1}{3}$.
In the context of rough path analysis,
it is natural to treat the case where $\frac{1}{4}<H\leq \frac{1}{2}$,
while it is still unclear whether such an extension is possible.
One of the key ingredients is Proposition~\ref{prop490801},
which implies $F=(F_t=\varphi(Y_t,J_t,J^{-1}_t))_{t\in[0,1]}$
belongs to $\goodClass(\RR)$ 
and follows from results in Section~\ref{malliavin}.
Therefore Proposition~\ref{prop490801} might be extendable 
if the results in Section~\ref{malliavin} can be extended for $\frac{1}{4}<H\leq \frac{1}{2}$;
however these results are shown in the case $\frac{1}{3}<H\leq \frac{1}{2}$.
In other words, the third-level rough paths are not taken into account,
even though they are required for the case where $\frac{1}{4}<H\leq \frac{1}{3}$.
In particular, Lemma~\ref{derivative of F1} is nontrivial.
This situation also appears in Remark~\ref{Remark on summands}~(2).
Hence we consider the case where $\frac{1}{3}<H\leq \frac{1}{2}$
for simplicity and leave the extension to $\frac{1}{4}<H\leq \frac{1}{2}$ as a topic for future work.
See also Remark~\ref{rem3892803123}~(2).

This paper is organized as follows.
We state our main results in Section~\ref{main results}.
In Section~\ref{Wiener chaos}, 
we summarize basic results related to fBm, Cameron-Martin space 
and Wiener chaos.
Additionally, we recall a condition on Gaussian processes 
under which rough path analysis works well.
In Section~\ref{df app of iterated integrals},
	we present finite dimensional approximations of certain multiple Wiener integrals
	(Corollary~\ref{finite dimensional approximation}).
	This is crucially important for estimates of the Malliavin derivatives of
	the functionals of $Y_t,J_t, J_t^{-1}$.
	To this end, we explore the relation between the multiple Wiener integrals
	and elements in the symmetric tensor products of the Cameron-Martin space.
	In these calculations, we use results in multidimensional Young integrals.
	We refer the readers for several definitions and results of multidimensional
	Young integrals to Section~\ref{appendix I}.
In Section~\ref{malliavin}, we demonstrate that the higher order
Malliavin derivative $(D^rF_t)[h_1,\ldots,h_r]$ can be represented
as iterated integrals of $h_1,\ldots,h_r$,
where $F_t$ is a functional of $Y_t, J_t, J_t^{-1}$
and where $\{h_i\}_{i=1}^r$ are elements of the Cameron-Martin space.
In these calculations, we provide a self-contained proof of the 
higher order Malliavin differentiability
of the functionals of $Y_t, J_t, J_t^{-1}$ (Theorem~\ref{Malliavin differentiability}).
We prove Theorems~\ref{moment estimate} 
and~\ref{limit theorem of weighted Hermite variation processes}
in Sections~\ref{weighted hermite variation} and~\ref{weak convergence},
respectively.
In the proof of Theorem~\ref{limit theorem of weighted Hermite variation processes},
we use FCLT for L\'evy area variation processes (Theorem~\ref{Levy area variation}).
Although the proof is quite standard,
we present the proof herein for the sake of completeness.
In Section~\ref{Holder esitmates of sum processes of Wiener chaos of order 3},
we present discrete H{\"o}lder estimates of weighted sum processes of
the Wiener chaos of orders $3$ (Proposition~\ref{estimate of third order}).
This proposition is shown similarly to Theorem~\ref{moment estimate}
with the help of results in Section~\ref{appendix I}.

\section{Main results}\label{main results}
In this section, we state our main results and present some remarks on them.
Before the main results, we introduce notation.
Let $m$ be a natural number. 
Set $\tmi=i2^{-m}$ for $0\leq i\leq 2^m$ and $D_m=\{\tmi\}_{i=0}^{2^m}$.
For every partition $\cP=\{u_i\}_{i=0}^n$ of $[s,t]$, that is,
$s=u_0<\cdots<u_n=t$, we define $|\cP|=\max_{1\leq i\leq n}|u_i-u_{i-1}|$.
For a constant $0<\lambda<1$, we define the H{\"o}lder norm
for discrete process $F=(F_t)_{t\in D_m}$ by
\begin{align*}
	\|F\|_{\lambda}
	=
		\sup_{s,t\in D_m,s\neq t}
			\frac{|F_t-F_s|}{|t-s|^\lambda}.
\end{align*}
The standard basis of $\RR^d$ is denoted by $\{e_\alpha\}_{\alpha=1}^d$.
For a real-valued function $\phi(u_1,\dots,u_r)$ and $s_i<t_i$ ($1\leq i\leq r$),
we denote by $\phi([s_1,t_1]\times\cdots\times[s_r,t_r])$ the 
rectangular increment,
for example $\phi([s_1,t_1]\times[s_2,t_2])=
\phi(t_1,t_2)-\phi(s_1,t_2)-\phi(t_1,s_2)+\phi(s_1,s_2)$.
We refer to Section~\ref{appendix I}
for precise definition.

Let $B=(B^1,\dots,B^d)$ be a $d$-dimensional fBm 
with the Hurst parameter $\frac{1}{3}<H\leq\frac{1}{2}$.
Section~\ref{Wiener chaos} presents a summary of the property of fBm.
Because of the property, we can construct the rough path lift $(B,\BB)$ of $B$.
The basic references of rough path analysis are, {\it e.g.},
\cite{friz-hairer,friz-victoir,gubinelli,lyons98,lq}.
We write 
\begin{align*}
	B_t
	&=
		\sum_{\alpha=1}^d B^\alpha_t e_\alpha,
	&
	B^\alpha_{s,t}
	&=
		B^\alpha_t-B^\alpha_s,
	&
	B_{s,t}
	&=
		B_t-B_s
	=
		\sum_{\alpha=1}^d B^\alpha_{s,t}e_\alpha,
	&
	\BB_{s,t}
	&=
		\sum_{\alpha,\beta=1}^d B^{\alpha,\beta}_{s,t}e_\alpha\otimes e_\beta.
\end{align*}
Here we recall that $B^{\alpha,\beta}_{s,t}$ is given as follows:
$B^{\alpha,\alpha}_{s,t}=\frac{1}{2}(B^{\alpha}_{s,t})^2$ and, for $\alpha\neq \beta$,
\begin{align*}
	B^{\alpha,\beta}_{s,t}
	=
		\lim_{|\cP|\to 0}
		\sum_{i=1}^{n}
		B^{\alpha}_{s,u_{i-1}}B^{\beta}_{u_{i-1},u_i}
	\quad 
	\text{in $L^2$},
\end{align*}
where $\cP=\{u_i\}_{i=0}^n$ is a partition of $[s,t]$.

Consider the solution $Y$ to the following RDE driven by $B$ on $\RR^n$:
\begin{align}
	Y_t=\xi+\int_0^t\sigma(Y_u)\,dB_u+\int_0^t b(Y_u)\,du,\label{rde}
\end{align}
where $\xi\in\RR^n$ is a deterministic initial point, $\sigma\in C^\infty_b(\RR^n,\mathcal{L}(\RR^d,\RR^n))$
and $b\in C^\infty_b(\RR^n,\RR^n)$.
We denote by $J_t$ the derivative process $J_t=\partial_\xi Y_t(\xi)$.
It is well-known that $J_t$ is invertible; also, $J_t$ and $J_t^{-1}$ are solutions to
\begin{align}
	J_t
	&=
		I
		+
		\int_0^t D\sigma(Y_u)[J_u]\,dB_u
		+
		\int_0^t Db(Y_u)[J_u]\,du,\label{jacobian}\\
	J_t^{-1}
	&=
		I
		-
		\int_0^t J_u^{-1}D\sigma(Y_u)\,dB_u
		-
		\int_0^t J_u^{-1} Db(Y_u)\,du.\label{inverse of jabobian}
\end{align}

We now introduce function spaces of 
smooth functionals in the sense of Malliavin calculus.
The Malliavin derivative is the derivative in the direction to
the Cameron-Martin space.
Here let $\mathcal{H}^d$ denote the Cameron-Martin space 
associated with the fBm $B$.
For a non-negative integer $k$ and $p\ge 1$, let 
$\mathbb{D}^{k,p}(\RR^N)$ denote the set of all $\RR^N$-valued measurable
functions on $\Omega$ which are $k$-times differentiable in the sense of
Malliavin and all its derivatives and itself belong to $L^p$.
Also let $\mathbb{D}^{\infty}(\RR^N)
=\cap_{k\in \ZZ^{+},p\ge 1}\mathbb{D}^{k,p}(\RR^N)$.
We refer the readers for the basic results for these notions to
\cite{nourdin-peccati, nualart, shigekawa}.
Further we write $L^{\infty-}=\cap_{p\ge 1}L^p$ for notational simplicity.
To state our main results, we introduce the following good class of stochastic 
processes.

\begin{definition}\label{def489140919490}
Let $\psi_{s,t}\in \mathcal{H}^1$ be the corresponding element to
the increment of one-dimensional fBm $B^1_{s,t}$; also,
set $\psi^\alpha_{s,t}=\psi_{s,t}\otimes e_{\alpha}\in \mathcal{H}^d$.
Let $\goodClass(\RR^N)$ denote
	the set of all families of random variables 
$(F_t)_{t\in [0,1]}$
	satisfying the following condition.
\begin{enumerate}
	\item	$F_t\in \mathbb{D}^{\infty}(\RR^N)$ for all $0\le t\le 1$.
 \item There exists a random variable $C\in L^{\infty-}$ such that
\begin{align}
\label{Lq estimate for F} |F_t|\le C \quad \text{for all $0\le t\le 1$}.
\end{align}
\item 	For $r\geq 1$ and $\alpha_1,\ldots,\alpha_r\in \{1,\ldots,d\}$,
	set 
	\begin{align}
		\label{eq4i319014}
		\phi^{\alpha_1,\ldots,\alpha_r}_t(u_1,\ldots,u_r)
		=
			(
				D^rF_t,
				\psi^{\alpha_1}_{u_1}\odot\cdots\odot\psi^{\alpha_r}_{u_r}
			)_{(\mathcal{H}^d)^{\odot r}}.
	\end{align}
	Then there exists a continuous modification of
	$\phi^{\alpha_1,\ldots,\alpha_r}_t$ with respect to 
	$(u_1,\ldots,u_r)\in [0,1]^{r}$
	which satisfies
	\begin{align}
\label{estimate of phi alpha_r}
		|
			\phi^{\alpha_1,\ldots,\alpha_r}_t
				([s_1,t_1]\times\cdots\times [s_r,t_r])
		|
		\le
			C'
			\prod_{i=1}^r
			(t_i-s_i)^{2H}.
	\end{align}
	Here $C'$ is a random variable in $L^{\infty-}$,
	which may depend on $r$ and $\alpha_1,\ldots,\alpha_r\in \{1,\ldots,d\}$.
\end{enumerate}
\end{definition}
As stated in the Introduction, the assumption imposed on $\goodClass(\RR^N)$ is
stronger than what is required for our main theorems.
For this, see Remark~\ref{rem47312897319} and Remark~\ref{rem3892803123} (1).
We note that $\goodClass(\RR^N)$ contains 
$\left(\varphi(Y_t,J_t,J_t^{-1})\right)_{t\in [0,1]}$,
where $\varphi$ is an $\RR^N$-valued $C^{\infty}$ mapping 
such that $\varphi$ itself and all the derivatives are at most polynomial growth.
Furthermore, we show the more general result
$\spanIteratedIntegrals(\RR^N)\subset\goodClass(\RR^N)$ in Proposition~\ref{prop490801}.
Here $\spanIteratedIntegrals(\RR)$ is the set of all linear spans of 
iterated integrals. The definition is given as Definition~\ref{def tilde I(R)}.
We do not assume continuity of the mapping
$t\mapsto F_t$ in any sense in the definition above.
Such examples are given in Remark~\ref{Remark on summands}.

Now we state our main results.

\begin{theorem}\label{moment estimate}
Let $(F_t)\in \goodClass(\RR)$ and $1\leq \alpha,\beta\leq d$ be distinct.
	Let $0\le s\le t\le 1$.
	Let
	\begin{align*}
		I^m_{s,t}(F)
		&=
			\sum_{i=\floor{2^ms}+1}^{\floor{2^mt}}
				F_{\tmim}
				B^{\alpha,\beta}_{\tmim,\tmi},
		&
		\tilde{I}^m_{s,t}(F)
		&=
			\sum_{i=\floor{2^ms}+1}^{\floor{2^mt}}
				F_{\tmkm}
				B^{\alpha}_{\tmim,\tmi}
				B^{\beta}_{\tmim,\tmi}.
	\end{align*}
	For any positive integer $p$, 
	there exists a positive constant $C_p$ that is independent
	of $m$ such that
	\begin{align*}
		\Big|
			E
				\Big[
					\Big(
						(2^m)^{2H-\frac{1}{2}}
						I^m_{s,t}(F)
					\Big)^{p}
				\Big]
		\Big|
		+
		\Big|
			E
				\Big[
					\Big(
						(2^m)^{2H-\frac{1}{2}}
						\tilde{I}^m_{s,t}(F)
					\Big)^p
				\Big]
		\Big|
	&\le
		C_p
		\left(
			\frac{\floor{2^mt}-\floor{2^ms}}{2^m}
		\right)^{\frac{p}{2}}.
	\end{align*}
\end{theorem}

Below, we write $I^m_t(F)=I^m_{0,t}(F)$ and so on.

\begin{corollary}\label{cor to moment estimate}
	Let $I^m_t(F)$ and $\tilde{I}^m_t(F)$ $(t\in D_m)$ be 
the processes defined in 
	Theorem~$\ref{moment estimate}$.
	For all $0<\theta<\frac{1}{2}$ and positive integers $p$, we have
	\begin{align*}
		\sup_m
			\left\{
				\|\|(2^m)^{2H-\frac{1}{2}}I^m(F)\|_\theta\|_{L^p}
				+
				\|\|(2^m)^{2H-\frac{1}{2}}\tilde{I}^m(F)\|_\theta\|_{L^p}
			\right\}
		<
			\infty.
	\end{align*}
\end{corollary}

To prove weak convergence of weighted sum processes, 
it is necessary to assume some continuity property
of $(F_t)$ as follows.
This continuity property holds for the elements in $\mathcal{I}(\RR^N)$.

\begin{theorem}\label{limit theorem of weighted Hermite variation processes}
Let $(F_t)=(F^{\alpha,\beta}_t ; 1\le \alpha,\beta\le d)$, where
$(F^{\alpha,\beta}_t)\in \goodClass(\RR^N)$.
Suppose that $(F^{\alpha,\beta}_t)$ is a measurable function on 
	the product space $[0,1]\times \Omega$ and 
stochastically continuous, that is, 	
$\lim_{t\to u}F^{\alpha,\beta}_t=F^{\alpha,\beta}_u$ in probability
	for all $u\in [0,1]$.
	Let
	\begin{align*}
		I^m_{t}(F)
		&=
			\sum_{i=1}^{\floor{2^mt}}
			\sum_{\alpha,\beta=1}^d
				F^{\alpha,\beta}_{\tmim}
				d^{m,\alpha,\beta}_{\tmim,\tmi},
	\end{align*}
	where
	$
		d^{m,\alpha,\beta}_{\tmim,\tmi}
		=
			\frac{1}{2}
			B^{\alpha}_{\tmim,\tmi}
			B^{\beta}_{\tmim,\tmi}
			-
			B^{\alpha,\beta}_{\tmim,\tmi}
	$.
	Then
	\begin{align*}
		\left\{(2^m)^{2H-\frac{1}{2}}I^m_t(F)\right\}_{0\le t\le 1}
		\Longrightarrow
		\left\{
			C
			\int_0^t
				\sum_{\alpha,\beta=1}^d
					F^{\alpha,\beta}_s
					dW^{\alpha,\beta}_s
						\right\}_{0\le t\le 1}
		\text{weakly in $D([0,1], \RR^{N})$},
	\end{align*}
	where $W_t=(W^{\alpha,\beta}_t)$ is an independent process of $B$ 
such that
$(W^{\alpha,\beta}_t ; 1\le \alpha<\beta<d)$ is 
the $\frac{d(d-1)}{2}$-dimensional standard Brownian motion and such that
$W^{\alpha,\beta}_t=-W^{\beta,\alpha}_t$ hold for all $\alpha,\beta$,
and $C$ is a constant given by
	\begin{align*}
		C
		=
		  \Bigg\{
			  E[(B^{\alpha,\beta}_{0,1})^2]
			  +2\sum_{k=1}^{\infty}E[B^{\alpha,\beta}_{0,1}B^{\alpha,\beta}_{k,k+1}]
			-\frac{1}{4}(E[(B^{\alpha}_{0,1})^2])^2
			-\frac{1}{2}
			\sum_{k=1}^{\infty}E[B^{\alpha}_{0,1}B^{\alpha}_{k,k+1}]^2
		\Bigg\}^{1/2},
	\end{align*}
	where $\alpha\ne \beta$.
\end{theorem}

Note that the limit process 
$
	\int_0^t
		\sum_{\alpha,\beta=1}^d
			F^{\alpha,\beta}_s
			dW^{\alpha,\beta}_s
$
in Theorem~\ref{limit theorem of weighted Hermite variation processes}
is well-defined because $F^{\alpha,\beta}_s$ and $W^{\alpha,\beta}_s$
are independent.
We add some remarks about the theorems presented above.

\begin{remark}\label{rem47312897319}
	\begin{enumerate}
		\item	By checking the proof of Theorem~\ref{moment estimate} and Corollary~\ref{cor to moment estimate},
				it is easy to see that the conclusions hold under the following weaker
				assumptions on $(F_t)$:
				\begin{enumerate}
					\item	$F_t$ is $k$-times stochastic G\^ateaux differentiable
							in the directions $\psi^{\alpha}_{u,u'}$ for sufficiently large $k$.
					\item	For $F_t$ and for all $\alpha_1.\ldots,\alpha_r$ $(1\le r\le k)$, 
							the estimates (\ref{Lq estimate for F}) and
							(\ref{estimate of phi alpha_r}) hold with $C, C'\in L^q$ for sufficiently large $q$.
				\end{enumerate}
				Here $k$ and $q$ should be chosen according to 
				$p$, $\theta$ in the statements.
				We apply Theorem~\ref{moment estimate} to 
				Theorem~\ref{limit theorem of weighted Hermite variation processes}
				in the case where $p=4$.
				Therefore, the assumption on $F$ in Theorem~\ref{limit theorem of weighted Hermite variation processes}
				can also be relaxed.
				Consequently, the assumption $\sigma, b\in C^{\infty}_b$ can also be relaxed.

				Note that under the assumption in 
				Theorem~\ref{limit theorem of weighted Hermite variation processes},
				$\lim_{t\to u}\|F^{\alpha,\beta}_t-F^{\alpha,\beta}_u\|_{L^p}=0$ hold for all $u$
				and $p\ge 1$.
				In the proof of Theorem~\ref{limit theorem of weighted Hermite variation processes},
				we use $\lim_{t\to u}\|F^{\alpha,\beta}_t-F^{\alpha,\beta}_u\|_{L^2}=0$ for any $u$.
				Therefore, if we relax the assumption of the integrability of $C$ in (\ref{Lq estimate for F}), then
				it may be necessary to assume this $L^2$ continuity.
		\item	In the calculation of the moment of 
				$I^m_{s,t}(F)$ and $\tilde{I}^m_{s,t}(F)$ in Theorems~\ref{moment estimate},
				finite products of elements in the Wiener chaos appear.
				To estimate the finite products,
				one must obtain each term of the Wiener chaos expansion
				of the products.
				Moreover, we estimate the moments using the 
				integration by parts formula 
				in the Malliavin calculus.
				This is an extension of the method used
				in \cite{NourdinNualartTudor2010, naganuma2015, aida-naganuma2020}
				to multidimensional case.
				In this step, when we apply these theorems to $(F_t)\in \mathcal{I}(\RR^N)$, 
				we need estimates of higher order Malliavin derivatives
				of the functionals of $Y_t, J_t, J_t^{-1}$.
				To obtain the estimates of the Malliavin derivatives, one must
				assume that $\sigma, b$ are sufficiently smooth.
				This assumption is very strong, which
				seems to be shortcoming of our approach compared to
				earlier work \cite{liu-tindel, liu-tindel2020}.
				However our proof does not use regularity of the time variable of $(F_{t})$.
				This point is an advantage of our approach.
				See also Remarks~\ref{Remark on summands} and~\ref{rem9410u09u13}.
	\end{enumerate}
\end{remark}

\section{Preliminaries}
\label{Wiener chaos}

First, we summarize basic notation.
For an $\RR^N$-valued continuous function $f=(f_t)_{t\in I}$ 
defined on an interval
 $I\subset [0,\infty)$, $\|f\|_{p\hyp var, [s,t]}$ denotes the 
$p$-variation norm of $f$ on $[s,t]\subset I$.
Next, let us consider a two-variable continuous function
$f:[0,1]^2\to\RR$.
Write $f([s,s']\times [t,t'])=f(s',t')-f(s,t')-f(s',t)+f(s,t)$.
We may denote $f([s,s']\times [0,t])=f([s,s'],t)$ loosely.
For $f$ and $p\ge 1$, 
the notation $V_p(f ; [s,t]\times[s',t'])$ denotes
the $p$-variation norm of $f$ on $[s,t]\times[s',t']\subset [0,1]^2$.
Several definitions and results
of multidimensional Young integrals are presented in Section~\ref{appendix I}.

Next we summarize the basic facts related to fBm.
Let $R(s,t)$ be the covariance function of the one-dimensional fBm $B$ starting at $0$
with Hurst parameter $H\in(0,1)$,
namely $R(s,t)=E[B_sB_t]=\frac{1}{2}\{s^{2H}+t^{2H}-|s-t|^{2H}\}$.
Let
\begin{align}
	\rho_H(v)=\frac{1}{2}\left(|v+1|^{2H}+|v-1|^{2H}-2|v|^{2H}\right),
	\qquad
	v\in \RR.
	\label{rho}
\end{align}
Note $\rho_H(v)=E[B_1B_{v,v+1}]=R([0,v]\times [v,v+1])$ if $v\ge 0$.
For $0<H\leq \frac{1}{2}$, we have $\sum_{k=0}^{\infty}|\rho_H(k)|<\infty$.
This follows from $\rho_{\frac{1}{2}}(v)=0$ for $v\geq 1$,
and $\rho_H(v)\sim -H(1-2H)v^{-2(1-H)}$ as $v\to\infty$ for $0<H<\frac{1}{2}$.

\begin{lemma}\label{properties of R}
Let $R(s,t)$ be the covariance function of the one-dimensional fBm $(B_t)_{t\ge 0}$
with Hurst parameter $0<H\le \frac{1}{2}$.
\begin{enumerate}
	\item	Let $0\le s<s'$, $0\le t<t'$, $u\ge 0$ and $a\ge 0$.
			We have
			$R([s,s']\times [t,t'])=R([s+u,s'+u]\times [t+u,t'+u])$
			and
			$R([as,as']\times [at,at'])=a^{2H}R([s,s']\times [t,t'])$. 
    \item	The function $[0,\infty)\ni u\mapsto R([s,t]\times [0,u])$ is
			decreasing on $[0,s]\cup [t,\infty)$ and 
			increasing on $[s,t]$.
			Furthermore, the following estimates hold.
			\begin{gather*}
				R([s,t]\times [u,v])>0\quad \text{for all $u<s<t<v$,}\\
				\begin{aligned}
		\|R([s,t],\cdot)\|_{1\hyp var, [0,\infty)} &\le 3|t-s|^{2H}, & 	
        	\|R([s,t],\cdot)\|_{\infty, [0,\infty)}    &\le 3 |t-s|^{2H}.
				\end{aligned}
			\end{gather*}
	\item	For any $k,l\ge 1$, we have
			\begin{align*}
			V_{(2H)^{-1}}(R ; [k-1,k]\times [l-1,l])
			&\le
			C|\rho_H(k-l)|,
			\end{align*}
			where $C$ is a constant depending only on $H$.
	\item	It holds that
			\begin{align*}
			V_{(2H)^{-1}}(R ; [\tmkm,\tmk]\times [\tmlm,\tml])&\le C
			\frac{|\rho_H(k-l)|}{2^{2Hm}},\\
			V_{(2H)^{-1}}(R ; [s,t]\times [0,1])&\le
			3|t-s|^{2H}.
			\end{align*}
\end{enumerate}
\end{lemma}

\begin{proof}
	(1) follows from the stationarity and the scaling property of the 
   fBm
   $B_t$ starting at $0$:
   \begin{align*}
	\{B_{t+u}-B_u\}_{t\ge 0}\overset{\text{d}}{=}\{B_t\}_{t\ge 0},\qquad
   \{B_{at}\}_{t\ge 0}\overset{\text{d}}{=}\{a^{H}B_t\}_{t\ge 0}.
   \end{align*}
   (2) follows from an elementary calculation.

   We prove (3).
   Let $k-1=s_0<\cdots<s_N=k$ and $l-1=t_0<\ldots<t_M=l$ be partitions
   of $[k-1,k]$ and $[l-1,l]$, respectively.
   First, we consider the case $k<l$.
   In this case, we have $s_{i-1}<s_i<t_{j-1}<t_j$ for all $1\le i\le N$ and $1\le j\le M$.
   By the property of (2) and using an elementary inequality,
   $\sum_{i=1}^n|a_i|^p\le (\sum_{i=1}^n|a_i|)^p$ with $p=\frac{1}{2H}$,
   we obtain
	\begin{multline*}
		\sum_{i=1}^N
		\sum_{j=1}^M 
			|R([s_{i-1},s_i]\times[t_{j-1},t_j])|^{\frac{1}{2H}}
		\le
			\left|
				\sum_{i=1}^N
				\sum_{j=1}^M
					|R([s_{i-1},s_i]\times [t_{j-1}\times t_j])|
			\right|^{\frac{1}{2H}}\\
		=
   			|R([k-1,k]\times [l-1,l])|^{\frac{1}{2H}}
		=
			|R([0,1]\times [l-k-1,l-k])|^{\frac{1}{2H}}
		=
			|\rho_H(l-k)|^{\frac{1}{2H}}.
   \end{multline*}
   Next we consider the case $k=l$.
   Using the elementary inequality which we have used and
   the estimate of the total variation 
   of the function $t\mapsto R([s_{i-1},s_i],t)$ in (2),
   we have
   	\begin{multline*}
		\sum_{i=1}^N
		\sum_{j=1}^M
   			|R([s_{i-1},s_i]\times[t_{j-1},t_j])|^{\frac{1}{2H}}
		\le
		   \sum_{i=1}^N
			   	\left(\sum_{j=1}^M
				   |R([s_{i-1},s_i]\times[t_{j-1},t_j])|
				\right)^{\frac{1}{2H}}\\
		\le
			\sum_{i=1}^N
				\|R([s_{i-1},s_i],\cdot)\|_{1\hyp var, [0,\infty)}^{\frac{1}{2H}}
   		\le
			\sum_{i=1}^N
				3^{\frac{1}{2H}}
				|s_{i}-s_{i-1}|
		\le
			3^{\frac{1}{2H}}
		=
			3^{\frac{1}{2H}}
			\rho_H(0).
   \end{multline*}
   This completes the proof of (3).

   Actually, the first estimate in (4) follows from (1) and (3).
   The second one can be deduced by a similar argument for the case $k=l$ in (3).
\end{proof}

Next we introduce a class of Gaussian processes under which we work in 
Sections~\ref{df app of iterated integrals} and \ref{malliavin}.
The condition is given as the following.

\begin{condition}\label{condition on R}
We consider the following conditions on canonically defined Gaussian process 
$(B_t(\omega))=(B^{\alpha}_t(\omega))_{\alpha=1}^d$ 
$(\omega\in \Omega=C([0,1],\RR^d))$ starting at $0$.
\begin{enumerate}
\item $E[B^{\alpha}_t]=0$ for all $\alpha$ and $0\le t\le 1$.
 \item $B^1_t,\ldots,B^d_t$ are independent and identically distributed.
\item $R(s,t)=E[B^\alpha_s B^\alpha_t]$ satisfies 
that there exists $\frac{1}{3}<H\le \frac{1}{2}$ and $C>0$
such that $V_{(2H)^{-1}}(R ; [s,t]^2)\leq C|t-s|^{2H}$ for $0\leq s<t\leq 1$.
\end{enumerate}
\end{condition}

This condition holds for fBm with 
the Hurst parameter $\frac{1}{3}<H\le \frac{1}{2}$.
We have the following result for
Gaussian processes satisfying the condition presented above.
We refer the readers to \cite{friz-hairer} for these results
which are very useful for the study of the Malliavin derivatives of
elements in $\mathcal{I}(\RR)$.

\begin{theorem}\label{friz-hairer a theorem}
\begin{enumerate}
 \item Any Cameron-Martin path $h\in \mathcal{H}^d$ is finite $(2H)^{-1}$-variation.
\item Let $\frac{1}{3}<H^-<H$. There exists a full measure subset 
$\Omega'\subset \Omega$ which satisfies
$\Omega'+\mathcal{H}^d\subset \Omega'$ and
for any $\omega\in \Omega'$, 
$B(\omega)$ can be lifted to an $H^-$-H\"older geometric rough path
$\mathbf{B}(\omega)=(B(\omega),\mathbb{B}(\omega))$.
Let $C(B)=\|B\|_{H^-}+\|\BB\|_{2H^-}$, where $\|~\|_{H^-}, \|~\|_{2H^-}$ denote the H\"older norms.
Then $C(B)\in L^{\infty-}(\Omega)$.
\item $\mathbf{B}(\omega+h)=T_h\mathbf{B}(\omega)$ $(\omega\in \Omega')$ holds, where
$T_h\mathbf{B}(\omega)$ is the translated rough path of $\mathbf{B}(\omega)$.
\end{enumerate}
\end{theorem}

Here it is helpful to recall the definitions of the Cameron-Martin space and 
Wiener chaos.
Let $\mathcal{H}^d$ be the Cameron-Martin subspace of $\Omega$. 
Let $L^2(\Omega,\mu)=\oplus_{n=0}^{\infty}
\mathscr{H}_n$ be the Wiener chaos decomposition,
where $\mathscr{H}_n$ is the $n$-th Wiener chaos.
Then there exists an isomorphism 
between two Hilbert spaces $\mathcal{H}^d$ and $\mathscr{H}_1$ by 
\begin{align*}
	\mathscr{H}_1 \ni X\mapsto h_X\in \mathcal{H}^d,
	\qquad
	\text{where}
	\qquad
	h_X(t)=\left(E[XB^1_t],\ldots,E[XB^d_t]\right).
\end{align*}
When $d=1$,
we denote the corresponding element to
the random variable $B_u(\omega)\in \mathscr{H}_1$ by $\psi_u\in \mathcal{H}^1$.
By definition, $\psi_u(t)$ is equal to the covariance function
$R(u,t)=E[B^{\alpha}_uB^{\alpha}_t]$ as a continuous function.

Let $\{h_i\}_{i=1}^{\infty}$ be a complete orthonormal system of $\mathcal{H}^1$.
Let $\{e_\alpha\}_{\alpha=1}^d$ be the standard orthonormal base of $\RR^d$.
Then $\mathcal{H}^d\equiv \mathcal{H}^1\otimes \RR^d$ and $\{h^{\alpha}_i\}_{i,\alpha}$ is
a complete orthonormal system of $\mathcal{H}^d$, where $h^{\alpha}_i=h_i\otimes e_{\alpha}$.
One of the orthonormal basis of $(\mathcal{H}^d)^{\otimes r}$ is
$
	\{
		h^{\alpha_1}_{i_r}\otimes\cdots\otimes
		h^{\alpha_r}_{i_r}~|~i_1,\ldots, i_r\ge 1,
		1\le \alpha_1,\ldots,\alpha_r\le d
	\}
$.
Two Hilbert spaces $(\mathcal{H}^d)^{\otimes r}$ and
$\left((\mathcal{H}^1)^{\otimes r}\right)\otimes (\RR^d)^{\otimes r}$
are isomorphic to each other by the map
$h^{\alpha_1}_{i_r}\otimes\cdots\otimes
	       	h^{\alpha_r}_{i_r}\mapsto 
(h_{i_1}\otimes\cdots\otimes h_{i_r})\otimes 
(e_{\alpha_1}\otimes\cdots\otimes e_{\alpha_r})$.
Below, $\mathcal{S}$ denotes the symmetrization operator
on the space of tensor products $(\mathcal{H}^d)^{\otimes r}$.
That is, it is defined by
\begin{align*}
	\mathcal{S}(h^{\alpha_1}_{i_1}\otimes\cdots \otimes h^{\alpha_r}_{i_r})
=\frac{1}{r!}
\sum_{\sigma\in \mathfrak{G}_r}
h^{\alpha_{\sigma(1)}}_{i_{\sigma(1)}}\otimes\cdots
\otimes h^{\alpha_{\sigma(r)}}_{i_{\sigma(r)}},
\end{align*}
where $\mathfrak{G}_r$ denotes the permutation group of $(1,\ldots,r)$.

There exists a one-to-one correspondence
between the set of $p$-th Wiener chaos $\mathscr{H}_p$
and the symmetric tensor product of the Cameron-Martin space
$(\mathcal{H}^d)^{\odot p}$.
Actually they are isomorphism between two Hilbert spaces.
Let us recall product formula for Wiener chaos.
We denote the $p$-th It{\^o}-Wiener integral by $I_p$,
which is a map from $(\mathcal{H}^d)^{\odot p}$ to $\mathscr{H}_p$.

     \begin{proposition}\label{product formula}
Let $p, q$ be positive integers and
let $f\in (\mathcal{H}^d)^{\odot p}$
 and $g\in (\mathcal{H}^d)^{\odot q}$.

\begin{enumerate}
	\item	Let $0\le r\le \min(p,q)$.
			Let $f\underset{r}{\tilde{\otimes}} g$
			be the symmetrization of the $r$-th contraction of $f$ and $g$.
			Then the mapping
			\begin{align*}
				(\mathcal{H}^d)^{\odot p}
				\times
				(\mathcal{H}^d)^{\odot q}
				\ni
					(f,g)
				\mapsto
					f\underset{r}{\tilde{\otimes}}g
				\in
					(\mathcal{H}^d)^{\odot p+q-2r}
			\end{align*}
			is continuous linear.
	\item	It holds that
			\begin{align*}
				I_p(f)I_q(g)
				=
					\sum_{r=0}^{p\wedge q}
						r!\binom{p}{r}
						\binom{q}{r}I_{p+q-2r}
						(f\underset{r}{\tilde{\otimes}}g).
			\end{align*}
	\item	Let $f=x_1\odot\cdots\odot x_p\in (\mathcal{H}^d)^{\odot p}$
			and $g=y_1\odot\cdots\odot y_q\in (\mathcal{H}^d)^{\odot q}$,
			where $x_i,y_j\in \mathcal{H}^d$.
			Then,
			\begin{align*}
				I_p(f)I_q(g)
				&=
				\sum_{r=0}^{p\wedge q}
				\sum_{\substack{
						I=\{i_1,\ldots,i_r\}\subset \{1,\ldots,p\},\\
						J=\{j_1,\ldots,j_r\}\subset \{1,\ldots,q\}}}
				\sum_{\sigma\in\mathfrak{G}_r}
					\prod_{k=1}^r
					\left(x_{i_k},y_{j_{\sigma(k)}}\right)_{\mathcal{H}^d}
					Z_{p,q,r,I,J}\\
				&=
					\sum_{r=0}^{p\wedge q}
					\sum_{\substack{
							I\subset \{1,\ldots,p\}, J\subset \{1,\ldots,q\}\\
							\text{with $|I|=|J|=r$}}}
						r!
						\left(
							\underset{i\in I}{\odot}x_i,
							\underset{j\in J}{\odot}y_j
						\right)_{(\mathcal{H}^d)^{\odot r}}
						Z_{p,q,r,I,J},
			\end{align*}
			where
			\begin{align*}
				Z_{p,q,r,I,J}
				=
				I_{p+q-2r}
					\left(
						\underset{i\in \{1,\ldots,p\}\setminus I}{\odot} x_i
						\odot
						\underset{j\in \{1,\ldots,q\}\setminus J}{\odot}y_j
					\right).
			\end{align*}
		\end{enumerate}
\end{proposition}
\begin{proof}
	Statements (1) and (2) are standard facts and we omit the proof.
	We prove (3). The first identity in (3) follows from (2) and
	\begin{align}
		\label{eq90411133}
		r!\binom{p}{r}\binom{q}{r}
		f\underset{r}{\tilde{\otimes}}g
	   	=
		   	\sum_{\substack{
				I=\{i_1,\ldots,i_r\}\subset \{1,\ldots,p\},\\
				J=\{j_1,\ldots,j_r\}\subset \{1,\ldots,q\}}}
			   	\sum_{\sigma\in\mathfrak{G}_r}
					\prod_{k=1}^r
					\left(x_{i_k},y_{j_{\sigma(k)}}\right)_{\mathcal{H}^d}
				\underset{i\in \{1,\ldots,p\}\setminus I}{\odot} x_i
				\odot
				\underset{j\in \{1,\ldots,q\}\setminus J}{\odot}y_j.
	\end{align}
	The second first identity in (3) follows from the first one and
	\begin{align}
		\label{eq45819011}
		\left(
			\underset{i\in I}{\odot}x_i,
			\underset{j\in J}{\odot}y_j
		\right)_{(\mathcal{H}^d)^{\odot r}}
		=
			\frac{1}{(r!)^2}
			\sum_{\sigma,\tau\in\mathfrak{S}_r}
			\prod_{k=1}^r
				\left(x_{i_{\sigma(k)}},y_{j_{\tau(k)}}\right)_{\mathcal{H}^d}
		=
			\frac{1}{r!}
			\sum_{\sigma\in\mathfrak{S}_r}
			\prod_{k=1}^r
				\left(x_{i_k},y_{j_{\sigma(k)}}\right)_{\mathcal{H}^d}.
	\end{align}
	This identity follows from the definition.
	In what follows, we show \eqref{eq90411133}.
	Below $I=(i_1,\ldots,i_r)$ and $J=(j_1,\ldots,j_r)$ 
	respectively denote subsets with the order of $\{1,\ldots,p\}$ and
	$\{1,\ldots,q\}$.
	As usual, we use the notation $I=\{i_1,\ldots,i_r\}$
	and $J=\{j_1,\ldots,j_r\}$ to denote subsets.
	Let $\tilde{\sigma}$ and $\tilde{\tau}$ move in the set of
the bijective mappings between
$\{r+1,\ldots,p\}\to \{1,\ldots,p\}\setminus I$
and $\{r+1,\ldots,q\}\to \{1,\ldots,q\}\setminus J$.
Using these notations, we have
\begin{align*}
	f\otimes_r g
	&=
		\frac{1}{p!q!}
		\sum_{\sigma\in\mathfrak{G}_p,\tau\in \mathfrak{G}_q}
			\prod_{i=1}^r
			(x_{\sigma(i)},y_{\tau(i)})_{\mathcal{H}^d}
			x_{\sigma(r+1)}\otimes\cdots\otimes x_{\sigma(p)}
			\otimes
			y_{\tau(i+1)}\otimes\cdots\otimes y_{\tau(q)}\\
	&=
		\frac{1}{p!q!}
		\sum_{\substack{I=(i_1,\ldots,i_r),\\ J=(j_1,\ldots,j_r)}}
		\sum_{
			\substack{\tilde{\sigma}: \{r+1,\ldots,p\}\to \{1,\ldots,p\}\setminus I,\\
			\tilde{\tau} :\{r+1,\ldots,q\}\to \{1,\ldots, q\}
			\setminus J}}
			\prod_{k=1}^r(x_{i_k},y_{j_k})_{\mathcal{H}^d}\\
	&\phantom{=}\qquad\qquad\qquad\qquad\qquad\qquad\qquad
			\times
			x_{\tilde{\sigma}(r+1)}\otimes\cdots\otimes x_{\tilde{\sigma}(p)}
			\otimes
			y_{\tilde{\tau}(r+1)}\otimes\cdots\otimes y_{\tilde{\tau}(q)}.
\end{align*}
Applying the symmetrization operator $\mathcal{S}$, we obtain
\begin{align*}
	f\underset{r}{\tilde{\otimes}}g
	&=
		\frac{(p-r)!(q-r)!}{p!q!}
		\sum_{\substack{I=(i_1,\ldots,i_r),\\ J=(j_1,\ldots,j_r)}}
			\prod_{k=1}^r(x_{i_k},y_{j_k})_{\mathcal{H}^d}
				\underset{i\in \{1,\ldots,p\}\setminus I}{\odot} x_i
				\odot
				\underset{j\in \{1,\ldots,q\}\setminus J}{\odot}y_j\\
	&=
		\frac{(p-r)!(q-r)!}{p!q!}
		\sum_{\substack{I=\{i_1,\ldots,i_r\},\\ J=\{j_1,\ldots,j_r\}}}
			\sum_{\sigma,\tau\in \mathfrak{G}_r}
				\prod_{k=1}^r(x_{i_{\sigma(k)}},y_{j_{\tau(k)}})_{\mathcal{H}^d}
				\underset{i\in \{1,\ldots,p\}\setminus I}{\odot} x_i
				\odot
				\underset{j\in \{1,\ldots,q\}\setminus J}{\odot}y_j.
\end{align*}
Combining this with \eqref{eq45819011}, we arrived at \eqref{eq90411133}.
\end{proof}

\section{Finite dimensional approximations of iterated integrals}
\label{df app of iterated integrals}

Throughout this section, we assume that Condition~\ref{condition on R} holds for
$d$-dimensional Gaussian process $(B_t)$.
Recall that $\frac{1}{3}<H\leq\frac{1}{2}$ is assumed in Condition~\ref{condition on R}.
We denote by $\psi_{s,t}\in \mathcal{H}^1$ the corresponding element to
the increment of one-dimensional fBm $B^1_{s,t}$; also,
set $\psi^\alpha_{s,t}=\psi_{s,t}\otimes e_{\alpha}\in \mathcal{H}^d$.
In this section, we will identify the elements in the tensor product of the Cameron-Martin space
corresponding to multiple Wiener integrals. 
Finally, we give a finite dimensional approximation of multiple Wiener integrals.

First, we will find the element in $(\mathcal{H}^d)^{\otimes 2}$
corresponding to $B^{\alpha,\beta}_{s,t}$.
When $\alpha\ne \beta$, the iterated integral $B^{\alpha,\beta}_{s,t}$
is defined as
\begin{align*}
	B^{\alpha,\beta}_{s,t}
	&=
		\lim_{|\cP|\to 0}
			\sum_{i=1}^{n}
			B^{\alpha}_{s,u_{i-1}}B^{\beta}_{u_{i-1},u_i}
			\quad \textrm{in $L^2$},
\end{align*}
where $\cP=\{u_i\}_{i=0}^n$ is a partition of $[s,t]$.
	       Then, we identify the corresponding element
	       in $(\mathcal{H}^d)^{\odot 2}$ to
	       $B^{\alpha,\beta}_{s,t}$.
	       Because
 	       $(\psi^{\alpha}_u,\psi^{\beta}_v)_{\mathcal{H}^d}=0$,
	       using Proposition~\ref{product formula},
	       we have
	    \begin{align*}
			\sum_{i=1}^{n}
		 		B^{\alpha}_{s,u_{i-1}}B^{\beta}_{u_{i-1},u_i}
		 	=
			\sum_{i=1}^n
				I_1(\psi_{s,u_{i-1}}^{\alpha})
				I_1(\psi_{u_{i-1},u_i}^{\beta})
			=
				I_2
					\left(
					\sum_{i=1}^n
						\psi_{s,u_{i-1}}^{\alpha}
						\odot
						\psi_{u_{i-1},u_i}^{\beta}
					\right).
		\end{align*}
	       			Because we know that the sum of
				random variables on the
				left-hand side converges in $L^2$,
				$
\sum_{i=1}^n
		 \psi_{s,u_{i-1}}^{\alpha}
		 \odot \psi_{u_{i-1},u_i}^{\beta}$
		 converges in $(\mathcal{H}^d)^{\odot 2}$.
We denote the limit by
\begin{align*}
 \tpsi^{\alpha,\beta}_{s,t}:=\int_s^t\psi^{\alpha}_{s,u}\odot d\psi^{\beta}_u.
\end{align*}
   To be explicit, we have
		 \begin{align*}
		  \sum_{i=1}^n
		 \psi_{s,u_{i-1}}^{\alpha}
		  \odot
		  \psi_{u_{i-1},u_i}^{\beta}
		  		  &=
		  \frac{1}{2}
		  \sum_{i=1}^n
		  \left(\psi_{s,u_{i-1}}e_{\alpha}\otimes
		  \psi_{u_{i-1},u_i}e_{\beta}
		  +
\psi_{u_{i-1},u_i}e_{\beta}\otimes
\psi_{s,u_{i-1}}e_{\alpha}\right).
		  		 \end{align*}
Furthermore, because $\sum_{i=1}^n
		 \psi_{s,u_{i-1}}e_{\alpha}
		  \otimes
		  \psi_{u_{i-1},u_i}e_{\beta}$
and
$\sum_{i=1}^n
\psi_{u_{i-1},u_i}e_{\beta}\otimes
\psi_{s,u_{i-1}}e_{\alpha}
$
are orthogonal,
we see that
$\lim_{|\cP|\to 0}
\sum_{i=1}^n\psi_{s,u_{i-1}}\otimes \psi_{u_{i-1},u_i}$
converges in $(\mathcal{H}^1)^{\otimes 2}$.

For the case in which $\alpha=\beta$,
because we consider geometric rough paths,
we have $B^{\alpha,\alpha}_{s,t}=\frac{1}{2}(B^{\alpha}_{s,t})^2$.
Using Proposition~\ref{product formula} and 
$\left(\psi^{\alpha}_{s,t},\psi^{\alpha}_{s,t}\right)_{\mathcal{H}}
=E[(B^{\alpha}_{s,t})^2]=R([s,t]\times [s,t])$,
we obtain
\begin{align*}
	B^{\alpha,\alpha}_{s,t}
	=
		\frac{1}{2}
		I_1(\psi^{\alpha}_{s,t})^2
	=
		\frac{1}{2}
		\left\{
			I_2\left(\psi^{\alpha}_{s,t}\odot\psi^{\alpha}_{s,t}\right)
			+\left(\psi^{\alpha}_{s,t},\psi^{\alpha}_{s,t}\right)_{\mathcal{H}}
		\right\}
	=
		\frac{1}{2}
		I_2\left(\psi^{\alpha}_{s,t}\odot\psi^{\alpha}_{s,t}\right)
		+
		\frac{1}{2}
		R([s,t]\times [s,t]).
\end{align*}

Let $2\le l\le d$.
We next define general $l$-th iterated integral of
$\psi_u$ and $\psi^{\alpha_1}_{u_1},\ldots,\psi^{\alpha_l}_{u_l}$ 
$(1\le \alpha_i\le d)$ for $l\ge 3$
as elements of
$(\mathcal{H}^1)^{\otimes l}$ and
$(\mathcal{H}^d)^{\otimes l}$ respectively.
To this end, similarly to the case $l=2$,
we consider an $l$-dimensional Gaussian process
$B_t=(B^1_t,\ldots,B^l_t)$ $(0\le t\le 1)$ which satisfies Condition~\ref{condition on R}.
Let $(B_{s,t},\BB_{s,t})$ be the corresponding
rough path.
Next we consider a consecutive sequence $\{1,\ldots,l\}$ and
the $l$-th iterated integral $B^{1,\ldots,l}_{s,t}$,
which is defined as
an integral of controlled paths inductively after we obtain iterated
integrals $B^{\alpha,\beta}_{s,t}$ $(1\le \alpha,\beta\le l)$.
That is, suppose we have defined the iterated integral
$B^{1,\ldots,r-1}$ ($3\le r\le l$).
Then $l$-th iterated integrals can be defined as the pointwise limit
\begin{align*}
 B^{1,\ldots,l}_{s,t}&=\int_s^tB^{1,\ldots,l-1}_{s,u}dB_u^l
 =\lim_{|\cP|\to 0}\sum_{i=1}^{n}
 \left\{
 B^{1,\ldots,l-1}_{s,u_{i-1}}B^l_{u_{i-1},u_i}+
 B^{1,\ldots,l-2}_{s,u_{i-1}}B^{l-1,l}_{u_{i-1},u_i}
 \right\},
\end{align*}
where $\cP=\{u_i\}_{i=0}^n$ is a partition of $[s,t]$.
Here we estimate
\begin{align*}
	A(\cP)
	=
		E
		\Bigg[
			\left(
				\sum_{i=1}^{n}
					B^{1,\ldots,l-2}_{s,u_{i-1}}B^{l-1,l}_{u_{i-1},u_i}
			\right)^2
		\Bigg]
	=
		\sum_{i,j=1}^n
			E
				\left[
					B^{1,\ldots,l-2}_{s,u_{i-1}}B^{1,\ldots,l-2}_{s,u_{j-1}}
				\right]
			E
				\left[
					B^{l-1,l}_{u_{i-1},u_i}B^{l-1,l}_{u_{j-1},u_j}
				\right].
\end{align*}
Here let $\vep$ be a positive number such that $2(2H-\vep)>1$.
Note $p=(2H-\vep)^{-1}$ satisfies $1<p<2$.
Using the moment estimate of $B^{1,\ldots,l-2}_{u,v}$
and the results of multidimensional Young integrals
(multidimensional Young integrals are explained further in Section~\ref{appendix I}),
we have
\begin{align*}
	A(\cP)
	&\leq
		C_{s,t,l}
		\sum_{i,j=1}^n
			\left|
				\int_{[u_{i-1},u_i]\times[u_{j-1},u_j]}R([u_{i-1},u]\times [u_{j-1},v])dR(u,v)
			\right|\\
	&\leq
		C_{s,t,l}
		C_{\vep,H}
		\sum_{i,j=1}^n
			V_p(R ; [u_{i-1},u_i]\times [u_{j-1},u_j])^2\\
&\le C_{s,t,l}
		C_{\vep,H}
\max_{i,j}V_p(R ; [u_{i-1},u_i]\times [u_{j-1},u_j])^{2-p}
\sum_{i,j=1}^n V_p(R ; [u_{i-1},u_i]\times [u_{j-1},u_j])^p.
\end{align*}
Because $V_{(2H)^{-1}}(R ; [0,1]^2)<\infty$,
we have $\lim_{|\mathcal{P}|\to 0}\max_{i,j}V_p(R ; [u_{i-1},u_i]\times [u_{j-1},u_j])=0$.
Combining Theorem~\ref{FV} and the superadditivity of the $p$-variation norm,
we have $\lim_{|\mathcal{P}|\to 0}A(\mathcal{P})=0$.
Therefore, we obtain the following.

\begin{lemma}
We consider $d$-dimensional Gaussian process $(B_t)$ satisfying 
Condition~$\ref{condition on R}$.
Let $2\le l\le d$
and $\cP=\{u_i\}_{i=0}^n$ be a partition of $[s,t]$.
Then we have
 \begin{align}
 B^{1,\ldots,l}_{s,t}=\lim_{|\cP|\to 0}\sum_{i=1}^{n}
B^{1,\ldots,l-1}_{s,u_{i-1}}
 B^l_{u_{i-1},u_i}\qquad
 \mbox{in $L^2$}.\label{L2limit}
\end{align}
\end{lemma}
Using this result, we prove the following lemma.

\begin{lemma}\label{psi^{(l)}}
We consider $d$-dimensional Gaussian process $(B_t)$ satisfying 
Condition~$\ref{condition on R}$.
In this lemma, we set $\mathcal{H}=\mathcal{H}^1$.
 Let $l\ge 2$ be a positive integer.
 \begin{enumerate}
  \item[$(1)$]
	    Let $0\le s\le t\le 1$ and
	    $\cP=\{u_i\}_{i=0}^n$ be a partition of $[s,t]$.
Let $\psi^{(1)}_{s,t}=\psi_{s,t}$.
	    The following inductive definition of
	    $\psi^{(l)}_{s,t}\in \mathcal{H}^{\otimes l}$
	    is well-defined and the
	    sequence converges in $\mathcal{H}^{\otimes l}$.
	     \begin{align}
	      \psi^{(l)}_{s,t}=
	      \lim_{|\cP|\to 0}\sum_{i=1}^{n}
	      \psi^{(l-1)}_{s,u_{i-1}}\otimes
	      \psi_{u_{i-1},u_i}.\label{vepl convergence}
	     \end{align}
Moreover, it holds that
\begin{align}
 \|\psi^{(l)}_{s,t}\|_{\mathcal{H}^{\otimes l}}^2&=
R^{l}_s(t,t),\label{covariance of epl}
\end{align}
where
$R^l_s(u,v)$ $(s\le u,v\le t)$ is defined by the well-defined Young integrals
\begin{align*}
	R^1_s(u,v)
	&=
		R([s,u]\times [s,v]),
	&
	R^l_s(u,v)
	&=
		\int_{[s,u]\times[s,v]}	
			R^{l-1}_s(u',v')
			dR(u',v').
\end{align*}
  \item[$(2)$] 
Let $1\le \alpha_1,\ldots,\alpha_l\le d$ be mutually different integers.
Then we have
	\begin{align}
		\label{eq489301433}
	B^{\alpha_1,\ldots,\alpha_l}_{s,t}
	=
		I_l
			\big(
				\mathcal{S}
					\big(
						\psi_{s,t}^{(l)}
						e_{\alpha_{1}}
						\otimes
						\cdots
						\otimes 
						e_{\alpha_l}
					\big)
			\big).
	\end{align}
  \item[$(3)$]	      For a partition
		      $
		       \cP=\{u_i\}_{i=0}^n
		      $ of $[s,t]$,
		   we define inductively by 
$\psi^{(1),\cP}_{s,u_j}=\psi_{s,u_j}$ for $1\leq j\leq n$ and 
$\psi^{(1),\cP}_{s,s}=0$.
For $l\ge 2$, we define $\psi^{(l),\cP}_{s,u_{j}}\in 
\mathcal{H}^{\otimes l}$
$(1\le j\le n)$ inductively by
		      \begin{align*}
		       \psi^{(l),\cP}_{s,u_j}
		       &=
\begin{cases}
		       \sum_{i=1}^{j}
		       \psi^{(l-1),\cP}_{s,u_{i-1}}
		      \otimes\psi_{u_{i-1},u_i},
		       & 1\le j\le n,
\\
0, & j=0.
\end{cases}
			      \end{align*}
		      Then it holds that for all $l\ge 1$,
		      \begin{align}
		       \lim_{|\cP|\to 0}
		       \|\psi^{(l),\cP}_{s,t}-\psi^{(l)}_{s,t}\|
		       _{\mathcal{H}^{\otimes l}}=0.\label{eplP convergence}
		      \end{align}
		      
 \end{enumerate}
\end{lemma}

\begin{proof}
It seems clear that it is sufficient to prove statement (2) in the case in which
 $\alpha_1=1, \ldots, \alpha_l=l$ and $l\le d$.
 We prove (1) and (2) simultaneously by an induction on $l$.
 The statement holds for $l=2$.
 Suppose (1) and (2) holds up to $l-1$.
 Then,
 by the observation (\ref{L2limit}), the following convergence holds in
 $L^2$ sense,
 \begin{align}
  B^{1,\ldots,l}_{s,t}&=\lim_{|\cP|\to 0}\sum_{i=1}^{N}
  I_{l-1}\left(\mathcal{S}
\left(\psi^{l-1}_{s,u_{i-1}}e_1\otimes\cdots\otimes e_{l-1}
  \right)\right)
  I_1(\psi_{u_{i-1},u_i}e_l)
  \nonumber\\
  &=
  I_l\left(\lim_{|\cP|\to 0}
  \sum_{i=1}^{n}
  \mathcal{S}\left(\mathcal{S}\left(
  \psi^{l-1}_{s,u_{i-1}}e_1\otimes\cdots\otimes e_{l-1}\right)
  \otimes \psi_{u_{i-1},u_i}e_l\right)
  \right).\label{L2limit2}
 \end{align}
Note that
\begin{multline*}
	\mathcal{S}
		(
			\mathcal{S}
			(
				(h_1\otimes\cdots\otimes h_{l-1})
				(e_1\otimes\cdots\otimes e_{l-1})
			)
			\otimes
			h_le_l
		)\\
	\begin{aligned}
		&=
			\mathcal{S}
				\bigg(
					\bigg(
						\frac{1}{(l-1)!}
						\sum_{\sigma\in \mathfrak{G}_{l-1}}
							(h_{\sigma(1)}\otimes\cdots\otimes h_{\sigma(l-1)})
							(e_{\sigma(1)}\otimes\cdots \otimes e_{\sigma(l-1)})
					\bigg)
					h_le_l
				\bigg)\\
		&=
			\frac{1}{l!(l-1)!}
			\sum_{\sigma\in \mathfrak{G}_{l-1},\tau\in \mathfrak{G}_l}
				(h_{\tau(\sigma(1))}\otimes\cdots\otimes h_{\tau(\sigma(l-1))}
				\otimes h_{\tau(l)})\\
		&\qquad\qquad\qquad\qquad\qquad\qquad\qquad\qquad
				\times
				(e_{\tau(\sigma(1))}\otimes\cdots\otimes e_{\tau(\sigma(l-1))}\otimes 
				e_{\tau(l)})\\
&=\frac{1}{l!(l-1)!}
\sum_{\substack{\sigma,\tau\in \mathfrak{G}_l \\
\text{with $\sigma(l)=l$}}}
(h_{\tau(\sigma(1))}\otimes\cdots\otimes h_{\tau(\sigma(l-1))}
\otimes h_{\tau(\sigma(l))})\\
&\qquad\qquad\qquad\qquad\qquad\qquad\qquad\qquad
\times
(e_{\tau(\sigma(1))}\otimes\cdots\otimes e_{\tau(\sigma(l-1))}\otimes 
e_{\tau(\sigma(l))})\\
&=\frac{1}{(l-1)!}
\sum_{\substack{\sigma\in \mathfrak{G}_l\\ \text{with $\sigma(l)=l$}}}
\mathcal{S}
\left((h_1\otimes\cdots\otimes h_l) (e_1\otimes\cdots\otimes e_l)\right)\\
&=\mathcal{S}
\left((h_1\otimes\cdots\otimes h_l) (e_1\otimes\cdots\otimes e_l)\right).
	\end{aligned}
\end{multline*}
 Therefore,
 \begin{align*}
  \sum_{i=1}^{n}\mathcal{S}\left(\mathcal{S}\left(
  \psi^{(l-1)}_{s,u_{i-1}}e_1\otimes\cdots\otimes e_{l-1}\right)
  \otimes \psi_{u_{i-1},u_i}e_l\right)
  =
\mathcal{S}
\left(\sum_{i=1}^{n}
  (\psi^{(l-1)}_{s,u_{i-1}}\otimes
  \psi_{u_{i-1},u_i})(e_{1}\otimes \cdots\otimes
  e_{l})\right).
 \end{align*}
 Because the indices $1,\ldots, l$ differ, the convergence 
(\ref{L2limit2}) 
 implies (\ref{vepl convergence}) and (\ref{eq489301433}).
We prove (\ref{covariance of epl}) by an induction on $l$.
The case $l=1$ holds because
\begin{align*}
 \|\psi^{(1)}_{s,t}\|_{\mathcal{H}}^2&=
E[(B_t-B_s)^2]=R([s,t]\times[s,t])=R^1_s(t,t)
\end{align*}
Suppose the case of $l-1$ holds.
Let $\cP=\{u_i\}_{i=0}^n$ be a partition of
$[s,t]$.
Then
\begin{align*}
	\left\|
		\sum_{i=1}^n\psi^{(l-1)}_{s,u_{i-1}}\otimes\psi_{u_{i-1},u_i}
	\right\|_{\mathcal{H}^{\otimes l}}^2
	&=
		\sum_{i=1}^n
		\sum_{j=1}^n
			\big(\psi^{(l-1)}_{s,u_{i-1}},\psi^{(l-1)}_{s,u_{j-1}}\big)_{\mathcal{H}^{\otimes (l-1)}}
			\big(\psi_{u_{i-1},u_i},\psi_{u_{j-1},u_j}\big)_{\mathcal{H}}\\
	&=
		\sum_{i=1}^n
		\sum_{j=1}^n
			R^{l-1}_s(u_{i-1},u_{j-1})
			R([u_{i-1},u_i]\times[u_{j-1},u_{j}])\\
	&\to
		R^l_s(t,t) \quad \text{as $|\cP|\to 0$}.
\end{align*}

 We prove (3).
Let us consider another partition
$\cP'=\{v_i\}_{i=0}^{n'}$ of $[s,t]$.
Inductively, we define
\begin{align*}
	R^{\cP\times \cP',1}_s(u_i,v_{j})
	&=
		R([s,u_i]\times [s,v_j]),\\
	R^{\cP\times \cP',l}_s(u_i,v_{j})
  	&=
		\sum_{k=l-1}^i
		\sum_{k'=l-1}^j
			R^{\cP\times \cP',l-1}_s(u_{k-1},v_{k'-1})
			R([u_{k-1},u_k]\times [v_{k'-1},v_{k'}]).
\end{align*}
We use the convention that if the set 
$\{(k,k')~|~l-1\le k\le i, l-1\le k'\le j\}$ is empty,
we set $R^{\cP\times \cP',l}_s(u_i,v_{j})=0$.
That is, $R_s^{\cP\times \cP',l}(u_i,v_j)=0$
if $i\le l-2$ or $j\le l-2$.
Also by the definition, we note that
$R^{\cP\times \cP',l}(u_{l-1},\cdot)=
R^{\cP\times \cP',l}(\cdot,v_{l-1})=0$.
 Note that
\begin{align*}
	\left(
		\psi^{(2),\cP}_{s,u_i},
		\psi^{(2),\cP'}_{s,v_j}
	\right)_{\mathcal{H}^{\otimes 2}}
	&=
		\sum_{k=1}^i
		\sum_{k'=1}^j
			\big(\psi_{s,u_{k-1}},\psi_{s,v_{k'-1}}\big)_{\mathcal{H}}
			\big(\psi_{u_{k-1},u_k},\psi_{v_{k'-1},v_{k'}}\big)_{\mathcal{H}}\\
	&=
		\sum_{k=1}^i
		\sum_{k'=1}^j
			R([s,u_{k-1}]\times [s,v_{k'-1}])
			R([u_{k-1},u_k]\times [v_{k'-1},v_{k'}]) \\
	&=
		R^{\cP\times\cP',2}_s(u_i,v_j).
\end{align*}
The identity holds for $i=0$ or $j=0$ also.
It is therefore easy to obtain the following identity by induction.
For all $0\le i\le n,\, 0\le j\le n'$,
\begin{align}
	\left(
		\psi^{(l),\cP}_{s,u_i},
		\psi^{(l),\cP'}_{s,v_j}
	\right)_{\mathcal{H}^{\otimes l}}
	&=
		\sum_{k=1}^i
		\sum_{k'=1}^j
			R^{\cP\times \cP',l-1}_s(u_{k-1},v_{k'-1})
			R([u_{k-1},u_k]\times [v_{k'-1},v_{k'}]).
\end{align}
By using Lemma~\ref{for psi^{(l)}} inductively on $l$, we see that,
for any $\vep>0$, there exists an $\delta>0$ such that,
for $\cP$ and $\cP'$ which
satisfy $\max(|\cP|, |\cP'|)\le \delta$,
it holds that
\begin{align}
	\label{eq492823902}
	&
	\max
		\Big\{
			V_{(2H^{-})^{-1}}\Big(R^{\cP\times\cP,l}_s-R^l_s; I_{\cP^2}\Big),
			V_{(2H^{-})^{-1}}\Big(R^{\cP'\times\cP',l}_s-R^l_s; I_{(\cP')^2}\Big),\\ \notag
	&
	\qquad
	\qquad
	\qquad
	\qquad
	\qquad
			V_{(2H^{-})^{-1}}\Big(R^{\cP\times\cP',l}_s-R^l_s; I_{\cP\times \cP'}\Big)
		\Big\}
	\le
		\vep,
\end{align}
where $\frac{1}{3}<H^-<H$.

We next prove (\ref{eplP convergence}) by induction on $l$.
Clearly, the case where $l=2$ holds.
Suppose the case of $l-1$ holds.
Let $\cP=\{u_i\}_{i=1}^n$ be a partition of 
$[s,t]$
and let $\cP'$ be a refinement of the partition of $\cP$.
Then, we have
\begin{multline*}
	\left\|
		\psi^{(l),\cP}_{s,t}-
		\sum_{i=1}^n\psi^{(l-1)}_{s,u_{i-1}}\otimes \psi_{u_{i-1},u_i}
	\right\|_{\mathcal{H}^{\otimes l}}^2
	=
		\left\|
			\sum_{i=1}^n
				\left(
				\psi^{(l-1),\cP}_{s,u_{i-1}}-
				\psi^{(l-1)}_{s,u_{i-1}}\right)\otimes \psi_{u_{i-1},u_i}
		\right\|_{\mathcal{H}^{\otimes l}}^2\\
	\begin{aligned}
		&=
			\lim_{|\cP'|\to 0}
				\left\|
					\sum_{i=1}^n
						\left(
							\psi^{(l-1),\cP}_{s,u_{i-1}}
							-
							\psi^{(l-1),\cP'}_{s,u_{i-1}}
						\right)
						\otimes
						\psi_{u_{i-1},u_i}
				\right\|_{\mathcal{H}^{\otimes l}}^2\\
			&=
				\lim_{|\cP'|\to 0}
					\sum_{i,j=1}^n
						\big\{
								R^{\cP\times\cP,l-1}_s(u_{i-1},u_{j-1})
								+R^{\cP'\times\cP',l-1}_s(u_{i-1},u_{j-1})\\
			&\qquad\qquad\qquad\qquad\qquad\qquad\qquad
								-2R^{\cP\times\cP',l-1}_s(u_{i-1},u_{j-1})
						\big\}
						R([u_{i-1},u_i]\times[u_{j-1},u_j]).
	\end{aligned}
\end{multline*}
Combining the above with \eqref{eq492823902} and \eqref{vepl convergence},
we arrive at the desired convergence.
\end{proof}

\begin{corollary}\label{finite dimensional approximation}
We consider $d$-dimensional Gaussian processes $(B_t)$ satisfying 
Condition~$\ref{condition on R}$.
 Let $1\le \alpha_1,\ldots,\alpha_l\le d$.
Let $\cP=\{u_i\}_{i=0}^n$ be a partition of
$[s,t]$.
\begin{enumerate}
 \item[$(1)$] Let
\begin{align*}
 \psi^{\alpha_1,\ldots,\alpha_l,\cP}_{s,t}
=\sum_{1\le j_1<\cdots<j_l\le n}
\psi^{\alpha_1}_{s,u_{j_1}}\otimes\cdots\otimes
 \psi^{\alpha_k}_{u_{j_k-1},u_{j_k}}\otimes
\cdots \otimes\psi^{\alpha_l}_{u_{j_l-1},u_{j_l}}.
\end{align*}
Then 
$\lim_{|\cP|\to 0}\psi^{\alpha_1,\ldots,\alpha_l,\cP}_{s,t}$
converges in $(\mathcal{H}^d)^{\otimes l}$.
\item[$(2)$] 
Suppose the indices $\alpha_1,\ldots,\alpha_l$ mutually differ.
Let
\begin{align*}
 \tilde{\psi}^{\alpha_1,\ldots,\alpha_l,\cP}_{s,t}
=\sum_{1\le j_1<\cdots<j_l\le n}
\psi^{\alpha_1}_{s,u_{j_1}}\odot\cdots\odot
 \psi^{\alpha_k}_{u_{j_k-1},u_{j_k}}\odot
\cdots \odot\psi^{\alpha_l}_{u_{j_l-1},u_{j_l}}.
\end{align*}
Then 
$\lim_{|\cP|\to 0}\tilde{\psi}
^{\alpha_1,\ldots,\alpha_l,\cP}_{s,t}$
converges in $(\mathcal{H}^d)^{\odot l}$ and
\begin{align*}
 B^{\alpha_1,\ldots,\alpha_l}_{s,t}&=
\lim_{|\cP|\to 0}I_{l}
\left(\tilde{\psi}^{\alpha_1,\ldots,\alpha_l,\cP}_{s,t}\right)
\quad \text{in $L^p$ for all $p\ge 1$.}
\end{align*}
\end{enumerate}
\end{corollary}

\begin{remark}
Let
$\psi^{\alpha_1,\ldots,\alpha_l}_{s,t}=
\lim_{|\cP|\to 0}\psi^{\alpha_1,\ldots,\alpha_l,\cP}_{s,t}$.
If the same index appears in $\alpha_1,\ldots,\alpha_l$,
then $B^{\alpha_1,\ldots,\alpha_l}_{s,t}\ne
I_{l}\left(\psi^{\alpha_1,\ldots,\alpha_l}_{s,t}\right)$ in general.
\end{remark}

\section{Malliavin derivatives of iterated rough integrals}
\label{malliavin}

Throughout this section, we always assume that the 
driving Gaussian process $(B_t(\omega))$ $(\omega\in \Omega=C([0,1],\RR^d))$
satisfies Condition~\ref{condition on R}.
Recall that $\frac{1}{3}<H\leq\frac{1}{2}$ is assumed in Condition~\ref{condition on R}.
This section introduces a class of Wiener functionals
$\mathcal{I}(\RR)$ and presents calculation of the Malliavin derivatives.
Hereinafter, as in the Introduction, $Y_t$ and $J_t$ respectively denote the solutions to (\ref{rde}) and
(\ref{jacobian}).
It is known that $\sup_{0\le t\le 1}(|J_t|+|J^{-1}_t|)\in L^{\infty-}$ holds \cite{cll2013}.

\begin{definition}\phantomsection\label{def tilde I(R)}
\begin{enumerate}
 \item 
	We define $\IteratedIntegrals_l(\RR)$ inductively as presented below.
	\begin{enumerate}
		\item[(i)]	Let $\varphi$ denote a $C^{\infty}$ function on
				$\RR^n\times \mathcal{L}(\RR^n)\times 
				\mathcal{L}(\RR^n)$ with values in $\RR$
				such that all the derivatives and itself are 
				at most polynomial order growth.
				We denote the total set of functions given as
				$a=(a(t))=\left(\varphi(Y_t,J_t,J^{-1}_t)\right)$ for all
				such $\varphi$
				by $\IteratedIntegrals_0(\RR)$.
		\item[(ii)]	For $\alpha_1,\ldots,\alpha_l\in \{0,1,\ldots,d\}$
		and $(a_1(t)),\ldots,(a_l(t))\in \IteratedIntegrals_0(\RR)$, define
			\begin{align*}
		I^{\alpha_1}_{a_1}(t)
			&=\int_0^ta_1(s)\,dB^{\alpha_1}_s,\quad
		I^{\alpha_1,\ldots,\alpha_l}_{a_1,\ldots,a_l}(t)
		=\int_0^tI^{\alpha_1,\ldots,\alpha_{l-1}}
		_{a_1,\ldots,a_{l-1}}(s)a_l(s)\,
		dB^{\alpha_l}_s,\quad l\ge 2,
			\end{align*}
		where
		$B^{\alpha}_t=(B_t,e_{\alpha})$ and
		$B^0_t=t$.
		We call $I^{\alpha_1,\ldots,\alpha_l}_{a_1,\ldots,a_l}$
		an $l$-iterated integral
		and denote the sets of $l$-iterated integrals by
		$\IteratedIntegrals_l(\RR)$.
\end{enumerate}
\item 	Let $\spanIteratedIntegrals_l(\RR)$ denote the set of
linear span of $\IteratedIntegrals_l(\RR)$
and set $\spanIteratedIntegrals(\RR)=\cup_{l\ge 0}\spanIteratedIntegrals_l(\RR)$.
\item Let $\spanIteratedIntegrals(\RR^N)$ denote the set of all 
$\RR^N$-valued stochastic processes $F(t)=(F_1(t),\ldots,F_N(t))$ $(0\le t\le 1)$,
where $(F_i(t))\in \spanIteratedIntegrals(\RR)$.
	\end{enumerate}
\end{definition}

Note that $(F_t)\in \mathcal{I}(\RR^N)$ satisfies the property 
$\sup_{t\in [0,1]}|F_t|\in L^{\infty-}$, which can be checked by estimate of the rough integrals.

The integration by parts formula for rough integrals implies the following lemma.

\begin{lemma}\label{product of iterated integrals}
    \begin{enumerate}
    	\item	We have
				\begin{multline*}
					I^{\alpha_1,\ldots,\alpha_{l_1}}_{a_1,\ldots,a_{l_1}}(t)
					I^{\tal_1,\ldots,\tal_{l_2}}_{\ta_1,\ldots,\ta_{l_2}}(t)
					=
						\int_0^t
							I^{\alpha_1,\ldots,\alpha_{l_1-1}}_{a_1,\ldots,a_{l_1-1}}(s)
							I^{\tal_1,\ldots,\tal_{l_2}}_{\ta_1,\ldots,\ta_{l_2}}(s)
							a_{l_1}(s)\,
							dB^{\alpha_{l_1}}_s\\
						+
						\int_0^t
							I^{\alpha_1,\ldots,\alpha_{l_1}}_{a_1,\ldots,a_{l_1}}(s)
							I^{\tal_1,\ldots,\tal_{l_2-1}}_{\ta_1,\ldots,\ta_{l_2-1}}(s)
							\ta_{l_2}(s)\,
							dB^{\tal_{l_2}}_s.
		       	\end{multline*}
		       	The relation still holds for $l_i=1$
				if we use the convention $I^{\alpha_1,\ldots,\alpha_{l_{i-1}}}_{a_1,\ldots,a_{l_{i-1}}}(t)=1$.
    	\item	We see that 
				$
					I^{\alpha_1,\ldots,\alpha_{l_1}}_{a_1,\ldots,a_{l_1}}(t)
		 			I^{\tal_1,\ldots,\tal_{l_2}}_{\ta_1,\ldots,\ta_{l_2}}(t)
				$
				is a finite sum of $l_1+l_2$ iterated integrals of
				$I^{\gamma_1,\ldots,\gamma_{l_1+l_2}}$.
				Here $(\gamma_1,\ldots,\gamma_{l_1+l_2})$ is
				a permutation of $\alpha_1,\ldots,\alpha_{l_1},\tal_1,\ldots,\tal_{l_2}$
				and the defining functions depend on $a_1,\ldots,a_{l_1},\ta_1,\ldots,\ta_{l_2}$.
	\end{enumerate}
\end{lemma}

Next, iterated integrals with respect to $\mathcal{H}^1$-path are introduced.
They are used to express Malliavin derivatives of $(F_t)\in\spanIteratedIntegrals(\RR)$.
          Let
   $h_i\in \mathcal{H}^1$ and $a_i\in \spanIteratedIntegrals(\RR)$ ($i=1,2,\ldots$).
Then 
we are able to define an iterated integral as described below
in the sense of the Young integral.
   That is, inductively, we define the following:
\begin{align*}
	\mathscr{A}_{a_1}[h_1](t)
	&=
		\int_0^t
			a_1(s)\,
			dh_1(s),\\
	\mathscr{A}_{a_1,\ldots,a_l}[h_1,\ldots,h_l](t)
	&=
		\int_0^t
			\mathscr{A}_{a_1,\ldots,a_{l-1}}[h_1,\ldots,h_{l-1}]_s
			a_l(s)\,
			dh_l(s),\quad l\geq 2.
\end{align*}
We may omit denoting the functions $a_1,\ldots,a_l$.
   A similar lemma to Lemma~\ref{product of iterated integrals}
   for these integrals holds true.
   Additionally, we have the following estimates.
\begin{lemma}\label{lem490313333}
	Let $a_i\in \spanIteratedIntegrals(\RR)$ and $0\le s_i<t_i\le 1$ $(1\le i\le r)$.
	For $0\le t,t_1,\ldots,t_r\le 1$, let
	\begin{align*}
	\phi_t(t_1,\ldots,t_r)
	&=
		\mathscr{A}_{a_1,\ldots,a_r}[\psi_{t_1},\ldots,\psi_{t_r}](t).
	\end{align*}
	Then, we have
	\begin{align}
		\label{eq843901890284}
		\phi_t([s_1,t_1]\times\cdots\times [s_r,t_r])
		&=
			\mathscr{A}_{a_1,\ldots,a_r}
				[\psi_{s_1,t_1},\ldots, \psi_{s_r,t_r}](t)
	\end{align}
	and
	\begin{align}
		\label{eq418908410941}
		\max_{0\le t\le 1}
			\left|
				\mathscr{A}_{a_1,\ldots,a_r}
					[\psi_{s_1,t_1},\ldots,\psi_{s_r,t_r}](t)
			\right|
		&\le
			3^r
				\left(\prod_{k=1}^r\|a_i\|_{\infty}\right)
				\left(\prod_{i=1}^r(t_i-s_i)^{2H}\right).
	\end{align}
Particularly, $\phi_t(t_1,\ldots,t_r)$ is a continuous function of
the variable $(t_1,\ldots,t_r)\in [0,1]^r$.
\end{lemma}
   
\begin{proof}
	The multi-linearity of the mapping 
	$
		(h_1,\ldots,h_r)
		\mapsto 
	   \mathscr{A}[h_1,\ldots, h_r]
	$
	implies \eqref{eq843901890284}.

	Also, $\|\psi_{s,t}\|_{1\hyp var, [0,1]}\le 3(t-s)^{2H}$ from Lemma~\ref{properties of R}.
	Therefore,
	\begin{align*}
		|\mathscr{A}_a(\psi_{s_1,t_1})(t)|
		=
			\left|\int_0^ta(u)d\psi_{s_1,t_1}(u)\right|
		\le
			\|a\|_{\infty}\|\psi_{s_1,t_1}\|_{1\hyp var}
		\le
			3\|a\|_{\infty}(t_1-s_1)^{2H}.
	\end{align*}
	Therefore, the estimate 
	\eqref{eq418908410941} is easily obtained using induction on $r$.
\end{proof}

Let $(F_t)\in\spanIteratedIntegrals(\RR)$.
First, we give a representation of the $r$-times stochastic G\^ateaux derivative of $F_t$
in Lemma~\ref{derivative of F}.
By using Lemma~\ref{derivative of F},
we can show $F_t\in \mathbb{D}^{\infty}(\RR)$
in Theorem~\ref{Malliavin differentiability}.
However, we will not use $F_t\in \mathbb{D}^{\infty}(\RR)$
in our proof of the main results.

Let us explain the definition of the stochastic G\^ateaux derivative in this paper.
Let $h\in \mathcal{H}^d$
and consider the full measure subset $\Omega'$ in Theorem~\ref{friz-hairer a theorem}.
Let $L^{\infty -}C^1_h(\Omega'\to \RR^N)$ 
be the set of all $\RR^N$-valued $L^{\infty -}$ functions $F$ on
$\Omega'$ such that
$u(\in \RR)\mapsto F(\omega+uh)$ is $C^1$
and satisfying that
the G\^ateaux derivative
$D_hF(\omega):=\lim_{u\to 0}u^{-1}(F(\omega+uh)-F(\omega))$ belongs to
$L^{\infty -}(\Omega')$.
We extend the domain of $D_h$ as described hereinafter.
Let $\mathcal{D}(D_h,\RR^N)$ be the set of all 
$F\in L^{\infty-}(\Omega,\RR^N)$ such that
for any $p>1$ there exist $F_{p,n}\in L^{\infty -}C^1_h(\Omega',\RR^N)$ $(n=1,2,\ldots)$
satisfying $\lim_{n\to\infty}F_{p,n}=F$ in $L^p$ and
$\lim_{n\to\infty}D_hF_{p,n}$ converges in $L^{p}(\Omega,\RR^N)$.
Then we define $D_hF:=\lim_{n\to\infty}D_hF_{p,n}$.
The limit is independent of the choice of the sequence and $p$.
Furthermore, it holds that $D_hF\in L^{\infty -}$.
In this paper, we call this derivative a stochastic G\^ateaux derivative in the direction
$h$.
For this derivative, we have the following.

\begin{lemma}\label{gateaux derivative}
 \begin{enumerate}
\item Let $\{F_n\}_{n=1}^{\infty}\subset \mathcal{D}(D_h,\RR^N)$ and
suppose $\lim_{n\to\infty}F_n=F$ in $L^{\infty-}$ and
$\lim_{n\to\infty}D_hF_n=G$ in $L^{\infty-}$.
Then $F\in \mathcal{D}(D_h,\RR^N)$ and $D_hF=G$ holds.
\item 
Let $F_i\in \mathcal{D}(D_h,\RR^{d_i})$ $(1\le i\le n)$.
Let $\varphi$ be a $C^{1}$ function on $\RR^{d_1}\times\cdots\times\RR^{d_n}$
such that $\varphi$ itself and its derivatives are at most polynomial growth.
Then $\varphi(F_1,\ldots,F_n)\in \mathcal{D}(D_h,\RR)$ and
\[
 D_h\varphi(F_1,\ldots,F_n)=\sum_{i=1}^n(\partial_{x_i}\varphi)(F_1,\ldots,F_n)[D_hF_i].
\]
 \end{enumerate}
\end{lemma}

\begin{proof}
These follow from standard calculations.
\end{proof}

\begin{lemma}\label{derivative of Y and J}
	For any $h\in \mathcal{H}^d$, we have $Y_t\in L^{\infty -}C^1_h(\Omega'\to \RR^n)$,
$J_t, J_t^{-1}\in L^{\infty -}C^1_h(\Omega'\to \mathcal{L}(\RR^n,\RR^n))$
and
\begin{align}
\label{derivative of Y}	D_{h}Y_t
	&=
		J_t\int_0^tJ_s^{-1}\sigma(Y_s)dh_s,\\
\label{derivative of J}	D_{h}J_t
	&=
		J_t\int_0^tJ_s^{-1}(D\sigma)(Y_s)[J_s]dh_s
		+J_t\int_0^tJ_s^{-1}(D^2\sigma)(Y_s)[D_{h}Y_s,J_s]dB_s \\
	&\quad
		+J_t\int_0^tJ_s^{-1}(D^2b)(Y_s)[D_{h}Y_s,J_s]ds.\nonumber
\end{align}
\end{lemma}
\begin{proof}
	$C^1$ property of $u(\in \RR)\mapsto Y_t(\omega+uh), 
J_t(\omega+uh), J^{-1}_t(\omega+uh)$
follows from the constant variation method and 
the continuity of rough integral with respect to
the driving rough path.
This calculation can be done by pathwise.
See Proposition 11.19 in \cite{friz-hairer}.
The integrability follows from Theorem~\ref{friz-hairer a theorem} (2) and
$\sup_{0\le t\le 1}(|J_t|+|J^{-1}_t|)\in L^{\infty-}$.
\end{proof}

Below we use the following notation.
Let $h_1,\ldots ,h_n\in \mathcal{H}^d$.
If $F\in L^{\infty-}(\Omega,\RR^N)$ satisfies 
\begin{itemize}
 \item[(i)] $F\in \mathcal{D}(D_{h_1},\RR^N)$ and set $F_1=D_{h_1}F$,
\item[(ii)] $\{F_i\}_{i=1}^n\subset L^{\infty-}(\Omega,\RR^N)$ can be defined inductively as
$F_{i-1}\in \mathcal{D}(D_{h_i},\RR^N)$ holds and $F_i=D_{h_{i}}F_{i-1}$ $(2\le i\le n)$,
\end{itemize}
then we write
$F\in \mathcal{D}(D_{h_n}\cdots D_{h_1},\RR^N)$ and 
$D_{h_1,\ldots,h_i}F:=F_i$.

Note that $F\in \mathcal{D}(D_{h_n}\cdots D_{h_1},\RR)$ for all 
$h_i\in \mathcal{H}^d$ $(1\le i\le n)$ 
is not sufficient to conclude that $F\in \mathbb{D}^{\infty}(\RR)$.
To prove $F\in \mathbb{D}^{n,\infty-}(\RR)$, we need to
prove that there exists $\Xi_r\in L^{\infty-}(\Omega,(\mathcal{H}^d)^{\odot r})$ such that
$(\Xi_r,h_1\otimes\cdots \otimes h_r)_{(\mathcal{H}^d)^{\otimes r}}
=D_{h_1,\ldots,h_r}F$ for all $1\le r\le n$.
In this paper, we prove this by using Lemma~\ref{derivative of F} for
$(F_t)\in\mathcal{I}(\RR)$.
The higher order Malliavin differentiability of
$Y_t, J_t, J^{-1}_t$ has already been studied in
\cite{inahama, hairer-pillai}.

We now give the representation formula for 
the $r$-times stochastic G\^ateaux differential of elements in $\mathcal{I}(\RR)$
using the iterated integral $\mathscr{A}$.

\begin{lemma}\label{derivative of F}
 Let $h_i\in \mathcal{H}^1$ and
 $v_i=(v_i^j)_{j=1}^d\in \RR^d$ $(1\le i\le r)$.
Let $\boldsymbol{J}=\{\boldsymbol{j}=(j_1,\ldots,j_r)~|~1\le j_1,\ldots,j_r\le d\}$ and write
$\boldsymbol{v}^{\boldsymbol{j}}=\prod_{i=1}^rv_i^{j_i}$.
 Let $(F_t)\in \spanIteratedIntegrals(\RR)$.
Then $F_t\in \mathcal{D}\left(D_{h_rv_r}\cdots D_{h_1v_1},\RR\right)$ $(0\le t\le 1)$ 
and the following holds:
 there exist $N\in \mathbb{N}$, 
$F_{i,\sigma,\boldsymbol{j}},
a_{k,i,\sigma,\boldsymbol{j}}\in \spanIteratedIntegrals(\RR)$
~$(1\le i\le N, 1\le k\le r, \boldsymbol{j}\in \boldsymbol{J})$ such that
 \begin{align}
  D_{h_1v_1,\ldots,h_rv_r}F_t
  &=\sum_{1\le i\le N, \sigma\in \mathfrak{G}_r,\boldsymbol{j}\in \boldsymbol{J}}
  \boldsymbol{v}^{\boldsymbol{j}} F_{i,\sigma,\boldsymbol{j}}(t)
  \mathscr{A}_{a_{1,i,\sigma,\boldsymbol{j}},\ldots,
a_{r,i,\sigma,\boldsymbol{j}}}
[h_{\sigma(1)},\ldots, h_{\sigma(r)}](t).
\label{Malliavin derivative and iterated integrals}
 \end{align}
\end{lemma}

\begin{remark}\label{rem3892803123}
	\begin{enumerate}
		\item	Lemma~\textup{\ref{derivative of F}} above
				is a stronger result for our purpose.
				Actually, such a strong result is not needed
				to prove our main theorems.
				As stated in Remark~\ref{rem47312897319},
				it is sufficient to show higher order stochastic 
				G\^ateaux differentiability in the directions
				$\psi^{\alpha}_{s,t}$ of $F_t$ and
				the estimates \eqref{eq4i319014} and \eqref{estimate of phi alpha_r}.
				In this sense, giving Lemma~\textup{\ref{derivative of F}}
				for general $h_1,\dots,h_r$
				is more than what is needed.
				We will use Lemma~\textup{\ref{derivative of F}}
				to show $F_t\in \mathbb{D}^{\infty}(\RR)$.
		\item	As stated in the Introduction,
				when we try to extend Lemma~\ref{derivative of F1} to the case where $\frac{1}{4}<H\leq \frac{1}{2}$,
				we need to include the third-level rough paths in \eqref{eq532094829018} and so on.
				If we can extend Lemma~\ref{derivative of F1}, then 
				we might not need to change the following proof of Lemma~$\ref{derivative of F}$.
				Of course, we assume $\frac{1}{3}<H\leq \frac{1}{2}$ in the following proof.
	\end{enumerate}
\end{remark}

We start by proving the case $r=1$ in the following form.
We now recall the definition of the Gubinelli derivative.
Let $V$ be a finite dimensional vector space.
For a $V$-valued controlled path $Z_t$,
the Gubinelli derivative $Z^{(\beta)}_t$ 
is defined to be a $V$-valued continuous path
such that there exists $C>0$ satisfying
\begin{align*}
 |Z_t-Z_s-\sum_{\beta=1}^dZ^{(\beta)}_sB^\beta_{s,t}|
\le C\left(|t-s|+\|h\|_{(2H^-)^{-1}\hyp var,[s,t]}^{(2H^-)^{-1}}\right)^{2H^-},
\quad 0\le s<t\le 1.
\end{align*}
Here, note that the right-hand side contains the $(2H^-)^{-1}$-variation norm of
$h$ because the integrals which we are dealing with contain 
the integrals with respect to
$h$.
In general, $Z^{(\beta)}_t$ is not uniquely defined, but in the present case
they are defined by their definitions of the processes.
Indeed,
we can calculate the Gubinelli derivatives of $Y_t$, $J_t$, $J_t^{-1}$,
$D_{h v}Y_t$, $D_{h v}J_t$, $D_{hv}J_t^{-1}$ as presented below.
\begin{align*}
Y_t^{(\beta)}&=\sigma(Y_t)e_{\beta}, \qquad
J_t^{(\beta)}=(D\sigma)(Y_t)[J_t]e_{\beta}, \qquad
(J_t^{-1})^{(\beta)}=-J_t^{-1}(D\sigma)(Y_t)[\cdot]e_{\beta}, \\
 (D_{ hv}Y)^{(\beta)}_t&=(D\sigma)(Y_t)[D_{h v}Y_t]e_{\beta},\\
(D_{h v}J)^{(\beta)}_t&=(D\sigma)(Y_t)[D_{h v}J_t]e_{\beta}+
(D^2\sigma)(Y_t)[D_{h v}Y_t,J_t]e_{\beta},\\
(D_{h v}J^{-1}_t)^{(\beta)}&=
J_t^{-1}(D_{h v}J_t)J_t^{-1}(D\sigma)(Y_t)[\cdot]e_{\beta}-
J_t^{-1}(D^2\sigma)(Y_t)[D_{h v}Y_t,\cdot]e_{\beta}.
\end{align*}
Using this, we obtain ``the commutativity of the two derivatives'' for
$(a_t)\in \IteratedIntegrals_0(\RR)$, that is, 
$D_{h v}(a^{(\beta)}_t)=(D_{h v}a_t)^{(\beta)}$.
We use these results in the following proof.

\begin{lemma}\label{derivative of F1}
For $\{a_i\}_{i=0}^{\infty}\subset \IteratedIntegrals_0(\RR)$
and $\{\alpha_i\}_{i=1}^{\infty}\subset \{0,1,\ldots,d\}$ define inductively
iterated integrals belonging to
$\IteratedIntegrals_k(\RR)$ $(k\ge 0)$ by
\begin{align}
	I_0(t)
	&=
		a_0(t),
	&
	I_k(t)
	&=
		\int_0^t 
			a_k(s)
			I_{k-1}(s)
			dB^{\alpha_k}_s,
\quad k\ge 1.\label{Ik+1}
\end{align}
Let $h\in \mathcal{H}^1$ and $v=(v^j)_{j=1}^d\in \RR^d$.
Then, the following hold.
\begin{enumerate}
 \item $I_k(t)\in \mathcal{D}(D_{hv},\RR)$ $(0\le t\le 1)$ hold for all $k$.
Moreover $D_{hv}I_k(t)$ are controlled paths for all $k$.
\item There exist $N\in \mathbb{N}$ which may depend on $k$,
$\hat{I}^k_{j,i}, \check{I}^k_{j,i}\in
\spanIteratedIntegrals(\RR)$ $(1\le i\le N, 1\le j\le d)$
such that
\begin{align}
D_{h v}I_k(t)&=\sum_{i=1}^N\sum_{j=1}^dv^j\check{I}^k_{j,i}(t)\int_0^t
\hat{I}^k_{j,i}(s)dh_s,
\quad k\ge 0,\label{integral1}\\
D_{h v}I_k(t)&=
\int_0^tD_{h v}\left(a_k(s)I_{k-1}(s)\right)dB^{\alpha_k}_s+
\int_0^ta_k(s)I_{k-1}(s)(v,e_{\alpha_k})(1-\delta_{\alpha_k,0})
dh_s,\quad k\ge 1.\label{integral2}
\end{align}
\end{enumerate}
\end{lemma}

\begin{remark}
All elements in $\IteratedIntegrals_k$ are obtained by \eqref{Ik+1}.
See Definition~\ref{def tilde I(R)}.
\end{remark}

 \begin{proof}[Proof of Lemma~$\ref{derivative of F1}$]
We prove (1) and (2) by induction.
We consider the case $k=0$.
By Lemma~\ref{derivative of Y and J}, we see that
$D_{hv}Y_t$ and the first term of $D_{hv}J_t$ are 
of the forms of \eqref{integral1}.
We consider the second term in $D_{hv}J_t$.
Let $e_1,\dots,e_n$ be the standard basis of $\RR^n$.
Putting $D_{h v}Y_t=\sum_{\alpha=1}^n J_te_\alpha \int_0^t (J_s^{-1}\sigma(Y_s)v,e_\alpha) dh_s$ into
$A_t=\int_0^tJ_s^{-1}(D^2\sigma)(Y_s)[D_{h v}Y_s,J_s]dB_s$, we have
\begin{align*}
	A_t
	&=
		\sum_{\alpha}
			\int_0^tJ_s^{-1}(D^2\sigma)(Y_s)[J_se_{\alpha},J_s]
				\left(\int_0^s\left(J_u^{-1}\sigma(Y_u)v,e_{\alpha}\right)dh_u\right)
				dB_s.
\end{align*}
Using the integration by parts formula
$\int_0^tV_udU_u=V_tU_t-V_0U_0-\int_0^tU_udV_u$ for
controlled rough paths $V,U$,
we have
\begin{align*}
	A_t
&=\sum_{\alpha}\int_0^t(J_u^{-1}\sigma(Y_u)v,e_{\alpha})dh_u
\int_0^tJ_s^{-1}(D^2\sigma)(Y_s)[J_se_{\alpha},J_s]dB_s\\
&\quad -\sum_{\alpha}
\int_0^t\left(\int_0^uJ_s^{-1}(D^2\sigma)(Y_s)[J_se_{\alpha},J_s]dB_s\right)
\left(J_u^{-1}\sigma(Y_u)v,e_{\alpha}\right)dh_u.
\end{align*}
Therefore, this term is also of the form of (\ref{integral1}).
As for $J_t^{-1}$, we can calculate 
the derivative as $D_{h v}(J_t^{-1})=-J_t^{-1}D_{h v}J_t J_t^{-1}$.
Using this derivative and also by using the integration by parts formula,
we see that the case $k=0$ holds.

Assume that (1) and (2) hold up to $k$.
We prove the case of $k+1$.
By the assumption of the induction, we see that the rough integral
in (\ref{integral2}) is well-defined.
When $\alpha_{k+1}=0$, the proof is easy. We consider the case
$\alpha_{k+1}\ge 1$.
In order to prove $I_{k+1}(t)\in \mathcal{D}(D_{hv},\RR)$ and
calculate $D_{hv}I_{k+1}(t)$, we consider approximation processes
of $I_{k+1}(t)$.
Let $\cP=\{t_i\}_{i=0}^n$ be a partition of $[0,t]$ and set
\begin{align}\label{eq532094829018}
 I_{k+1}^{\cP}(t)&=
\sum_{i=1}^n\left(a_{k+1}(t_{i-1})I_k(t_{i-1})B^{\alpha_{k+1}}_{t_{i-1},t_i}
+\sum_{\beta=1}^d\left(a_{k+1}I_k\right)^{(\beta)}_{t_{i-1}}
B^{\beta,\alpha_{k+1}}_{t_{i-1},t_i}\right).
\end{align}
Then $\lim_{|\cP|\to 0}I^{\cP}_{k+1}(t)=I_{k+1}(t)$ in
$L^{\infty-}$.
Also we have
$
	D_{h v}I^{\cP}_{k+1}(t)
	=
		I^{\cP,1}(t)+I^{\cP,2}(t)+I^{\cP,3}(t)
$,		
where
\begin{align*}
	I^{\cP,1}(t)
	&=
		\sum_{i=1}^nD_{h v}\left(a_{k+1}(t_{i-1})I_k(t_{i-1})\right)
		B^{\alpha_{k+1}}_{t_{i-1},t_i}
		+\sum_{i=1}^n
		\sum_{\beta=1}^dD_{h v}
		\left\{\left(a_{k+1}I_k\right)^{(\beta)}_{t_{i-1}}\right\}
		B^{\beta,\alpha_{k+1}}_{t_{i-1},t_i},\\
	I^{\cP,2}(t)
	&=
		\sum_{i=1}^n
			a_{k+1}(t_{i-1})I_k(t_{i-1})(v,e_{\alpha_{k+1}})h_{t_{i-1},t_i},
	\quad 
	I^{\cP,3}(t)
	=
		\sum_{i=1}^n\sum_{\beta=1}^d\left(a_{k+1}I_k\right)^{(\beta)}_{t_{i-1}}
			D_{h v}B^{\beta,\alpha_{k+1}}_{t_{i-1},t_i}.
\end{align*}
It is easy to see that
$\lim_{|\cP|\to 0}I^{\cP,2}(t)
=\int_0^ta_{k+1}(s)I_k(s)(v,e_{\alpha_{k+1}})dh(s)$ in $L^{\infty-}$.
Next we will show
\begin{align}
	\label{conv 0 IP3}
	\lim_{|\cP|\to 0}I^{\cP,3}(t)=0
	\qquad
	\text{in $L^{\infty-}$.}
\end{align}
We have
\begin{align*}
 D_{h v}B^{\beta,\alpha_{k+1}}_{t_{i-1},t_i}&=
(v,e_{\beta})\int_{t_{i-1}}^{t_i}h_{t_{i-1},s}dB^{\alpha_{k+1}}_s
+(v,e_{\alpha_{k+1}})
\int_{t_{i-1}}^{t_i}B^{\beta}_{t_{i-1},s}dh_s.
\end{align*}
Note that
$\theta:=2H+H^->1$.
By the estimate of Young integral, we obtain
\begin{align*}
J_i&:=\Bigg|\int_{t_{i-1}}^{t_i}h_{t_{i-1},s}dB^{\alpha_{k+1}}_s\Bigg|+
\left|\int_{t_{i-1}}^{t_i}B^{\beta}_{t_{i-1},s}dh_s\right|\\
&\le
C\bigg\{
\|B^{\alpha_{k+1}}\|
_{(H^-)^{-1}\hyp var, [t_{i-1},t_i]}
+\|B^{\beta}\|
_{(H^-)^{-1}\hyp var, [t_{i-1},t_i]}
\bigg\}\|h\|_{(2H)^{-1}\hyp var, [t_{i-1},t_i]}\\
&\le
C\bigg\{\|B^{\alpha_{k+1}}\|
_{(H^-)^{-1}\hyp var, [t_{i-1},t_i]}^{\theta(H^-)^{-1}}+
\|B^{\beta}\|
_{(H^-)^{-1}\hyp var, [t_{i-1},t_i]}^{\theta(H^-)^{-1}}+
\|h\|_{(2H)^{-1}\hyp var, [t_{i-1},t_i]}
^{(2H)^{-1}}\bigg\}
\|h\|_{(2H)^{-1}\hyp var, [t_{i-1},t_i]}^{1-\frac{1}{\theta}},
\end{align*}
where we have used an elementary inequality 
$ab\le \frac{a^p}{p}+\frac{b^q}{q}$ $(a,b \ge 0, \frac{1}{p}+\frac{1}{q}=1)$
in the third inequality.
Because $\|B\|_{(H^-)^{-1}\hyp var, [0,1]}<\infty$
and $\|h\|_{(2H)^{-1}\hyp var, [0,1]}<\infty$,
this implies
\begin{align*}
 \sum_{i=1}^nJ_i
\le
C
\bigg\{
\|B^{\alpha_{k+1}}\|_{(H^-)^{-1}\hyp var, [0,1]}^{\theta(H^-)^{-1}}
+\|B^{\beta}\|_{(H^-)^{-1}\hyp var, [0,1]}^{\theta(H^-)^{-1}}
+\|h\|_{(2H)^{-1}\hyp var, [0,1]}
^{(2H)^{-1}}
\bigg\}
\max_i
\|h\|_{(2H)^{-1}\hyp var, [t_{i-1},t_i]}^{1-\frac{1}{\theta}}.
\end{align*}
Because for any $F=(F_t)\in
\spanIteratedIntegrals(\RR)$, $\max_t|F_t|\in L^{\infty-}$ holds,
this estimate 
implies \eqref{conv 0 IP3}.

We consider $I^{\cP,1}(t)$.
We show
$
(D_{h v}\{a_{k+1}(t)I_k(t)\})^{(\beta)}
=D_{h v}
\{(a_{k+1}I_k)^{(\beta)}_t\}
$.
Here we write $t=t_{i-1}$ for notational simplicity.
By the definition of $I_k(t)$,
$I_k^{(\beta)}(t)=a_k(t)I_{k-1}(t)\delta_{\beta,\alpha_k}$.
Therefore
\begin{align*}
	D_{h v}\left\{\left(a_{k+1}I_k\right)^{(\beta)}_t\right\}
	&=
		D_{h v}
			\left\{
				a_{k+1}^{(\beta)}(t)I_k(t)
				+a_{k+1}(t)a_{k}(t)I_{k-1}(t)\delta_{\beta,\alpha_k}
			\right\}\\
	&=
		\{D_{h v}a_{k+1}^{(\beta)}(t)\}I_k(t)
		+a_{k+1}^{(\beta)}(t)D_{h v}I_{k}(t) \\
	&\phantom{=}\qquad
		+
		\{D_{h v}a_{k+1}(t)\}
		a_k(t)
		I_{k-1}(t)\delta_{\beta,\alpha_k}
		+
		a_{k+1}(t)
		D_{h v}\{a_k(t)I_{k-1}(t)\}
		\delta_{\beta,\alpha_k}.
\end{align*}
On the other hand,
\begin{align*}
	(D_{h v}\{a_{k+1}(t)I_k(t)\})^{(\beta)}
	&=
		\left(
			D_{h v}a_{k+1}(t)I_k(t)
			+
			a_{k+1}(t)D_{h v}I_k(t)
		\right)^{(\beta)}\\
	&=
		(D_{h v}a_{k+1}(t))^{(\beta)}I_k(t)
		+D_{h v}a_{k+1}(t)I_k^{(\beta)}(t)  \\
	&\phantom{=}\qquad
		+a_{k+1}^{(\beta)}(t)D_{h v}I_k(t)
		+a_{k+1}(t)\left(D_{h v}I_k(t)\right)^{(\beta)}\\
	&=
		\{D_{h v}a_{k+1}^{(\beta)}(t)\}I_k(t)
		+D_{h v}a_{k+1}(t)a_k(t)I_{k-1}(t)\delta_{\beta,\alpha_k} \\
	&\phantom{=}\qquad
		+a_{k+1}^{(\beta)}(t)D_{h v}I_k(t)
		+a_{k+1}(t)D_{h v}\{a_k(t)I_{k-1}(t)\}\delta_{\beta,\alpha_k}.
\end{align*}
In the last line above, we used ``the commutativity of the two derivatives'' for
$(a_{k+1}(t))\in \IteratedIntegrals_0(\RR)$.
That is, $(D_{h v}a_{k+1}(t))^{(\beta)}=D_{h v}a_{k+1}^{(\beta)}(t)$, and 
$\left(D_{h v}I_k(t)\right)^{(\beta)}=D_{hv}\{a_k(t)I_{k-1}(t)\}\delta_{\beta,\alpha_k}$,
which follows from the assumption of
the induction.
Consequently, we obtain
\begin{align*}
 \lim_{|\cP|\to 0}I^{\cP,1}(t)&=
\int_0^tD_{h v}\left(a_{k+1}(s)I_k(s)\right)dB^{\alpha_{k+1}}_s\quad 
\text{in $L^{\infty-}$}.
\end{align*}
Therefore, we have proved (\ref{integral2}) in the case of $k+1$.
The representation (\ref{integral1}) for $D_{hv}I_{k+1}(t)$ follows from 
the representation (\ref{integral1}) of $D_{hv}I_k(t)$
and the integration by parts formula for controlled paths.
 \end{proof}

\begin{proof}[Proof of Lemma~$\ref{derivative of F}$]
We can prove the assertion by induction by 
using Lemma~\ref{derivative of F1},
integration by parts formula of rough integrals.
By (\ref{integral1}), we see that the statement holds in the case where
$r=1$.
We assume the assertion holds in the case of $r$.
We denote by $S=\{i_1<\cdots<i_{r'}\}$ a subset of
$\{2,\ldots,r+1\}$ and let 
$S^c=\{j_1<\cdots<j_{r-r'}\}$.
We allow $S=\emptyset$.
By (\ref{integral1}), we have
\begin{align*}
 D_{h_1v_1}F_t=
\sum_{i=1}^N\sum_{j=1}^dv_1^j\check{I}_{j,i}(t)\int_0^t
\hat{I}_{j,i}(s)dh_1(s).
\end{align*}
Using the approximation,
$\int_0^t\hat{I}_{j,i}(s)dh_1(s)=
\lim_{|\mathcal{P}|\to 0}\sum_{l}\hat{I}_{j,i}(s_{l-1})(h_1(s_l)-h_1(s_{l-1}))$
and Lemma~\ref{gateaux derivative},
\begin{align*}
&D_{h_1v_1,\ldots,h_{r+1}v_{r+1}}F_t\\
&\quad=
\sum_{i,j}v_1^j\sum_{S\subset \{2,\ldots,r+1\}}
D_{h_{i_1}v_{i_1},\ldots,h_{i_{r'}}v_{i_{r'}}}\check{I}_{j,i}(t)
\int_0^t
(D_{h_{j_1}v_{j_1},\ldots,h_{j_{r-r'}}v_{j_{r-r'}}}
\hat{I}_{j,i}(s))dh_1(s).
\end{align*}
By the assumption of the induction,
\[
 D_{h_{i_1}v_{i_1},\ldots,h_{i_{r'}}v_{i_{r'}}}\check{I}_{j,i}(t)
\quad
\text{and}\quad
D_{h_{j_1}v_{j_1},\ldots,h_{j_{r-r'}}v_{j_{r-r'}}}
\hat{I}_{j,i}(s)
\]
can be written as in (\ref{Malliavin derivative and iterated integrals}).
Then, applying the integration by parts formula to this identity, we complete
the proof.
\end{proof}

We will show $F_t\in \mathbb{D}^{\infty}(\RR)$
in the following (Lemma~\ref{iterated integral for CM},
Remark~\ref{Hilbert-Schmidt property}, and Theorem~\ref{Malliavin differentiability}).

\begin{lemma}\label{iterated integral for CM}
We consider one-dimensional Gaussian process $B_t$ satisfying 
Condition~$\ref{condition on R}$.
  Let $\mathcal{H}$ be the Cameron-Martin space of $B$.
  Let $a_i$ be a finite $(H^-)^{-1}\hyp$ variation path
   $(1\le i\le r)$.
Let $\{h_i\}_{i=1}^{\infty}$ be an orthonormal basis of $\mathcal{H}$.
   For
$\{c_{i_1,\ldots,i_r}\}\in l^2$, define
\begin{align*}
	f_N=f_N(t_1,\ldots,t_r)=
	\sum_{1\le i_1,\ldots,i_r\le N}c_{i_1,\ldots,i_r}h_{i_1}(t_1)
	 \cdots h_{i_r}(t_r).	
\end{align*}
We set
  \begin{align}
   \mathscr{A}[f_N](t)&=\sum_{1\le i_1,\ldots,i_r\le N}
   c_{i_1,\ldots,i_r}\mathscr{A}_{a_1,\ldots,a_r}[h_{i_1},\ldots,h_{i_r}],
  \end{align}
where $\mathscr{A}_{a_1,\ldots,a_n}$ is the iterated integral which we already defined.
  Then we have the following estimate
  \begin{align}
   \max_{0\le t\le 1}\left|\mathscr{A}[f_N](t)\right|\le C\,
\left(\prod_{i=1}^r\left(\|a_i\|_{(H^{-})^{-1}\hyp var}+
\|a_i\|_{\infty}\right)\right)\,
   \|f_N\|_{\mathcal{H}^{\otimes r}},\label{estimate of A}
  \end{align}
   where $C$ is independent of $N$.
   In particular, the linear map $f_N\mapsto \mathscr{A}[f_N](t)\in \RR$
   can be extended to a uniquely determined continuous linear functional
   from $\mathcal{H}^{\otimes r}$ to $\RR$
   and the operator norm can be estimated in the same way.
We denote the continuous linear functional by the same notation $\mathscr{A}[f](t)$
$(f\in \mathcal{H}^{\otimes r})$.
Furthermore, it holds that, for any $g_i\in \mathcal{H}$ $(1\le i\le r)$
\begin{align}
 \mathscr{A}[g_1\otimes\cdots\otimes g_r](t)=
\mathscr{A}_{a_1,\ldots,a_r}[g_1,\ldots,g_r](t).\label{consistent}
\end{align}
 \end{lemma}

 \begin{proof}
Let $\{B^l_t\}_{l=1}^r$ be independent copies of $(B_t)$.
There exists an orthonormal basis of the Wiener chaos of order 1, $\{Z^l_i\}_{i=1}^{\infty}$,
such that
   $h_i(t)=E[Z^l_iB^l_t]$ $(i=1,2,\ldots)$.
Let $t^m_k=kt2^{-m}$ $(0\le k\le 2^m)$
  and we consider the dyadic partition $\mathcal{D}_m=\{(u,v)~|~u, v\in \{t^m_k\}\}$ 
of $[0,t]^2$.
By applying the estimate in one-dimensional Young integral successively, we have
\begin{align*}
&\mathscr{A}_{a_1,\ldots,a_r}[h_{i_1},\ldots,h_{i_r}](t)\\
&=
\lim_{m\to\infty} 
\sum_{0\le k_1<\cdots<k_r\le 2^m}
a_1(t_{k_1}^m)\cdots a_r(t^m_{k_r})
\left(h_{i_1}(t^m_{k_1+1})-h_{i_1}(t^m_{k_1})\right)
\cdots \left(h_{i_r}(t^m_{k_{i_r}+1})-h_i(t^m_{k_{i_r}})\right)\\
&=\lim_{m\to\infty} 
\sum_{0\le k_1<\cdots<k_r\le 2^m}
a_1(t_{k_1}^m)\cdots a_r(t^m_{k_r})
E\left[Z^1_{i_1}B^1_{t^m_{k_1},t^m_{k_1+1}}
\cdots Z^r_{i_r}B^r_{t^m_{k_r},t^m_{k_r+1}}\right].
\end{align*}
Therefore,
  \begin{align}
&  \label{representation of AtfN} \mathscr{A}[f_N](t)=
   \sum_{1\le i_1,\ldots,i_r\le N}c_{i_1,\ldots,i_r}
   \mathscr{A}_{a_1,\ldots,a_r}[h_{i_1},\ldots,h_{i_r}](t)    \\
   &=\lim_{m\to\infty}
E\left[\left(\sum_{1\le i_1,\ldots,i_r\le N}
   c_{i_1,\ldots,i_r}Z^1_{i_1}\cdots Z^r_{i_r}\right)
\sum_{0\le k_1<\cdots<k_r\le 2^m}
a_1(t_{k_1}^m)\cdots a_r(t^m_{k_r})
B^1_{t^m_{k_1},t^m_{k_1+1}}
   \cdots B^r_{t^m_{k_r},t^m_{k_r+1}}\right].\nonumber
  \end{align}
Let $m$ be a sufficiently large number such that $2^m>r$
and let
$\{F_l(u,v)\}_{l=0}^r\subset C([0,t]^2_{\mathcal{D}_m})$ 
be functions defined on partition points $\mathcal{D}_m$
(See Section~\ref{appendix I} for this notion):
  \begin{align*}
   F_0(u,v)&=1,\\
   F_l(t^m_i,t^m_j)&=
   \sum_{0\le p\le i-1, 0\le q\le j-1}
   a_l(t^m_{p})a_l(t^m_{q})
   F_{l-1}(t^m_{p},t^m_{q})R([t^m_{p},t^m_{p+1}]\times 
   [t^m_{q},t^m_{q+1}])\quad i,j\ge 1,\quad l\ge 1,\\
   F_l(\tmi,\tmj)&=0\quad 0\le i\le l-1~\text{or}~0\le j\le l-1, \quad l\ge 1.
  \end{align*}
  Then
  \begin{align*}
&   E\left[
\left(
\sum_{0\le k_1<\cdots<k_r\le 2^m}
a_1(t_{k_1}^m)\cdots a_r(t^m_{k_r})
B^1_{t^m_{k_1},t^m_{k_1+1}}
\cdots B^r_{t^m_{k_r},t^m_{k_r+1}}
\right)^2
   \right]\\
   &=
\sum_{0\le k_1<k_2<\cdots<k_r\le 2^m-1,
0\le k_1'<k_2'<\cdots<k'_r\le 2^m-1}
a_1(t^m_{k_1})a_1(t^m_{k_1'})
\cdots a_r(t^m_{k_r})a_r(t^m_{k_r'})\\
&\qquad \times R([t^m_{k_1},t^m_{k_1+1}]\times
[t^m_{k'_1},t^m_{k_1'+1}])
\cdots
R([t^m_{k_r},t^m_{k_r+1}]\times
[t^m_{k'_r},t^m_{k_r'+1}])\\
   &=F_r(t,t).
  \end{align*}
Let us choose a strictly decreasing sequence
$
 H^->H_1>\cdots>H_r>1/3.
$
By induction, we prove
\begin{align}
 V_{(2H_l)^{-1}}(F_l ; [0,t]^2_{\mathcal{D}_m})
 \le C(\{H_k\}_{k=1}^l)\prod_{i=1}^l\left(\|a_i\|_{(H^{-})^{-1}\hyp var}+
\|a_i\|_{\infty}\right)^2,\quad 1\le l\le r.
\end{align}
For consideration of the case $l=1$,
by application of Lemma~\ref{cor to towghi} and
 Lemma~\ref{p-variation norm for product} (1),
  we have
  \begin{align}
   V_{(2H_1)^{-1}}(F_1; [0,t]^2_{\mathcal{D}_m})\le
   C(H_1)\left(\|a_1\|_{\infty}+\|a_1\|_{(H^-)^{-1}\hyp var}\right)^2
   V_{(2H)^{-1}}(R; [0,t]^2),
  \end{align}
  where $C(H_1)$ is independent of $m$.
  Suppose the assertion holds up to $l-1$.
  By the definition of $F_l$ and the assumption of
  induction, by applying
Lemma~\ref{cor to towghi} and
  Lemma~\ref{p-variation norm for product} (2),
we obtain
  \begin{align}
   V_{(2H_l)^{-1}}(F_l; [0,t]^2_{\mathcal{D}_m})
   &\le C(\{H_k\}_{k=1}^l)\left(\prod_{i=1}^{l}
   \left(\|a_i\|_{1/H^{-}\hyp var}+\|a_i\|_{\infty}\right)^2
   \right)
   V_{(2H)^{-1}}(R; [0,t]^2),
  \end{align}
  which completes the proof of induction.
  We return to the equation (\ref{representation of AtfN}).

  Note that
\begin{align}
 \|f_N\|_{\mathcal{H}^r}^2=\sum_{1\le i_1,\ldots,i_r\le N}
 |c_{i_1,\ldots,i_r}|^2
 =E\left[\left(\sum_{1\le i_1,\ldots,i_r\le N}
   c_{i_1,\ldots,i_r}Z^1_{i_1}\cdots Z^r_{i_r}\right)^2\right].
\end{align}
Therefore, the above and the Schwarz inequality imply (\ref{estimate of A}).

Finally, we prove (\ref{consistent}).
Let $g_k=\sum_{i=1}^{\infty}\alpha^i_kh_i$ be the orthogonal expansion
of $g_k$ $(1\le k\le r)$.
Then, by the continuity property of the Young integral and the definition of
$\mathscr{A}$, we have
\begin{align*}
\mathscr{A}_{a_1,\ldots,a_r}[g_1,\ldots,g_r] 
&=
\mathscr{A}_{a_1,\ldots,a_r}
\left[\sum_{i=1}^{\infty}\alpha^i_1h_i,\ldots, \sum_{i=1}^\infty\alpha^i_kh_i\right]\\
&=\lim_{N\to\infty}\sum_{1\le i_1,\ldots,i_r\le N}
\alpha^{i_1}_1\cdots\alpha^{i_r}_r
\mathscr{A}_{a_1,\ldots,a_r}[h_{i_1},\ldots,h_{i_r}]
\nonumber\\
&=\mathscr{A}[g_1\otimes\cdots\otimes g_r](t).
\end{align*}
This completes the proof.
 \end{proof}

\begin{remark}\label{Hilbert-Schmidt property}
	The lemma presented above shows that, for any $a_1,\ldots,a_r$ with
$\|a_i\|_{(H^-)^{-1}}<\infty$ $(1\le i\le r)$,
there exists a unique $\Xi_{a_1,\ldots,a_r}\in \mathcal{H}^{\otimes r}$
such that
\begin{align}
  \mathscr{A}_{a_1,\ldots,a_r}[h_1,\ldots,h_r]&=
\left(\Xi_{a_1,\ldots,a_r},h_1\otimes\cdots\otimes h_r\right)_{\mathcal{H}^{\otimes r}},\\
\|\Xi_{a_1,\ldots,a_r}\|_{\mathcal{H}^{\otimes r}}&\le
C\prod_{i=1}^{l}
   \left(\|a_i\|_{1/H^{-}\hyp var}+\|a_i\|_{\infty}\right).
 \end{align}
\end{remark}

Next, we prove the higher order Malliavin differentiability of
$(F_t)\in \mathcal{I}(\RR)$.

\begin{theorem}\label{Malliavin differentiability}
 Let $(F_t)\in \mathcal{I}(\RR)$.
Then $F_t\in \mathbb{D}^{\infty}(\RR)$ and
there exists a random variable $G\in L^{\infty-}(\Omega)$ which depends only on
 $r$ and $F$ such that, for all $t$,
 \begin{align}
  \|(D^rF_t)(\omega)\|_{\mathcal{H}^{\otimes r}}\le
  G(\omega).
 \end{align}
\end{theorem}

  \begin{proof}
By Lemma~\ref{derivative of F}, Lemma~\ref{iterated integral for CM}, and
Remark~\ref{Hilbert-Schmidt property}, we see that
there exists $\Xi_{r,F_t}\in L^{\infty-}(\Omega,(\mathcal{H}^d)^{\otimes r})$ such that
for any $h_i\in \mathcal{H}$ and $v_i\in\RR^d$ $(1\le i\le r)$, it holds that
\begin{align}
 D_{h_1v_1,\ldots,h_rv_r}F_t=\left(\Xi_{r,F_t},h_1v_1\otimes\cdots\otimes h_rv_r
\right)_{(\mathcal{H}^d)^{\otimes r}}.\label{XiFt}
\end{align}
Let $h\in \mathcal{H}^d$.
By the definition of $D_h$, it is easy to see that for $F\in \mathcal{D}(D_h,\RR)$
and smooth cylindrical function $G$,
$
 E[(D_hF)G]=E[F(D_h)^{\ast}G]
$
holds,
where $(D_h)^{\ast}G=-D_hG+(h,w)G$ and
$(h,w)$ denotes the Wiener integral.
Successively applying this integration by parts formula, one obtains
\begin{align}
 E\left[D_{h_1v_1,\ldots,h_rv_r}F_t\,G\right]&=
E\left[F_t(D_{h_1v_1})^{\ast}\cdots (D_{h_rv_r})^{\ast}G\right]\label{integration by parts}\\
&=E\left[F_t(D^r)^{\ast}(G h_1v_1\odot\cdots \odot h_rv_r)\right].\nonumber
\end{align}
Combining (\ref{XiFt}) and (\ref{integration by parts}),
we get $\Xi_{r,F_t}\in (\mathcal{H}^d)^{\odot r}$ almost surely.
This shows that $E[F_t(D^r)^{\ast}(G h_1v_1\odot\cdots \odot h_rv_r)]=
E[\left(\Xi_{r,F_t},Gh_1v_1\odot\cdots \odot h_rv_r\right)_{(\mathcal{H}^d)^{\odot r}}]$.
This implies that $D^rF_t=\Xi_{r,F_t}$ in weak sense.
By the alternative definition of
Sobolev spaces in \cite{shigekawa}
(see Section~4.2.7), we see that
$F_t\in \mathbb{D}^{\infty}(\RR)$ and
$D^rF_t=\Xi_{r,F_{t}}$, which
completes the proof.
\end{proof}

\section{Moment estimates of weighted sum processes of Wiener chaos
of order 2} \label{weighted hermite variation}

In this section, $(B_t)$ stands for the $d$-dimensional fBm
with the Hurst parameter $\frac{1}{3}<H\leq\frac{1}{2}$
and we show Theorem~\ref{moment estimate}.
This discussion begins with the following proposition on $\goodClass(\RR)$.
Other examples of $\goodClass(\RR)$ are presented in Remark~\ref{Remark on summands}.
\begin{proposition}\label{prop490801}
	We have $\spanIteratedIntegrals(\RR)\subset \goodClass(\RR)$.
\end{proposition}

\begin{proof}
	Let $(F_t)\in \spanIteratedIntegrals(\RR)$ and fix $t\in[0,1]$.
	From Theorem~\ref{Malliavin differentiability}, we have $F_t\in \mathbb{D}^{\infty}(\RR)$.
	It deduced from estimates of the rough integrals
	that $\sup_{t\in [0,1]}|F_t|\in L^{\infty-}$ holds.
	From Lemma~\ref{derivative of F},
	we see that the right-hand side of \eqref{eq4i319014}
	is expressed as a summation of the form
	\begin{align*}
		G(t)
		\mathscr{A}_{a_1,\ldots,a_r}
			[h_1,\ldots,h_r](t),
	\end{align*}
	where $G(t), a_1,\dots,a_r\in \spanIteratedIntegrals(\RR)$,
	and $h_1,\dots,h_r$ are a permutation 
	of $\psi_{u_1},\ldots,\psi_{u_r}$.

	We consider $\mathscr{A}_{a_1,\ldots,a_r} [\psi_{u_1},\ldots,\psi_{u_r}](t)$ as an example.
	From Lemma~\ref{lem490313333}, we see
	\begin{align*}
		\max_{0\le t\le 1}
			\left|
				\mathscr{A}_{a_1,\ldots,a_r}
					[\psi_{s_1,t_1},\ldots,\psi_{s_r,t_r}](t)
			\right|
		&\le
			C
			\prod_{i=1}^r(t_i-s_i)^{2H},
	\end{align*}
	where $C\in L^{\infty-}$.
	The proof is completed.
\end{proof}

\begin{remark}\label{Remark on summands}
	There are more examples of elements of $\goodClass(\RR)$.
	\begin{enumerate}
		\item	Let $\theta : [0,1]\to [0,1]$ be a Borel measurable mapping
				and let $(F_t)\in\spanIteratedIntegrals(\RR)$.
				Then $\{F_{\theta(t)}\}_{t\in[0,1]}\in \goodClass(\RR)$.
				We do not require regularity of the time variable for $\goodClass(\RR)$.
		\item	Let $F\in \mathcal{I}(\RR^N), G\in \spanIteratedIntegrals(\RR)$.
 Let $K(t,x)$ be a real-valued measurable function on $[0,1]\times \RR^N$.
Moreover, we assume that the function $x\mapsto K(t,x)$ is smooth and
$K$ itself and its all derivatives satisfy polynomial growth condition 
uniformly in $t\in [0,1]$.
				Then we see Volterra integral type processes
				$\int_0^tK(t,F_s)dG_s$
				and 
				$\int_0^1K(t,F_s)dG_s$
				belong to $\goodClass(\RR)$.
The proof is similar to the case of iterated integrals.
We give a sketch of the proof.
Because $G$ is a linear combination of the iterated integrals,
it is sufficient to consider the integral 
$I(t)=\int_0^t\varphi(F_s)dB^{\alpha}_s$, where
$\varphi$ is a smooth function on $\RR^N$ which
satisfies similar polynomial growth condition and
$F=(F^k)_{k=1}^N$ $(F^k\in \mathcal{I}(\RR), 1\le k\le N)$.
Because
$I(t)=\lim_{|\mathcal{P}|\to 0}I^{\mathcal{P}}(t)$,
where
\[
 I^{\mathcal{P}}(t)=\sum_{i=1}^n\varphi(F_{t_{i-1}})B^{\alpha}_{t_{i-1},t_i}+
\sum_{i=1}^n\sum_{k}\sum_{\beta}(\partial_{x_k}\varphi)(F_{t_{i-1}})
(F^k)^{(\beta)}_{t_{i-1}}B^{\beta,\alpha}_{t_{i-1},t_i},
\]
by Lemma~\ref{gateaux derivative}, it is sufficient to prove that
$\lim_{|\mathcal{P}|\to 0}D_{hv}I^{\mathcal{P}}(t)$ converges in 
$L^{\infty-}$ to show $I(t)\in \mathcal{D}(D_{hv},\RR)$.
This convergence can be checked by noting 
\[
 (D_{hv}\varphi(F_{t_{i-1}}))^{(\beta)}=
D_{hv}\Bigl(\sum_{k}\sum_{\beta}(\partial_{x_k}\varphi)(F_{t_{i-1}})
(F^k)^{(\beta)}_{t_{i-1}}\Bigr)
\]
as in the proof
of Lemma~\ref{derivative of F1}.
After establishing $I(t)\in \mathcal{D}(D_h,\RR)$
and
\begin{align*}
D_{hv}I(t)=
\int_0^t(\partial\varphi)(F_s)[D_{hv}F_s]
dB^{\alpha}_s+v^{\alpha}\int_0^t\varphi(F_s)dh_s,	
\end{align*}
one can obtain higher order differentiability of $I(t)$
by using induction
argument, which also shows the representation of the derivatives of $I(t)$
as in Lemma~\ref{derivative of F}.
This implies the desired result.
	\end{enumerate}
\end{remark}

Next we consider 
$B^{\alpha,\beta}_{s,t}$ and $B^{\alpha}_{s,t}B^{\beta}_{s,t}$.
Let $\alpha\neq\beta$ and
$0\leq s<t\leq 1$.
We consider finite dimensional approximation of
$B^{\alpha,\beta}_{s,t}$ as in Section~\ref{df app of iterated integrals}
using the equipartition of
$[s,t]$.
That is, we define
\begin{align}
 \tilde{B}^{\alpha,\beta}_{s,t}(n)&=\sum_{k=1}^{n}
B^{\alpha}_{s,s+\frac{k-1}{n}(t-s)}
B^{\beta}_{s+\frac{k-1}{n}(t-s),s+\frac{k}{n}(t-s)},\label{tBabstn}\\
\tpsi^{\alpha,\beta}_{s,t}(n)&=\sum_{k=1}^n
\psi^{\alpha}_{s,s+\frac{k-1}{n}(t-s)}\odot
\psi^{\beta}_{s+\frac{k-1}{n}(t-s),s+\frac{k}{n}(t-s)}.\label{tpsin}
\end{align}
Then
$
 \tilde{B}^{\alpha,\beta}_{s,t}(n)
=I_2\left(\tpsi^{\alpha,\beta}_{s,t}(n)\right)
$
and
\begin{align}
	\lim_{n\to\infty}\tilde{B}^{\alpha,\beta}_{s,t}(n)
	&=
		B^{\alpha,\beta}_{s,t}
	\quad 
	\text{in $L^p$ for all $p\ge 1$},\\
	\label{lim tpsin}
	\lim_{n\to\infty}\tpsi^{\alpha,\beta}_{s,t}(n)
	&
	=
		\tpsi^{\alpha,\beta}_{s,t}
	\quad
	\text{in $\mathcal{H}^{\odot 2}$}.
\end{align}

Let $1\le \alpha(\ne)\beta\le d$.
  Let $p(\ge 2)$ be a positive integer.
Let $0\le s_i<t_i\le 1$ $(1\le i\le p)$.
We calculate the Wiener chaos expansion of
$
	\prod_{i=1}^p
		B^{\alpha}_{s_i,t_i}B^{\beta}_{s_i,t_i}
$
and
$
	\prod_{i=1}^p
		\tilde{B}^{\alpha,\beta}_{s_i,t_i}(n)
$.
To this end, we introduce several notations.
First, we introduce an $\mathcal{H}^{\odot 2r}$-valued
$2p$ variables function
$g^{p,r}(u_1,\ldots,u_p,v_1,\ldots,v_p)$ $(u_i, v_j\in [0,1])$
$(1\le i,j\le p,\,\, 0\le r\le p)$ by
\begin{align*}
& g^{p,r}(u_1,\ldots,u_p,v_1\ldots,v_p)\nonumber\\
&\qquad=
\sum_{(a,b)\in S_r}\sum_{\{\{I_i\}_{i=1}^a, K\}, \{\{J_j\}_{j=1}^b, L\}}
\prod_{i=1}^aR(u_{I_i^{-}},u_{I_i^{+}})\prod_{j=1}^b
R(v_{J^{-}_j},v_{J^{+}_j})
\left(\underset{k\in K}{\odot}
\psi^{\alpha}_{u_k}\odot\underset{l\in L}{\odot}\psi^{\beta}_{v_l}
\right),
\end{align*}
where 
$S_r=\{(a,b)~|~a+b=p-r, 2a\leq p, 2b\leq p, a,b\in \ZZ_{\ge 0}\}$ and
$\{\{I_i\}_{i=1}^a, K\}, \{\{J_j\}_{j=1}^b, L\}$ move in the set of
the disjoint partition of $\{1,\ldots,p\}$ satisfying the following
rule:
\begin{itemize}
	\item	$
				\big(\cup_{i=1}^aI_i\big)\cup K=
				\big(\cup_{j=1}^bJ_j\big)\cup L=\{1,\ldots,p\},
			$
	\item	$I_i$ and $J_j$ $(1\le i\le a, 1\le j\le b)$ consists of two distinct
			elements of $\{1,\ldots,p\}$.
\end{itemize}
Here we denoted the smaller number
and the larger number in $I_i$ by
$I_i^-$ and $I_i^+$, respectively and so on.
Additionally, we used the convention that $\prod_{i=1}^aR(u_{I_i^{-}},u_{I_i^{+}})=1$
when $a=0$ and so on.
For example, 
$g^{p,p}(u_1,\ldots,u_p,v_1,\ldots,v_p)=\odot_{i=1}^p\psi^{\alpha}_{u_i}\odot
\odot_{j=1}^p\psi^{\beta}_{v_j}$.
Here we give more concrete examples in the case $p=6$ and $r=2$.
In this case $g^{6,2}(u_1,\dots,u_6,v_1,\dots,v_6)$ contains terms
\begin{align}
	\label{eqExample001}
	&
	R(u_1,u_2)R(u_3,u_4)R(u_5,u_6)
	\cdot
	R(v_3,v_6)\,\,
	\psi^{\beta}_{v_1}
	\odot
	\psi^{\beta}_{v_2}
	\odot
	\psi^{\beta}_{v_4}
	\odot
	\psi^{\beta}_{v_5},\\
	\label{eqExample002}
	&
	R(u_1,u_2)R(u_3,u_4)
	\cdot
	R(v_1,v_4)R(v_3,v_6)\,\,
	\psi^{\alpha}_{u_5}
	\odot
	\psi^{\alpha}_{u_6}
	\odot
	\psi^{\beta}_{v_2}
	\odot
	\psi^{\beta}_{v_5},\\
	\label{eqExample003}
	&
	R(u_1,u_2)
	\cdot
	R(v_1,v_4)
	R(v_2,v_5)
	R(v_3,v_6)\,\,
	\psi^{\alpha}_{u_3}
	\odot
	\psi^{\alpha}_{u_4}
	\odot
	\psi^{\alpha}_{u_5}
	\odot
	\psi^{\alpha}_{u_6}.
\end{align}

Using $g^{p,r}$, we define $\mathcal{H}^{\odot 2r}$-valued
functions
$f(n)^{p,r}_{(s_1,t_1),\ldots,(s_p,t_p)}$ by
\begin{align*}
 f(n)_{(s_1,t_1),\ldots,(s_p,t_p)}^{p,r}
=\sum_{l_1,\ldots,l_p=1}^n
g^{p,r}([s_1,t^1_{l_1-1}]\times\cdots\times [s_p,t^p_{l_{p-1}}]
\times [t^1_{l_1-1},t^1_{l_1}]\times\cdots\times [t^p_{l_p-1},t^p_{l_p}]),
\end{align*}
where 
$t^i_l=s_i+\frac{l}{n}(t_i-s_i)$ $(0\le l\le n)$.
Now, we are ready to state the expansion formula for
the products of the quadratic Wiener functionals.

\begin{lemma}\label{expansion formula of product}
Let $1\le \alpha(\ne)\beta\le d$.
  Let $p(\ge 2)$ be a positive integer.
Let $0\le s_i<t_i\le 1$ $(1\le i\le p)$.
Then we have
\begin{align}
	\prod_{i=1}^p
		B^{\alpha}_{s_i,t_i}B^{\beta}_{s_i,t_i}
	&=
		\sum_{r=0}^p
			I_{2r}\left(g^{p,r}(
[s_1,t_1]\times\cdots\times[s_p,t_p]\times [s_1,t_1]\times\cdots\times[s_p,t_p])
\right),\label{expansion formula of product 1}\\
\prod_{i=1}^p
		\tilde{B}^{\alpha,\beta}_{s_i,t_i}(n)
	&=
		\sum_{r=0}^p
			I_{2r}(f(n)_{(s_1,t_1),\ldots,(s_p,t_p)}^{p,r}).
\label{expansion formula of product 2}
\end{align}

\end{lemma}

\begin{proof}
It is sufficient to prove the following identity.
For any $0\le u_i,v_j\le 1$ $(1\le i,j\le p)$, 
it holds that
\begin{align}
 I_2(\psi^{\alpha}_{u_1}\odot\psi^{\beta}_{v_1})\cdots 
I_2(\psi^{\alpha}_{u_p}\odot \psi^{\beta}_{v_p})&=
\sum_{r=0}^pI_{2r}(g^{p,r}(u_1,\ldots,u_p,v_1,\ldots,v_p)).\label{gpr 2}
\end{align}
In fact, (\ref{expansion formula of product 1}) clearly follows from this identity.
Let us consider (\ref{expansion formula of product 2}).
Once this identity has been proved, then using the linearity
of the mapping $I_{2r}$, we obtain
\begin{align*}
\prod_{i=1}^pI_2(\psi^{\alpha}_{u_i,u'_i}\odot\psi^{\beta}_{v_i,v'_i})
&=\sum_{r=0}^pI_{2r}\left(g^{p,r}([u_1,u'_1]
\times\cdots\times[u_p,u'_p]\times[v_1,v'_1]\times\cdots\times 
[v_p,v'_p])\right)
\end{align*}
and
\begin{multline*}
	\prod_{i=1}^p I_{2}
		\Big(
			\psi^{\alpha}_{s_i,t^{i}_{l_i}-1}
			\odot
			\psi^{\beta}_{t^i_{l_i-1},t^i_{l_i}}
		\Big)\\
	=
		\sum_{r=0}^p
			I_{2r}
				\Big(
					g^{p,r}
						\Big(
							[s_1,t^1_{l_1-1}]\times\cdots \times [s_p,t^p_{l_p-1}]
							\times
							[t^1_{l_1-1},t^1_{l_1}]\times\cdots\times[t^p_{l_p-1},t^p_{l_p}]
						\Big)
				\Big),
\end{multline*}
which implies the desired identity.
We consider the case where $p=2$ of (\ref{gpr 2}).
Because $(\psi^{\alpha}_{u},\psi^{\beta}_v)_{\mathcal{H}^d}=0$,
we have
\begin{multline*}
	I_2(\psi^{\alpha}_{u_1}\odot\psi^{\beta}_{v_1})
	I_2(\psi^{\alpha}_{u_2}\odot \psi^{\beta}_{v_2})
	=
		(\psi^{\alpha}_{u_1},\psi^{\alpha}_{u_2})
		(\psi^{\beta}_{v_1},\psi^{\beta}_{v_2})
		+
		(\psi^{\alpha}_{u_1},\psi^{\alpha}_{u_2})
		I_2\left(\psi^{\beta}_{v_1}\odot\psi^{\beta}_{v_2}\right)\\
		+
		(\psi^{\beta}_{v_1},\psi^{\beta}_{v_2})
		I_2\left(\psi^{\alpha}_{u_1}\odot\psi^{\alpha}_{u_2}\right)
		+
		I_4(\psi^{\alpha_{u_1}}\odot\psi^{\beta}_{v_1}\odot
		\psi^{\alpha}_{u_2}\odot \psi^{\beta}_{v_2}),
\end{multline*}
which shows that the identity (\ref{gpr 2}) holds in the case of $p=2$.
Suppose (\ref{gpr 2}) holds for $p$.
Then 
\begin{align*}
	\prod_{i=1}^{p+1}I_2(\psi^{\alpha}_{u_i}\odot\psi^{\beta}_{v_i})
	&=
		\sum_{r=0}^p
			I_{2r}(g^{p,r})
			I_2(\psi^{\alpha}_{u_{p+1}}\odot\psi^{\beta}_{v_{p+1}})\\
	&=
		\sum_{\{\{I_i\}_{i=1}^a, K\}, \{\{J_j\}_{j=1}^b, L\}}
			\prod_{i=1}^a
				R(u_{I_i^{-}},u_{I_i^{+}})
			\prod_{j=1}^b
				R(v_{J^{-}_j},v_{J^{+}_j})\\
	&\qquad\qquad\qquad\qquad
			\times
			I_{2r}
				\left(
					\underset{k\in K}{\odot}\psi^{\alpha}_{u_k}
					\odot
					\underset{l\in L}{\odot}\psi^{\beta}_{v_l}
				\right)
			I_2(\psi^{\alpha}_{u_{p+1}}\odot\psi^{\beta}_{v_{p+1}}).
\end{align*}
By the formula in Proposition~\ref{product formula} (3),
we obtain
\begin{multline*}
	I_{2r}\left(
	\underset{k\in K}{\odot}
	\psi^{\alpha}_{u_k}\odot\underset{l\in L}{\odot}\psi^{\beta}_{v_l}\right)
	I_2(\psi^{\alpha}_{u_{p+1}}\odot\psi^{\beta}_{v_{p+1}})
	=
		I_{2r+2}\left(\underset{k\in K}{\odot}\psi^{\alpha}_{u_k}\odot
		\underset{l\in L}{\odot}\psi^{\beta}_{v_l}\odot \psi^{\alpha}_{u_{p+1}}\odot
	\psi^{\beta}_{v_{p+1}}\right)\\
	\begin{aligned}
		&
		+	
		\sum_{k'\in K}R(u_{k'},u_{p+1})I_{2r}
			\left(\underset{k\in K\setminus \{k'\}}{\odot}
			\psi^{\alpha}_{u_k}\odot
			\underset{l\in L}{\odot}\psi^{\beta}_{v_l}\odot 
			\psi^{\beta}_{v_{p+1}}\right)\\
		&
		+
		\sum_{l'\in L}R(v_{l'},v_{p+1})I_{2r}
			\left(\underset{k\in K}{\odot}
			\psi^{\alpha}_{u_k}\odot \psi^{\alpha}_{u_{p+1}}
			\odot\underset{l\in L\setminus\{l'\}}{\odot}\psi^{\beta}_{v_l}\right)\\
		&+
			\sum_{k'\in K, l'\in L}R(u_{k'},u_{p+1})
				R(v_{l'},v_{p+1})I_{2r-2}
				\left(\underset{k\in K\setminus \{k'\}}{\odot}
				\psi^{\alpha}_{u_k}\odot
				\underset{l\in L\setminus \{l'\}}{\odot}
				\psi^{\beta}_{v_l}\right).
	\end{aligned}
\end{multline*}
From these two identities above, we see that the case of $p+1$ holds.
This completes the proof.
\end{proof}

\begin{lemma}\label{estimate of sum of rho}
 Let $p$ and $q$ be positive integers with $p\ge 2$ and $q\ge 1$.
Suppose that we are given non-negative integers
$\{a(\{i,j\})\}$ for each pair $\{i,j\}$ $(i\ne j, 1\le i,j\le p)$
satisfying that $\sum_{1\le j\le p, j\ne i}a(\{i,j\})\le q$
for any $i$.
We write $\sum_{\{i,j\}}a(\{i,j\})=N$.
Let $\{\rho(n)\}_{l=0}^{\infty}$ be a sequence of non-negative numbers
with $0\le \rho(n)\le 1$ for all $n$ and $\sum_{n=0}^{\infty}\rho(n)\le C$,
where $C\ge 1$.
Then,
for any $0\le s<t\le 1$, we have
\begin{align}
	\sum_{k_1,\dots,k_p=\floor{2^m s}+1}^{\floor{2^mt}}
		\prod_{\{i,j\}}\rho(|k_i-k_j|)^{a(\{i,j\})}
	\le
		C^N
		(\floor{2^mt}-\floor{2^ms})^{p-\ceiling{\frac{N}{q}}}.
\label{estimate of sum of rho 2}
\end{align}
\end{lemma}

\begin{remark}
We have used the notation 
$\ceiling{x}=\min\{n\in \mathbb{Z}~|~n\ge x\}$.
It is clear that $2N\le pq$ holds.
Therefore $p-\ceiling{\frac{N}{q}}>0$.
\end{remark}

\begin{proof}
We denote the quantity on the left-hand side of 
(\ref{estimate of sum of rho 2})
by $I$.
 We prove this using induction on $p$.
Let $p=2$.
Write $a(\{1,2\})=q'$.
Then $N=q'$ and $q'\leq q$,
which imply $\ceiling{\frac{N}{q}}=1$ if $q'\ge 1$ and
$\ceiling{\frac{N}{q}}=0$ if $q'=0$.
When $q'=0$,
$I=(\floor{2^m t}-\floor{2^m s})^p$ and $N=0$ hold. Therefore, the inequality
clearly holds.
If $q'\ge 1$, then
\begin{multline*}
	I
	=
		\sum_{k_1,k_2=\floor{2^m s}+1}^{\floor{2^mt}}
			\rho(|k_1-k_2|)^{q'}
	\le
		C^{q'-1}
		\sum_{k_1,k_2=\floor{2^m s}+1}^{\floor{2^mt}}
			\rho(|k_1-k_2|)\\
	\le
		C^{N-1}
		\cdot 
		C(\floor{2^mt}-\floor{2^ms})
	=
		C^N
		(\floor{2^mt}-\floor{2^ms})^{p-1},
\end{multline*}
which proves the case $p=2$.

Suppose the case $p$ with any $q$ holds true and we prove
the case $p+1$ with any $q$.
We prove this by induction on $q$.
Let $q=1$.
In this case, there exist distinct natural numbers
$i_1,\ldots, i_{2N}$ such that
$a\left(\{i_{2l-1}, i_{2l}\}\right)=1$ $(1\le l\le N)$
and $a(\{i,j\})=0$ for other pairs $\{i,j\}$.
Consequently,
\begin{align*}
	I
	&=
		\sum_{k_1,\dots,k_p=\floor{2^m s}+1}^{\floor{2^mt}}
			\prod_{l=1}^N
				\rho(|k_{i_{2l-1}}-k_{i_{2l}}|)\\
	&=
		\sum_{\substack{\floor{2^ms}+1\le k_j\le \floor{2^mt},\\  1\le j\le p+1, j\ne i_1,\ldots,i_{2N}}}
			\left(
				\sum_{\substack{\floor{2^ms}+1\le k_{i_l}\le \floor{2^mt},\\ 1\le l\le 2N}}
					\prod_{l=1}^N
					\rho(|k_{i_{2l-1}}-k_{i_{2l}}|)
			\right)\\
	&\le
		\sum_{\substack{\floor{2^ms}+1\le k_j\le \floor{2^mt},\\  1\le j\le p+1, j\ne i_1,\ldots,i_{2N}}}
			C^N
			(\floor{2^mt}-\floor{2^ms})^{N}\\
	&\le
		(\floor{2^mt}-\floor{2^ms})^{p+1-N}C^N,
\end{align*}
which implies the case $q=1$ holds.
Suppose the case of $p+1$ until $q-1$ holds.
If $\sum_{j\ne i}a(\{i,j\})\le q-1$ for all $1\le i\le p+1$, then
by the assumption of the induction, we have the following desired estimate:
\begin{align*}
	\sum_{k_1,\dots,k_p=\floor{2^m s}+1}^{\floor{2^mt}}
		\prod_{\{i,j\}}\rho(|k_i-k_j|)^{a(\{i,j\})}
	&\le
		C^N
		\left(\floor{2^mt}-\floor{2^ms}\right)^{p+1-\ceiling{\frac{N}{q-1}}}\\
	&\le
		C^N\left(\floor{2^mt}-\floor{2^ms}\right)^{p+1-\ceiling{\frac{N}{q}}}.
\end{align*}
Suppose that there exists $i_0\in \{1,\ldots,p+1\}$ such that
$\sum_{j\ne i_0}a(\{i_0,j\})=q$.
Let 
$J_0=\{j~|~a(\{i_0,j\})\ge 1\}\subset \{1,\ldots,p+1\}$.
\begin{align*}
	I
	&=
		\sum_{\substack{\floor{2^ms}+1\le k_j\le \floor{2^mt}\\ 1\le j\le p+1, j\ne i_0}}
			\left(
				\prod_{\{l,l'\}\,\text{with}\, 1\le l,l' (\ne i_0)\le p+1}
					\rho(|k_l-k_{l'}|)^{a(\{l,l'\})}
			\right)\\
	&\qquad\qquad\qquad\qquad\qquad\qquad
			\times
			\left(
				\sum_{\floor{2^ms}+1\le k_{i_0}\le \floor{2^mt}}
					\prod_{j\in J_0}\rho(|k_{i_0}-k_j|)^{a(\{i_0,j\})}
			\right).
\end{align*}
We choose $j_0\in J_0$.
Because $0\le \rho(n)\le 1$, we have
\begin{align*}
 \sum_{\floor{2^ms}+1\le k_{i_0}\le \floor{2^mt}}
\prod_{j\in J_0}\rho(|k_{i_0}-k_j|)^{a(\{i_0,j\})}&\le
\sum_{\floor{2^ms}+1\le k_{i_0}\le \floor{2^mt}}
\rho(|k_{i_0}-k_{j_0}|)\le C.
\end{align*}
Therefore, we obtain
\begin{align*}
 I&\le
C \sum_{\substack{\floor{2^ms}+1\le k_l\le \floor{2^mt}\\
1\le l\le p+1, l\ne i_0
}}\prod_{\{l,l'\}\,\text{with}\, 1\le l,l' (\ne i_0)\le p+1}
\rho(|k_l-k_{l'}|)^{a(\{l,l'\})}=:C I'.
\end{align*}
Note that in the sum of $I'$, $l$ moves in the set
$\{1,\ldots,p+1\}\setminus \{i_0\}$, for which cardinality is $p$.
Therefore, applying the assumption of the induction to the term $I'$,
we get
\begin{align*}
 I'&\le C^{N-q}
\left(\floor{2^mt}-\floor{2^ms}\right)^{p-\ceiling{\frac{N-q}{q}}}
\le C^{N-q}
\left(\floor{2^mt}-\floor{2^ms}\right)^{p+1-\ceiling{\frac{N}{q}}}.
\end{align*}
Consequently, we have
\begin{align*}
 I&\le C^{N-q+1}
\left(\floor{2^mt}-\floor{2^ms}\right)^{p+1-\ceiling{\frac{N}{q}}}
\le C^N\left(\floor{2^mt}-\floor{2^ms}\right)^{p+1-\ceiling{\frac{N}{q}}},
\end{align*}
which proves the case $p+1$ holds.
\end{proof}

\begin{lemma}\label{estimate of integration by parts}
 Let $(F_t)\in \goodClass(\RR)$.
Let $M$ be a natural number and $0\le r\le p$.
\begin{enumerate}
	\item\label{estimate of integration by parts 1}
			Let $0\le k_1,\ldots,k_p\le 2^m$.
			Then it holds that 
			\begin{multline*}
				\sup_{0\le t_1,\ldots,t_M\le 1}
					\Big|
						D^{2r}(F_{t_1}\cdots F_{t_{M}})
							\Big[
								f(n)^{p,r}_{(\tau^m_{k_1-1},\tau^m_{k_1}),\ldots,(\tau^m_{k_p-1},\tau^m_{k_p})}
							\Big]
					\Big|\\
				+
					\sup_{0\le t_1,\ldots,t_M\le 1}
						\Big|
							D^{2r}(F_{t_1}\cdots F_{t_{M}})
								\Big[
									g^{p,r}
										\Big(
											[\tau^m_{k_1-1},\tau^m_{k_1}]\times\cdots \times [\tau^m_{k_p-1},\tau^m_{k_p}]
										\Big)
								\Big]
						\Big|\\
				\le
					C_{p,r,M}(B)(2^{-m})^{2H(p+r)}
					\sum_{(a,b)\in S_r}
					\sum_{\{I_i\}_{i=1}^a, \{J_j\}_{j=1}^b}
						\prod_{i=1}^a \rho_H(|k_{I_i^{+}}-k_{I^{-}_i}|)
						\prod_{j=1}^b\rho_H(|k_{J_j^{+}}-k_{J^{-}_j}|),
		\end{multline*}
		where $S_r$, $\{I_i\}_{i=1}^a$ and $\{J_j\}_{j=1}^b$ 
		are the same ones in the definition of $g^{p,r}$.
		Also $C_{p,r,M}(B)$ is a random variable satisfying $E[C_{p,r,M}(B)^p]<\infty$ for all $p\geq 1$.
\item\label{estimate of integration by parts 2}
	We have
	\begin{multline*}
		(2^m)^{2pH-\frac{p}{2}}
		\Biggl\{
			\sum_{k_1,\ldots,k_{p}=\floor{2^ms}+1}^{\floor{2^mt}}
				\sup_{0\le t_1,\ldots,t_M\le 1}
				\Big|
					D^{2r}(F_{t_1}\cdots F_{t_{M}})
						\Big[
							f(n)^{p,r}_{(\tau^m_{k_1-1},\tau^m_{k_1}),\ldots,(\tau^m_{k_p-1},\tau^m_{k_p})}
						\Big]
				\Big|\\
			+
			\sum_{k_1,\ldots,k_{p}=\floor{2^ms}+1}^{\floor{2^mt}}
				\sup_{0\le t_1,\ldots,t_M\le 1}
				\Big|
					D^{2r}(F_{t_1}\cdots F_{t_{M}})
						\Big[
								g^{p,r}
									\Big(
										[\tau^m_{k_1-1},\tau^m_{k_1}]\times\cdots\times [\tau^m_{k_p-1},\tau^m_{k_p}]
									\Big)
						\Big]
				\Big|
		\Biggr\}\\
		\le
			C_{p,r,M}(B)
			\left(\frac{1}{2^m}\right)^{\left(\frac{r}{2}\right)(4H-1)}
			\left(\frac{\floor{2^mt}-\floor{2^ms}}{2^m}\right)^{\frac{p+r}{2}}.
	\end{multline*}
\end{enumerate}
\end{lemma}

 \begin{proof}
We prove assertion~(\ref{estimate of integration by parts 1}).
By definition of $\goodClass(\RR)$, we have
\begin{multline*}
	D^{2r}(F_{t_1}\cdots F_{t_{M}})
		[
			g^{p,r}(u_1,\dots,u_p,v_1,\dots,v_p)
		]
\\
	=
	\sum_{(a,b)\in S_r}
	\sum_{\{\{I_i\}_{i=1}^a,K\}, \{\{J_j\}_{j=1}^b,L\}}
		\prod_{i=1}^aR(u_{I_i^{-}},u_{I_i^{+}})
		\prod_{j=1}^bR(v_{J_j^-},v_{J_j^{+}})
		\phi_{K,L}(u_k, v_l; k\in K, l\in L),
\end{multline*}
where 
$\phi_{K,L}$ satisfies the estimate
\begin{align*}
 |\phi_{K,L}([u_k,u_k'], [v_l,v_l']; k\in K, l\in L)|\le
C(B)\prod_{k\in K}(u_k'-u_k)^{2H}\prod_{l\in L}(v_l'-v_l)^{2H}.
\end{align*}
In addition, from Lemma~\ref{properties of R} we have
\begin{align*}
& V_{(2H)^{-1}}(R ; [\tau^m_{k_{I^{-}_i}-1}, \tau^m_{k_{I^{-}_i}}]
\times [\tau^m_{k_{I^{+}_i}-1}, \tau^m_{k_{I^{+}_i}}])\le
(2^{-m})^{2H}|\rho_H(|k_{I^{+}_i}-k_{I^{-}_i}|)|,\\
&V_{(2H)^{-1}}(R ; [\tau^m_{k_{J^{-}_i}-1}, \tau^m_{k_{J^{-}_i}}]
\times [\tau^m_{k_{J^{+}_i}-1}, \tau^m_{k_{J^{+}_i}}])\le
(2^{-m})^{2H}|\rho_H(|k_{J^{+}_i}-k_{J^{-}_i}|)|.
\end{align*}
By these estimates, we have
\begin{multline*}
	\Big|
		D^{2r}(F_{t_1}\cdots F_{t_{M}})
			\Big[
				g^{p,r}
					\Big(
						[\tau^m_{k_1-1},\tau^m_{k_1}]\times\cdots \times [\tau^m_{k_p-1},\tau^m_{k_p}]
					\Big)
			\Big]
	\Big|
	\le
		C_{p,r,M}(B)\\
		\times
		\sum_{(a,b)\in S_r}
			\Biggl\{
				\left(2^{-m}\right)^{2H(a+b)+2H(|K|+|L|)}
				\sum_{\{I_i\}_{i=1}^a, \{J_j\}_{j=1}^b}
					\prod_{i=1}^a \rho_H(|k_{I_i^{+}}-k_{I^{-}_i}|)
					\prod_{j=1}^b\rho_H(|k_{J_j^{+}}-k_{J^{-}_j}|)
			\Biggr\}.
\end{multline*}
Noting that $a+b+|K|+|L|=p+r$,
we obtained the desired estimates for the term containing the derivative
in direction to $g^{p,r}$.
For example note
$
	D^{2\cdot 2}(F_{t_1}\cdots F_{t_{M}})
	[
		g^{6,2}(u_1,\dots,u_6,v_1,\dots,v_6)
	]
$
contains the following term
\begin{align*}
	R(u_1,u_2)R(u_3,u_4)
	\cdot
	R(v_1,v_4)R(v_3,v_6)
	\cdot
	\phi_{\{5,6\},\{2,5\}}(u_5,u_6,v_2,v_5),
\end{align*}
which corresponds \eqref{eqExample002}.

For the derivative in direction to $f(n)^{p,r}$, first, note that
\begin{multline*}
	D^{2r}(F_{t_1}\cdots F_{t_{M}})
		\Big[
			f(n)^{p,r}_{(\tau^m_{k_1-1},\tau^m_{k_1}),\ldots,(\tau^m_{k_p-1},\tau^m_{k_p})}
		\Big]\\
	=
		\sum_{l_1,\ldots,l_p=1}^n
			D^{2r}(F_{t_1}\cdots F_{t_{M}})
				\left[
					g^{p,r}
						\left(
							\prod_{r=1}^p
								\Big[\tau^m_{k_r-1},t^{k_r}_{l_r-1}\Big]
							\times
							\prod_{r=1}^p
								\Big[t^{k_r}_{l_r-1},t^{k_r}_{l_r}\Big]
						\right)
				\right],
\end{multline*}
where $t^k_l=\tau^m_{k-1}+\frac{l}{n}\frac{1}{2^m}$.
This is a discrete multidimensional Young integral 
on $[\tau^m_{k_1-1},\tau^m_{k_1}]\times\cdots
\times [\tau^m_{k_p-1},\tau^m_{k_p}]$
with the equipartition.
By applying Lemma~\ref{properties of R},
Lemma~\ref{p-variation norm for product} 
and Proposition~\ref{estimate of fgh},
we see that
\begin{align*}
	\sup_{0\le t_1,\ldots,t_M\le 1}
		\Big|
			D^{2r}(F_{t_1}\cdots F_{t_{M}})
				\Big[
					f(n)^{p,r}_{(\tau^m_{k_1-1},\tau^m_{k_1}),\ldots,(\tau^m_{k_p-1},\tau^m_{k_p})}
				\Big]
		\Big|
\end{align*}
is bounded from above by a similar bound.
Readers might be aided in understanding the expressions above
by knowing that the term corresponding to \eqref{eqExample002} in
$
	D^{2\cdot 2}(F_{t_1}\cdots F_{t_{M}})
	[
		f(n)^{6,2}_{(\tau^m_{k_1-1},\tau^m_{k_1}),\ldots,(\tau^m_{k_6-1},\tau^m_{k_6})}
	]
$
is given by
\begin{align*}
	&
	\sum_{l_1,\ldots,l_6=1}^n
		R
			\big(
				\big[\tau^m_{k_1-1},t^{k_1}_{l_1-1}\big]
				\times
				\big[\tau^m_{k_2-1},t^{k_2}_{l_2-1}\big]
			\big)
		R
			\big(
				\big[\tau^m_{k_3-1},t^{k_3}_{l_3-1}\big]
				\times
				\big[\tau^m_{k_4-1},t^{k_4}_{l_4-1}\big]
			\big)\\
	&
	\qquad
	\qquad
	\qquad
		\times
		\phi_{\{5,6\},\{2,5\}}
			\big(
				\big[\tau^m_{k_5-1},t^{k_5}_{l_5-1}\big]
				\times
				\big[\tau^m_{k_6-1},t^{k_6}_{l_6-1}\big]
				\times
				\big[t^{k_2}_{l_2-1},t^{k_2}_{l_2}\big]
				\times
				\big[t^{k_5}_{l_5-1},t^{k_5}_{l_5}\big]
			\big)\\
	&
	\qquad
	\qquad
	\qquad
		\times
		R
			\big(
				\big[t^{k_1}_{l_1-1},t^{k_1}_{l_1}\big]
				\times
				\big[t^{k_4}_{l_4-1},t^{k_4}_{l_4}\big]
			\big)
		R
			\big(
				\big[t^{k_3}_{l_3-1},t^{k_3}_{l_3}\big]
				\times
				\big[t^{k_6}_{l_6-1},t^{k_6}_{l_6}\big]
			\big).
\end{align*}
Therefore, we complete the proof of assertion~(\ref{estimate of integration by parts 1}).

We prove assertion~(\ref{estimate of integration by parts 2}).
Denote by $S$ the left-hand side of the desired inequality.
By assertion~(\ref{estimate of integration by parts 1}) and 
applying Lemma~\ref{estimate of sum of rho} to the case where $q=2$
and $N=a+b$,
\begin{align*}
  	S
		&\le C_{p,r,M}(B) (2^m)^{2pH-\frac{p}{2}}
		\sum_{(a,b)\in S_r}
		\left(\frac{1}{2^m}\right)^{2H(p+r)}C^{a+b}
		\left(\floor{2^mt}-\floor{2^ms}\right)^{p-\ceiling{\frac{a+b}{2}}}
		\\
		&\le C_{p,r,M}(B)'
		\left(\frac{1}{2^m}\right)^{\frac{p}{2}(4H-1)-2H(p-r)+\frac{p-r}{2}}
		\left(\frac{\floor{2^mt}-\floor{2^ms}}{2^m}\right)^{p-\frac{a+b}{2}}
		\\
		&\le C_{p,r,M}(B)'
		\left(\frac{1}{2^m}\right)^{\left(\frac{r}{2}\right)(4H-1)}
		\left(\frac{\floor{2^mt}-\floor{2^ms}}{2^m}\right)^{\frac{p+r}{2}}.
\end{align*}
Because $r\ge 0$ and $4H-1>0$,
this completes the proof of assertion~(\ref{estimate of integration by parts 2}).
 \end{proof}

We are in a position to prove Theorem~\ref{moment estimate} and
Corollary~\ref{cor to moment estimate}.

\begin{proof}[Proof of Theorem $\ref{moment estimate}$]
For simplicity, we omit writing $F$ in the notation of $I^m_{s,t}(F)$ and $\tilde{I}^m_{s,t}(F)$.
We give the estimate of the moment of $I^m_{s,t}$.
The proof of the moment estimate of $\tilde{I}^m_{s,t}$ is similar to it.
We have
\begin{align*}
 E\left[\left((2^m)^{2H-\frac{1}{2}}I^m_{s,t}\right)^p\right]
 =
(2^m)^{2pH-\frac{p}{2}}
\sum_{k_1,\ldots,k_p=\floor{2^ms}+1}^{\floor{2^mt}}
E\left[\left(\prod_{i=1}^pF_{\tau^m_{k_i-1}}\right)
\left(
\prod_{i=1}^pB^{\alpha,\beta}_{\tau^m_{k_i-1},\tau^m_{k_i}}\right)\right].
\end{align*}

Using the Riemann sum approximation (\ref{tBabstn}),
Lemma~\ref{expansion formula of product}
and the integration by parts formula, we have
\begin{align*}
 E\left[\left(\prod_{i=1}^pF_{\tau^m_{k_i-1}}\right)
\left(
\prod_{i=1}^pB^{\alpha,\beta}_{\tau^m_{k_i-1},\tau^m_{k_i}}\right)\right]
&=
\lim_{n\to\infty}
E\left[\left(\prod_{i=1}^pF_{\tau^m_{k_i-1}}\right)
\left(\prod_{i=1}^p
\tilde{B}^{\alpha,\beta}_{\tau^m_{k_i-1},\tau^m_{k_i}}(n)\right)
\right]\\
&=
\lim_{n\to\infty}
\sum_{r=0}^{p}E\left[D^{2r}\left(\prod_{i=1}^pF_{\tau^m_{k_i-1}}\right)
\left[f(n)^{p,r}_{(\tau^m_{k_1-1},\tau^m_{k_1}),\ldots,
(\tau^m_{k_p-1},\tau^m_{k_p})}\right]\right].
\end{align*}
By Lemma~\ref{estimate of integration by parts},
the following estimate holds independently of $n$:
\begin{multline*}
	(2^m)^{2pH-\frac{p}{2}}
	\sum_{k_1,\ldots,k_p=\floor{2^ms}+1}^{\floor{2^mt}}
	\sum_{r=0}^{p}
		\left|
			E
				\left[
						D^{2r}\left(\prod_{i=1}^pF_{\tau^m_{k_{i}}}\right)
						\left[
							f(n)^{p,r}_{(\tau^m_{k_1-1},\tau^m_{k_1}),\ldots,(\tau^m_{k_p-1},\tau^m_{k_p})}
						\right]
				\right]
		\right|\\
	\le
		C_p
		\left(\frac{\floor{2^mt}-\floor{2^ms}}{2^m}\right)^{\frac{p}{2}}.
\end{multline*}
This completes the proof.
\end{proof}

\begin{proof}[Proof of Corollary~$\ref{cor to moment estimate}$]
	We will use an argument similar to those found in \cite{liu-tindel,aida-naganuma2023approach}.
	We will show the assertion for $I^m(F)$ only.
	In this proof, $C$ denotes a positive constant independent of $m$ and may change line by line.
	We consider the piecewise linear extension of $\{I^m_t(F)\}_{t\in\Dm}$
	and denote it by the same symbol.
	Set
	\begin{align*}
		G_{m,\theta}
		=
			\max_{s,t\in[0,1], s<t}
				\frac{|(2^m)^{2H-\frac{1}{2}}I^m_{s,t}(F)|}{|t-s|^\theta}.
	\end{align*}
	Let $p$ be a positive integer satisfying $p>1/(1-2\theta)$.
	Then the Garsia-Rodemich-Rumsey inequality (see \cite{nualart}) implies
	\begin{align*}
		|G_{m,\theta}|^{2p}
		\leq
			C
			\int_0^1
			\int_0^t
				\frac{|(2^m)^{2H-\frac{1}{2}}I^m_{s,t}(F)|^{2p}}{|t-s|^{2+2p\theta}}\,
				dsdt.
	\end{align*}
	From Theorem~\ref{moment estimate}, we have
	\begin{align*}
		E
			\big[
				\big\{
					(2^m)^{2H-\frac{1}{2}}
					I^m_{s,t}(F)
				\big\}^{2p}				
			\big]
	&\le
		3^{2p-1}
		C
		|t-s|^p.
	\end{align*}
	Therefore, we have
	\begin{align*}
		E[|G_{m,\alpha}|^{2p}]
		\leq
			2
			\int_0^1
			\int_0^t
				\frac{E\big[|(2^m)^{2H-\frac{1}{2}}I^m_{s,t}(F)|^{2p}\big]}{|t-s|^{2+2p\theta}}\,
				dsdt
		\leq
			C.
	\end{align*}
	Therefore, $\sup_{m\geq 1}\|G_{m,\theta}\|_{L^{2p}}<\infty$, which
	completes the proof.
\end{proof}

\begin{remark}\label{rem9410u09u13}
In our application \cite{aida-naganuma2023approach}, it is necessary to prove
\begin{align*}
	\|(2^m)^{2H-\frac{1}{2}}I^m(F)\|_{H^-}+
	\|(2^m)^{2H-\frac{1}{2}}\tilde{I}^m(F)\|_{H^-}<\infty,	
\end{align*}
where $\max\{\frac{1}{3},\frac{1}{2}(H+\frac{1}{4})\}<H^-<H$.
For this proof, we need the estimates in Theorem~\ref{moment estimate}
for large $p$.
When we apply our theorem to the case of $F\in \mathcal{I}(\RR^N)$,
this requires more differentiability of $\sigma, b$ 
than the assumption in Theorem~\ref{limit theorem of weighted Hermite variation processes}
and that in the previous results \cite{liu-tindel}.
\end{remark}

\section{Weak convergence of (weighted) sum processes of Wiener chaos of order 2}
\label{weak convergence}

In this section, $(B_t)$ stands for the fBm with the
Hurst parameter $\frac{1}{3}<H\leq \frac{1}{2}$.
The aim of this section is to prove
 Theorem~\ref{limit theorem of weighted Hermite variation processes} and
Theorem~\ref{Levy area variation} (an FCLT for ``weight-free'' sum processes).
In \cite{neuenkirch-tindel-unterberger2010}, a problem similar to Theorem~\ref{Levy area variation} was considered.
First, we show Theorem~\ref{limit theorem of weighted Hermite variation processes} 
using Theorem~\ref{Levy area variation}.

\begin{proof}[Proof of Theorem $\ref{limit theorem of weighted Hermite variation processes}$]
We prove this theorem by using Theorem~\ref{moment estimate} with the case $p=4$.
By the moment estimate, we see the relative compactness of the processes.
See \cite{BurdzySwanson2010}.
Therefore
it suffices to prove the weak convergence of the finite dimensional 
distributions of $(2^m)^{2H-\frac{1}{2}}I^m_t(F)$.
Let $0<t_1<\cdots<t_L=1$.
Let $m'$ be a positive integer and set
$
 F^{m'}_t=(F^{\alpha,\beta,m'}_t)_{\alpha,\beta}=
(F^{\alpha,\beta}_{[t]_{m'}^-})_{\alpha,\beta}.
$
Then by the stochastic continuity and the assumption (1) in Definition~\ref{def489140919490}, 
it holds that 
\begin{align}
 \lim_{m'\to\infty}\sup_{0\le t\le 1}
\|F^{\alpha,\beta}_t-F^{\alpha,\beta,m'}_t\|_{L^2}=0.\label{mm'}
\end{align}
Let us fix $\vep>0$.
We want to show that for sufficiently large $m'$ and any $m\ge m'$
\begin{align}
\label{integral mm'} \max_{1\le l\le L}
\|(2^m)^{2H-\frac{1}{2}}I^{m}_{t_l}(F^{m'})
-(2^m)^{2H-\frac{1}{2}}I^m_{t_l}(F)\|_{L^2}\le \vep.
\end{align}
Using $I^m_{t_l}(F^{m'})-I^m_{t_l}(F)=I^m_{t_l}(F^{m'}-F)$,
we can expand the integration of (\ref{integral mm'}) and estimate it using the expression in 
Lemma~\ref{estimate of integration by parts} (\ref{estimate of integration by parts 2}).
Note that the terms containing $D^{2r}(F^{m'}_t-F_t)$ with $r>0$
converge to $0$ due to the term $2^{-\frac{r}{2}(4H-1)m}$ as $m'\to \infty$.
The term corresponding to the case $r=0$ also converges to 0 because 
(\ref{mm'}) holds.
Consequently,
we see that (\ref{integral mm'}) holds for sufficiently large $m'$ and any $m(>m')$.
On the other hand, by Theorem~\ref{Levy area variation}, we obtain the weak convergence of the
finite dimensional distribution:
\begin{multline*}
\left((2^m)^{2H-\frac{1}{2}}
I^{m}_{t_1}(F^{m'}),\ldots, 
(2^m)^{2H-\frac{1}{2}}I^{m}_{t_L}(F^{m'})\right)\\
\Longrightarrow
\Bigg(\int_0^{t_1}\sum_{1\le \alpha,\beta\le d}
F^{\alpha,\beta}_{[s]^-_{m'}}dW^{\alpha,\beta}_s,
\ldots,
\int_0^{t_L}\sum_{1\le \alpha,\beta\le d}
F^{\alpha,\beta}_{[s]^-_{m'}}dW^{\alpha,\beta}_s
\Bigg)
\quad
\text{as\,\, $m\to\infty$}.
\end{multline*}
Note that the above $W$ is not the process defined in 
Theorem~\ref{limit theorem of weighted Hermite variation processes} but is
the Gaussian process defined in
Theorem~\ref{Levy area variation}.
Because $(B_t)$ and $(W^{\alpha,\beta}_t)$ are independent, using (\ref{mm'}), we get
\begin{align*}
\lim_{m'\to\infty}
\int_0^{t}\sum_{1\le \alpha,\beta\le d}
F^{\alpha,\beta}_{[s]^-_{m'}}dW^{\alpha,\beta}_s
=\int_0^t
\sum_{1\le \alpha,\beta\le d}
F^{\alpha,\beta}_{s}dW^{\alpha,\beta}_s
\end{align*}
in $L^2$.
Finally, taking the covariance constant $C$ into account,
this completes the proof.
\end{proof}

To state Theorem~\ref{Levy area variation},
we define $d\times d$-matrix valued discrete processes.
The components are given as follows: for $\alpha\neq \beta$, set
\begin{align*}
	\hat{Q}^{m,\alpha,\beta}_{\tmkm,\tmk}
	&=
		\frac{1}{2}
		B^{\alpha}_{\tmkm,\tmk}
		B^{\beta}_{\tmkm,\tmk},
	&
	\hat{Q}^{m,\alpha,\alpha}_{\tmkm,\tmk}
	&=
		0,\\
	\check{Q}^{m,\alpha,\beta}_{\tmkm,\tmk}
	&=
		0,
	&
	\check{Q}^{m,\alpha,\alpha}_{\tmkm,\tmk}
	&=
		\frac{1}{2}
		\left(
			(B^{\alpha}_{\tmkm,\tmk})^2-\Delta_m^{2H}
		\right),\\
	\tilde{Q}^{m,\alpha,\beta}_{\tmkm,\tmk}
	&=
		B^{\alpha,\beta}_{\tmkm,\tmk},
	&
	\tilde{Q}^{m,\alpha,\alpha}_{\tmkm,\tmk}
	&=
		0,\\
	Q^{m,\alpha,\beta}_{\tmkm,\tmk}
	&=
		\hat{Q}^{m,\alpha,\beta}_{\tmkm,\tmk}
		-
		\tilde{Q}^{m,\alpha,\beta}_{\tmkm,\tmk},
	&
	Q^{m,\alpha,\alpha}_{\tmkm,\tmk}
	&=
		0.
\end{align*}
As stated before, non-trivial components are given in terms of Wiener integrals as follows:
\begin{align*}
	\hat{Q}^{m,\alpha,\beta}_{\tmkm,\tmk}
	&=
		\frac{1}{2}
		I_2(\psi^{\alpha}_{\tmkm,\tmk}\odot\psi^{\beta}_{\tmkm,\tmk}),
	&
	\check{Q}^{m,\alpha,\alpha}_{\tmkm,\tmk}
	&=
		I_2\left(\frac{1}{2}(\psi_{\tmkm,\tmk}^\alpha)^{\odot 2}\right),\\
	\tilde{Q}^{m,\alpha,\beta}_{\tmkm,\tmk}
	&=
		I_2(\tilde{\psi}^{\alpha,\beta}_{\tmkm,\tmk}),
	&
	Q^{m,\alpha,\beta}_{\tmkm,\tmk}
	&=
		I_2
		\left(\frac{1}{2}
			\psi^{\alpha}_{\tmkm,\tmk}
			\odot
			\psi^{\beta}_{\tmkm,\tmk}
			-
			\tilde{\psi}^{\alpha,\beta}_{\tmkm,\tmk}
		\right).
\end{align*}
Here $\tilde{\psi}^{\alpha,\beta}_{\tmkm,\tmk}$ is given by \eqref{lim tpsin}.

Note that
$\hat{Q}^m_t=\sum_{i=1}^{\floor{2^mt}}\hat{Q}_{\tmim,\tmi}$,
$\check{Q}^m_t=\sum_{i=1}^{\floor{2^mt}}\check{Q}_{\tmim,\tmi}$,
and 
$Q^m_t=\sum_{i=1}^{\floor{2^mt}}{Q}^m_{\tmim,\tmi}$ 
are 
symmetric matrix-valued
diagonal matrix-valued,
and skew-symmetric matrix-valued discrete processes,
respectively.
Also, we define $\tilde{Q}_t=\sum_{i=1}^{\floor{2^mt}}
\tilde{Q}^{m}_{\tmim,\tmi}$.
We have the following limit theorem.

\begin{theorem}\label{Levy area variation}
 $\RR^d\times (\RR^d\otimes \RR^d)^4$-valued processes
\begin{align*}
\left\{
	\left(
		B_t,
		(2^m)^{2H-\frac{1}{2}}\hat{Q}^m_{[t]^-_m},
		(2^m)^{2H-\frac{1}{2}}\check{Q}^m_{[t]^-_m},
		(2^m)^{2H-\frac{1}{2}}\tilde{Q}^m_{[t]^-_m},
		(2^m)^{2H-\frac{1}{2}}Q^m_{[t]^-_m}
	\right)
\right\}_{0\le t\le 1}
\end{align*}
weakly converges in 
$D\left([0,1],\RR^d\times (\RR^d\otimes \RR^d)^4\right)$ to 
$
 \{(B_t, \hat{W}_t, \check{W}_t, \tilde{W}_t, W_t)\}_{0\le t\le 1},
$
where 
\begin{enumerate}
	\item	$
				\{(B_t, \hat{W}_t, \check{W}_t, \tilde{W}_t, W_t)\}_{0\le t\le 1},
			$
			is a continuous Gaussian process
			whose all means of the components are $0$.
			$\{(\hat{W}_t, \check{W}_t, \tilde{W}_t, W_t)\}_{0\le t\le 1}$
			are independent of $\{B_t\}_{0\le t\le 1}$.
			Moreover, all of their components which are not identically $0$
			are Brownian motions.
	\item	$\{\hat{W}_t\}$ is a symmetric matrix valued Gaussian process
			whose diagonal part is always 0.
	\item	$\{\check{W}_t\}$ and $\{W_t\}$ are
			diagonal matrix-valued and
			skew-symmetric matrix valued
			continuous Gaussian processes, respectively.
	\item	$\{\tilde{W}_t\}$ and $\{\check{W}_t\}$ are independent
			and $\{\hat{W}_t\}$ and $\{\check{W}_t\}$ are independent.
	\item	$\{W_t\}\stackrel{\mathrm{d}}{=}\{\hat{W}_t-\tilde{W}_t\}$  holds.
	\item	Let $\{Y_t\}, \{Z_t\}=\{\check{W}_t\}$ or $\{\tilde{W}_t\}$ or $\{\hat{W}_t\}$.
			Then $\{Y^{\alpha,\beta}_t\}$ and $\{Z^{\alpha',\beta'}_t\}$ are independent
			if $(\alpha,\beta)\ne(\alpha',\beta')$ with $\alpha\ge\beta$, $\alpha'\ge\beta'$.
	\item	All covariances of 
			$
				\{(B_t, \hat{W}_t, \check{W}_t, \tilde{W}_t, W_t)\}_{0\le t\le 1}
			$
			are calculated in the following lemma.
\end{enumerate}
\end{theorem}

For the proof of this theorem, 
it is sufficient to prove the following lemma by the
fourth moment theorem.
We refer the readers for
the fourth moment theorem to \cite{nourdin-peccati}.

\begin{lemma}\label{lemma for Levy area variation}
Let $\trho_H(i)=\int_{0\le u\le 1,i\le v\le i+1}
R([0,u]\times [i,v])dR(u,v)$ $(i=0,1,2,\ldots)$.
\begin{enumerate}
	\item	Let $\alpha\ne\beta$.
			Let for $0\le i, j\le 2^m$, we have
			\begin{align*}
				E
					\left[
						\hat{Q}^{m,\alpha,\beta}_{\tmim,\tmi}
						\hat{Q}^{m,\alpha,\beta}_{\tmjm,\tmj}
					\right]
				&=
					\frac{1}{2}
					E
						\left[
							\check{Q}^{m,\alpha,\alpha}_{\tmim,\tmi}
							\check{Q}^{m,\alpha,\alpha}_{\tmjm,\tmj}
						\right]
				=
					\frac{1}{4}\left(\frac{1}{2^m}\right)^{4H}
					|\rho_H(j-i)|^2,\\
				E
					\left[
						\tilde{Q}^{m,\alpha,\beta}_{\tmim,\tmi}
						\tilde{Q}^{m,\alpha,\beta}_{\tmjm,\tmj}
					\right]
				&=
					\left(\frac{1}{2^m}\right)^{4H}
					\trho_H(j-i),\\
				E
					\left[
						\tilde{Q}^{m,\alpha,\beta}_{\tmim,\tmi}
						\hat{Q}^{m,\alpha,\beta}_{\tmjm,\tmj}
					\right]
				&=
					\frac{1}{4}\left(\frac{1}{2^m}\right)^{4H}
					|\rho_H(j-i)|^2.
			\end{align*}
		Particularly for $s,t\ge 0$, we have
		\begin{align*} 
			\lim_{m\to\infty}
				(2^m)^{4H-1}
				E
					\Big[
						\Big(
							\hat{Q}^{m,\alpha,\beta}_{[s]^-_m,[t]^-_m}
						\Big)^2
					\Big]
			&=
				\frac{1}{2}
				\lim_{m\to\infty}
					(2^m)^{4H-1}
					E
						\Big[
							\Big(
								\check{Q}^{m,\alpha,\alpha}_{[s]^-_m,[t]^-_m}
							\Big)^2
						\Big]
			=
				\frac{\sigma^2}{4}
				(t-s),\\
			\lim_{m\to\infty}
				(2^m)^{4H-1}
				E
					\Big[
						\Big(
							\tilde{Q}^{m,\alpha,\beta}_{[s]^-_m,[t]^-_m}
						\Big)^2
					\Big]
			&=
				\tilde{\sigma}^2
				(t-s),\\
			\lim_{m\to\infty}
			(2^m)^{4H-1}
			E
				\Big[
					\tilde{Q}^{m,\alpha,\beta}_{0,[t]^-_m}
					\hat{Q}^{m,\alpha,\beta}_{0,[s]^-_m}
				\Big]
			&=
				\frac{\sigma^2}{4}
				t\wedge s,
		\end{align*}
		where
		\begin{align*}
			\tilde{\sigma}^2
			&=
				\trho_H(0)+2\sum_{i=1}^{\infty}\trho_H(i),
			&
			\sigma^2
			&=
				\rho_H(0)^2+2\sum_{i=1}^{\infty}\rho_H(i)^2.
		\end{align*}
	\item	Let $1\le \alpha,\beta,\alpha',\beta'\le d$.
			For $0\le i,j\le 2^m$, we have
			\begin{align*} 
				E
					\left[
						\hat{Q}^{m,\alpha,\beta}_{\tmim,\tmi}
						\hat{Q}^{m,\alpha',\beta'}_{\tmjm,\tmj}
					\right]
				&=
					E
						\left[
							\tilde{Q}^{m,\alpha,\beta}_{\tmim,\tmi}
							\tilde{Q}^{m,\alpha',\beta'}_{\tmjm,\tmj}
						\right]
				=
					E
						\left[
							\hat{Q}^{m,\alpha,\beta}_{\tmim,\tmi}
							\tilde{Q}^{m,\alpha',\beta'}_{\tmjm,\tmj}
						\right]
				=
					0
				\quad
				(\{\alpha,\beta\}\ne \{\alpha',\beta'\}),\\
				E
					\left[
						\hat{Q}^{m,\alpha,\beta}_{\tmim,\tmi}
						\check{Q}^{m,\alpha',\beta'}_{\tmjm,\tmj}
					\right]
				&=
					E
						\left[
							\tilde{Q}^{m,\alpha,\beta}_{\tmim,\tmi}
							\check{Q}^{m,\alpha',\beta'}_{\tmjm,\tmj}
						\right]
				=
					0
				\quad 
				(\text{for all $(\alpha,\beta), (\alpha',\beta')$}),\\
				E
					\left[
						\check{Q}^{m,\alpha,\beta}_{\tmim,\tmi}
						\check{Q}^{m,\alpha',\beta'}_{\tmjm,\tmj}
					\right]
				&=
					0
				\qquad
				((\alpha,\beta)\ne (\alpha',\beta')).
			\end{align*}
			Particularly for $s,t\geq $, we have
			\begin{align*} 
				E
					\left[
						\hat{Q}^{m,\alpha,\beta}_{[t]^-_m}
						\hat{Q}^{m,\alpha',\beta'}_{[s]^-_m}
					\right]
				&=
					E
						\left[
							\tilde{Q}^{m,\alpha,\beta}_{[t]^-_m}
							\tilde{Q}^{m,\alpha',\beta'}_{[s]^-_m}
						\right]
				=
					E
						\left[
							\hat{Q}^{m,\alpha,\beta}_{[t]^-_m}
							\tilde{Q}^{m,\alpha',\beta'}_{[s]^-_m}
				\right]
				=
					0
				\quad
				(\{\alpha,\beta\}\ne \{\alpha',\beta'\}),\\
			E
				\left[
					\hat{Q}^{m,\alpha,\beta}_{[t]^-_m}
					\check{Q}^{m,\alpha',\beta'}_{[s]^-_m}
				\right]
			&=
				E
					\left[
						\tilde{Q}^{m,\alpha,\beta}_{[t]^-_m}
						\check{Q}^{m,\alpha',\beta'}_{[s]^-_m}
					\right]
			=
				0
			\quad 
				(\text{for all $(\alpha,\beta), (\alpha',\beta')$}),\\
			E
				\left[
					\check{Q}^{m,\alpha,\beta}_{[t]^-_m}
					\check{Q}^{m,\alpha',\beta'}_{[s]^-_m}
			\right]
			&=
				0
			\quad ((\alpha,\beta)\ne (\alpha',\beta')).
			\end{align*}

	\item	Let
			\begin{align*}
				\Psi^{\alpha,\beta}_{i,j}
				&=
					\tpsi^{\alpha,\beta}_{\tmim,\tmi}
					\mathop{\underset{1}{\tilde{\otimes}}}
					\tpsi^{\alpha,\beta}_{\tmjm,\tmj}\quad (\alpha\ne \beta),\\
				\Phi^{\alpha,\beta}_{i,j}
				&=
					\left(\psi^{\alpha}_{\tmim,\tmi}\odot \psi^{\beta}_{\tmim,\tmi}\right)
					\mathop{\underset{1}{\tilde{\otimes}}}
					\left(\psi^{\alpha}_{\tmjm,\tmj}\odot \psi^{\beta}_{\tmjm,\tmj}\right),\\
				\Theta^{\alpha,\beta}_{i,j}
				&=
					\left(\psi^{\alpha}_{\tmim,\tmi}\odot \psi^{\beta}_{\tmim,\tmi}\right)
					\mathop{\underset{1}{\tilde{\otimes}}}
					\tpsi^{\alpha,\beta}_{\tmjm,\tmj}\quad (\alpha\ne \beta).
			\end{align*}
			Then it holds that
			\begin{align*}
				\max
					\Bigg\{
							\Bigg\|\sum_{i,j=1}^L\Psi^{\alpha,\beta}_{i,j}\Bigg\|_{(\mathcal{H}^d)^{\otimes 2}}^2,
							\Bigg\|\sum_{i,j=1}^L\Phi^{\alpha,\beta}_{i,j}\Bigg\|_{(\mathcal{H}^d)^{\otimes 2}}^2,
							\Bigg\|\sum_{i,j=1}^L\Theta^{\alpha,\beta}_{i,j}\Bigg\|_{(\mathcal{H}^d)^{\otimes 2}}^2
					\Bigg\}
				\le 
					\frac{CL}{2^{8Hm}}
					\Bigg(\sum_{n=0}^{\infty}|\rho_H(n)|\Bigg)^3.
			\end{align*}
\end{enumerate}
\end{lemma}

\begin{remark}
 Let $X^m_t$ and $Y^m_t$ be one of
$\tilde{Q}^{m,\alpha,\beta}_t, \check{Q}^{m,\alpha,\beta}_t, 
\hat{Q}^{m,\alpha,\beta}_t$ $(t\in D_m,\,\,1\le \alpha,\beta\le d)$.
Noting that
$\tilde{Q}^{m,\alpha,\beta}_t+\tilde{Q}^{m,\beta,\alpha}_t
=2\hat{Q}^{m,\alpha,\beta}_t$,
we see that the results in the above lemma gives all limits of
\[
 \lim_{m\to\infty}E[X^m_{[t]^-_m,[t']^-_m}
Y^m_{[s]^-_m,[s']^-_m}]=C_{X^m,Y^m}(t'\wedge s'-t'\wedge s-t\wedge s'+
t\wedge s)\quad  (0\le t<t',\,\,0\le s<s').
\]
That is, all covariances of $(B_t, \hat{W}_t, \check{W}_t, \tilde{W}_t, W_t)$
can be determined.
\end{remark}

\begin{proof}
The constants below are independent of $m$ and may change line by line.

(1)
First we consider $\hat{Q}^{m,\alpha,\beta}$.
By definition we have
\begin{align*}
	E
		\left[
			\hat{Q}^{m,\alpha,\beta}_{\tmim,\tmi}
			\hat{Q}^{m,\alpha,\beta}_{\tmjm,\tmj}
		\right]
	=
		\frac{1}{4}
		R\left([\tmim,\tmi]\times [\tmjm,\tmj]\right)^2
	=
		\frac{1}{4}
		\left(\frac{1}{2^m}\right)^{4H}
		R([i-1,i]\times [j-1,j])^2.
\end{align*}

We consider $\check{Q}^{m,\alpha,\beta}$.
Here, $(2^m)^{2H}(B^{\alpha}_{\tmim,\tmi})^2-1=H_2(2^{mH}B^{\alpha}_{\tmim,\tmi})$,
where $H_2$ is the Hermite polynomial of degree 2.
Therefore
\begin{align*}
E\left[\check{Q}^{m,\alpha,\alpha}_{\tmim,\tmi}
\check{Q}^{m,\alpha,\alpha}_{\tmjm,\tmj}\right]
&=
(2^m)^{-4H}\frac{1}{4}E\left[H_2(2^{mH}B^{\alpha}_{\tmim,\tmi})
H_2(2^{mH}B^{\alpha}_{\tmjm,\tmj})
\right]\nonumber\\
&=\frac{1}{2}E[B^{\alpha}_{\tmim,\tmi}B^{\alpha}_{\tmjm,\tmj}]^2\\
&=
\frac{1}{2}\frac{|\rho_H(i-j)|^2}{2^{4mH}}.
\end{align*}

Next we consider $\tilde{Q}^{m,\alpha,\beta}_{\tmim,\tmi}$.
Let $u_k=\tmim+\frac{k}{n}\frac{1}{2^m}$
and $v_l=\tmjm+\frac{l}{n}\frac{1}{2^m}$.
Recall that $\tilde{\psi}^{\alpha,\beta}_{s,t}(n)$ is 
a finite-dimensional approximation of
$\tilde{\psi}^{\alpha,\beta}_{s,t}$ which is defined in \eqref{tpsin}.
We have
\begin{align*}
	\left(
		\tpsi^{\alpha,\beta}_{\tmim,\tmi}(n),
		\tpsi^{\alpha,\beta}_{\tmjm,\tmj}(n)
	\right)_{(\mathcal{H}^d)^{\otimes 2}}
	&=
		\frac{1}{2}
		\sum_{k,l=1}^n
			R([\tmim,u_{k-1}]\times [\tmjm,v_{l-1}])
			R([u_{k-1},u_k]\times [v_{l-1},v_l]).
\end{align*}
Therefore,
\begin{align*}
	E
		\left[
			\tilde{Q}^{m,\alpha,\beta}_{\tmim,\tmi}
			\tilde{Q}^{m,\alpha,\beta}_{\tmjm,\tmj}
		\right]
	&=
		\lim_{n\to\infty}
		\sum_{k,l=1}^n
			R([\tmim,u_{k-1}]\times [\tmjm,v_{l-1}])
			R([u_{k-1},u_k]\times [v_{k'-1},v_l])\\
	&=
		\int_{\tmim\le u\le\tmi, \tmjm\le v\le \tmj}
			R\left([\tmim,u]\times[\tmjm,v]\right)
			dR(u,v)\\
	&=
		\left(\frac{1}{2^m}\right)^{4H}
		\int_{0\le u\le 1, j-i\le v\le j-i+1}
			R([0,u]\times [j-i,v])
			dR(u,v)\\
	&=
		\left(\frac{1}{2^m}\right)^{4H}\trho_H(j-i),
\end{align*}
where we have used the translation invariant property and
the scaling property of fBm (see Lemma~\ref{properties of R}).

Similarly, using finite dimensional approximation, we obtain
\begin{align*}
	E
		\left[
			\tilde{Q}^{m,\alpha,\beta}_{\tmim,\tmi}
			\hat{Q}^{m,\alpha,\beta}_{\tmjm,\tmj}
		\right]
	&=
		\frac{1}{2}
		\int_{\tmim}^{\tmi}
			R\left([\tmim,u]\times [\tmjm,\tmj]\right)
			dR\left([\tmim,u]\times [\tmjm,\tmj]\right)\\
	&=
		\frac{1}{2}
		\frac{1}{2^{4mH}}
		\int_{i-1}^iR\left([i-1,u]\times [j-1,j]\right)dR(u,[j-1,j])\\
	&=
		\frac{1}{4}
		\frac{1}{2^{4mH}}
		R([i-1,i]\times [j-1,j])^2.
\end{align*}

Noting
that $\sum_{n=0}^{\infty}|\rho_H(n)|<\infty$
and $\sum_{n=0}^{\infty}|\trho_H(n)|<\infty$,
all the proofs of the other identities in the assertion are elementary calculations.
We omit the proof.

(2) Assertions follow from the independence of $(B^{\alpha}_t)$ and $(B^{\beta}_t)$ $(\alpha\ne\beta)$.

(3) Let
$
	\Psi^{\alpha,\beta}_{i,j}(n)
	=
		\tpsi^{\alpha,\beta}_{\tmim,\tmi}(n)
		\mathop{\underset{1}{\tilde{\otimes}}}
		\tpsi^{\alpha,\beta}_{\tmjm,\tmj}(n)
	$.
By $\tpsi^{\alpha,\beta}_{\tmim,\tmi}=
\lim_{n\to\infty}\tpsi^{\alpha,\beta}_{\tmim,\tmi}(n)$ in 
$(\mathcal{H}^d)^{\odot 2}$ and the continuity of the contraction operation,
we have
$\lim_{n\to\infty}\Psi^{\alpha,\beta}_{i,j}(n)=\Psi^{\alpha,\beta}_{i,j}$
in $(\mathcal{H}^d)^{\odot 2}$.
Therefore it suffices to
give an estimate of
$\|\sum_{i,j=1}^L\Psi^{\alpha,\beta}_{i,j}(n)\|_{(\mathcal{H}^d)^{\otimes 2}}^2$
that is independent of $n$.
Here, we use the partition $\{u_k\}_{k=0}^n$ and $\{v_l\}_{l=0}^n$ of
$[\tmim,\tmi]$ and $[\tmjm,\tmj]$ in (1).
First, note that
$
	\Psi^{\alpha,\beta}_{i,j}(n)
	=
		\frac{1}{4}
		\{
			\Psi^{1,\alpha,\alpha}_{i,j}(n)
			+
			\Psi^{2,\beta,\beta}_{i,j}(n)
		\}
$,
where
\begin{align*}
	\Psi^{1,\alpha,\alpha}_{i,j}(n)
	&=
		\sum_{k,l=1}^n
			R([u_{k-1},u_k]\times [v_{l-1},v_l])
			\psi^{\alpha}_{\tmim,u_{k-1}}
			\odot
			\psi^{\alpha}_{\tmjm,v_{l-1}},\\
	\Psi^{2,\beta,\beta}_{i,j}(n)
	&=
		\sum_{k,l=1}^n
			R([\tmim,u_{k-1}]\times [\tmjm,v_{l-1}])
			\psi^{\beta}_{u_{k-1},u_k}
			\odot
			\psi^{\beta}_{v_{l-1},v_l}.
\end{align*}
Because $\Psi^{1,\alpha,\alpha}_{i,j}(n)$ and $\Psi^{2,\beta,\beta}_{i,j}(n)$
are orthogonal in $(\mathcal{H}^d)^{\otimes 2}$,
we see
\begin{align}
	\label{eq481901}
	\Bigg\|
		\sum_{i,j=1}^L\Psi^{\alpha,\beta}_{i,j}(n)
	\Bigg\|_{(\mathcal{H}^d)^{\otimes 2}}^2
	&=
		\frac{1}{4^2}
		\sum_{i,j,i',j'=1}^L
			\left(
				\Psi^{1,\alpha,\alpha}_{i,j}(n),
				\Psi^{1,\alpha,\alpha}_{i',j'}(n)
			\right)_{(\mathcal{H}^d)^{\otimes 2}}\\
	&\qquad\qquad
		+
		\frac{1}{4^2}
		\sum_{i,j,i',j'=1}^L
			\left(
				\Psi^{2,\beta,\beta}_{i,j}(n),
				\Psi^{2,\beta,\beta}_{i',j'}(n)
			\right)_{(\mathcal{H}^d)^{\otimes 2}}. \nonumber
\end{align}
Therefore, an estimate of
$\|\sum_{i,j=1}^L\Psi^{\alpha,\beta}_{i,j}(n)\|_{(\mathcal{H}^d)^{\otimes 2}}^2$
follows from those of
$
	(
		\Psi^{1,\alpha,\alpha}_{i,j}(n),
		\Psi^{1,\alpha,\alpha}_{i',j'}(n)
	)_{(\mathcal{H}^d)^{\otimes 2}}
$
and
$
	(
		\Psi^{2,\beta,\beta}_{i,j}(n),
		\Psi^{2,\beta,\beta}_{i',j'}(n)
	)_{(\mathcal{H}^d)^{\otimes 2}}
$.
Noting that
$
	2(\xi\odot\eta,\xi'\odot\eta')_{(\mathcal{H}^d)^{\otimes 2}}
	=
		(\xi,\xi')_{\mathcal{H}^d}
		(\eta,\eta')_{\mathcal{H}^d}
		+
		(\xi,\eta')_{\mathcal{H}^d}
		(\eta,\xi')_{\mathcal{H}^d}
$,
and
using 
$u'_{k'}=\tau^m_{i'-1}+\frac{k'}{n}\frac{1}{2^m}$
and $v'_{l'}=\tau^m_{j'-1}+\frac{l'}{n}\frac{1}{2^m}$,
we have
\begin{multline*}
	\left(
		\Psi^{1,\alpha,\alpha}_{i,j}(n),
		\Psi^{1,\alpha,\alpha}_{i',j'}(n)
	\right)_{(\mathcal{H}^d)^{\otimes 2}}
	=
		\frac{1}{2}
		\sum_{k,l,k',l'=1}^n
			R([u_{k-1},u_k]\times [v_{l-1},v_l])
			R([u'_{k'-1},u'_{k'}]\times [v'_{l'-1},v'_{l'}])\\
			\times
			\{
				R([\tmim,u_{k-1}]\times [\tau^m_{i'-1},u'_{k'-1}])
				R([\tmjm,v_{l-1}]\times [\tau^m_{j'-1},v'_{l'-1}])\\
				+
				R([\tmim,u_{k-1}]\times [\tau^m_{j'-1},v'_{l'-1}])
				R([\tmjm,v_{l-1}]\times [\tau^m_{i'-1},u'_{k'-1}])
			\}
\end{multline*}
and
\begin{multline*}
	\left(
		\Psi^{2,\beta,\beta}_{i,j}(n),
		\Psi^{2,\beta,\beta}_{i',j'}(n)
	\right)_{(\mathcal{H}^d)^{\otimes 2}}
	=
		\frac{1}{2}
		\sum_{k,l,k',l'=1}^n
			R([\tmim,u_{k-1}]\times [\tmjm,v_{l-1}])
			R([\tau^m_{i'-1},u'_{k'-1}]\times [\tau^m_{j'-1},v'_{l'-1}])\\
			\times
			\{
				R([u_{k-1},u_k]\times [u'_{k'-1},u'_{k'}])
				R([v_{l-1},v_l]\times [v'_{l'-1},v'_{l'}])\\
				+
				R([u_{k-1},u_k]\times [v'_{l'-1},v'_{l'}])
				R([v_{l-1},v_l]\times [u'_{k'-1},u'_{k'}])
			\}
\end{multline*}
From Lemma~\ref{properties of R} (4),
Theorem~\ref{Towghi} (1) and Lemma~\ref{p-variation norm for product} (1),
we arrive at
\begin{multline}
	\label{eq4519081}
	\Big|
		\left(
			\Psi^{1,\alpha,\alpha}_{i,j}(n),
			\Psi^{1,\alpha,\alpha}_{i',j'}(n)
		\right)_{(\mathcal{H}^d)^{\otimes 2}}
	\Big|
	+
	\Big|
		\left(
			\Psi^{2,\beta,\beta}_{i,j}(n),
			\Psi^{2,\beta,\beta}_{i',j'}(n)
		\right)_{(\mathcal{H}^d)^{\otimes 2}}
	\Big|\\
	\leq
		C
		\left(\frac{1}{2^m}\right)^{8H}
		|\rho_H(i-j)\rho_H(i'-j')|
		\{|\rho_H(i-i')\rho_H(j-j')|+|\rho_H(i-j')\rho_H(j-i')|\}.
\end{multline}
From \eqref{eq481901} and \eqref{eq4519081}, we see
\begin{align}
	\label{eq4819011}
	\Bigg\|
		\sum_{i,j=1}^L
			\Psi^{\alpha,\beta}_{i,j}(n)
	\Bigg\|_{(\mathcal{H}^d)^{\otimes 2}}^2
	&\le 
		CL
		\left(\frac{1}{2^m}\right)^{8H}
		\Bigg(\sum_{l=0}^{\infty}|\rho_H(l)|\Bigg)^3.
\end{align}

Next, we give an estimate of $\Theta^{\alpha,\beta}_{i,j}$.
Similarly, using the partition points
$\{v_l\}$ of $[\tmjm,\tmj]$,
we have
$
	\Theta^{\alpha,\beta}_{i,j}(n)
	=
		\frac{1}{4}
		\{
			\Theta^{1,\alpha,\alpha}_{i,j}(n)
			+
			\Theta^{2,\beta,\beta}_{i,j}(n)
		\}
$,
where
\begin{align*}
	\Theta^{1,\alpha,\alpha}_{i,j}(n)
	=
	\sum_{l=1}^n
		R([\tmim,\tmi]\times[v_{l-1},v_l])
		\psi^{\alpha}_{\tmim,\tmi}
		\odot
		\psi^{\alpha}_{\tmjm,v_{l-1}},\\
	\Theta^{2,\beta,\beta}_{i,j}(n)
	=
		\sum_{l=1}^n
			R([\tmim,\tmi]\times [\tmjm,v_{l-1}])
			\psi_{\tmim,\tmi}^{\beta}
			\odot
			\psi^{\beta}_{v_{l-1},v_l}.
\end{align*}
Therefore
\begin{multline*}
	\left(
		\Theta^{1,\alpha,\alpha}_{i,j}(n),
		\Theta^{1,\alpha,\alpha}_{i',j'}(n)
	\right)_{(\mathcal{H}^d)^{\otimes 2}}
	=
		\frac{1}{2}
		\sum_{l,l'=1}^n
			R([\tmim,\tmi]\times[v_{l-1},v_l])
			R([\tau^m_{i'-1},\tau^m_{i'}]\times[v'_{l'-1},v'_{l'}])\\
			\times
			\{
				R([\tmim,\tmi]\times[\tau^m_{i'-1},\tau^m_{i'}])
				R([\tmjm,v_{l-1}]\times[\tau^m_{j'-1},v'_{l'-1}])\\
				+
				R([\tmim,\tmi]\times[\tau^m_{j'-1},v'_{l'-1}])
				R([\tmjm,v_{l-1}]\times[\tau^m_{i'-1},\tau^m_{i'}])
			\}
\end{multline*}
and
\begin{multline*}
	\left(
		\Theta^{2,\beta,\beta}_{i,j}(n),
		\Theta^{2,\beta,\beta}_{i',j'}(n)
	\right)_{(\mathcal{H}^d)^{\otimes 2}}
	=
		\frac{1}{2}
		\sum_{l,l'=1}^n
			R([\tmim,\tmi]\times [\tmjm,v_{l-1}])
			R([\tau^m_{i'-1},\tau^m_{i'}]\times[\tau^m_{j'-1},v'_{l'}])\\
			\times
			\{
				R([\tmim,\tmi]\times[\tau^m_{i'-1},\tau^m_{i'}])
				R([v_{l-1},v_l]\times[v'_{l'-1},v'_{l'}])\\
				+
				R([\tmim,\tmi]\times[v'_{l'-1},v'_{l'}])
				R([v_{l-1},v_l]\times[\tau^m_{i'-1},\tau^m_{i'}])
			\}.
\end{multline*}
We deduce from these identities that the same estimate with \eqref{eq4519081} for
$
	(
		\Theta^{1,\alpha,\alpha}_{i,j}(n),
		\Theta^{1,\alpha,\alpha}_{i',j'}(n)
	)_{(\mathcal{H}^d)^{\otimes 2}}
$
and
$
	(
		\Theta^{2,\beta,\beta}_{i,j}(n),
		\Theta^{2,\beta,\beta}_{i',j'}(n)
	)_{(\mathcal{H}^d)^{\otimes 2}}
$
hold. Then we conclude \eqref{eq4819011}
being replaced $\Psi^{\alpha,\beta}_{i,j}(n)$ by $\Theta^{\alpha,\beta}_{i,j}(n)$.

Finally, we consider $\Phi^{\alpha,\beta}_{i,j}$.
Noting that
\begin{align*}
	\Phi^{\alpha,\beta}_{i,j}
	&=
		\frac{1}{4}
		R([\tmim,\tmi]\times[\tmjm,\tmj])
		\big\{
			\psi^{\beta}_{\tmim,\tmi}\odot\psi^{\beta}_{\tmjm,\tmj}
			+
			\psi^{\alpha}_{\tmim,\tmi}\odot\psi^{\alpha}_{\tmjm,\tmj}
		\big\},
\end{align*}
we have
\begin{multline*}
	\Bigg\|\sum_{i,j=1}^L\Phi^{\alpha,\beta}_{i,j}\Bigg\|_{(\mathcal{H}^d)^{\otimes 2}}^2
	=
		\sum_{i,j,i',j'=1}^L
			\left(\Phi^{\alpha,\beta}_{i,j},\Phi^{\alpha,\beta}_{i',j'}
			\right)_{(\mathcal{H}^d)^{\otimes 2}}\\
	\begin{aligned}
		&=
		\frac{1}{4^2}
		\frac{1}{2^{8mH}}
		\sum_{i,j,i',j'=1}^L
			|\rho_H(i-j)\rho_H(i'-j')|
			\big\{|\rho_H(i-i')\rho_H(j-j')|+|\rho_H(i-j')\rho_H(j-i')|\big\}\\
		&\le
			CL
			\left(\frac{1}{2^m}\right)^{8H}
			\left(\sum_{l=0}^{\infty}|\rho_H(l)|\right)^3.
	\end{aligned}
\end{multline*}
This completes the proof.
\end{proof}

\section{H\"older estimates of (weighted) sum processes of Wiener chaos of order $3$}
\label{Holder esitmates of sum processes of Wiener chaos of order 3}

Theorems~\ref{moment estimate} and \ref{limit theorem of weighted Hermite variation processes}
involve the weighted sum of elements in the Wiener chaos of order $2$.
Throughout this section, $(B_t)$ stands for the fBm with
Hurst parameter $\frac{1}{3}<H\leq\frac{1}{2}$.
In this section, we treat the weighted sums of elements in Wiener chaos of order $3$
similarly.
Set
\begin{align*}
	\tilde{\mathcal{K}}^3_m
	=
		\left\{
			\{B^{\alpha,\beta,\gamma}_{\tmim,\tmi}\}_{i=1}^{2^m},
			\,
			\{B^{\alpha,\beta}_{\tmim,\tmi}B^{\gamma}_{\tmim,\tmi}\}_{i=1}^{2^m},
			\,
			\{B^{\alpha}_{\tmim,\tmi}B^{\beta}_{\tmim,\tmi}B^{\gamma}_{\tmim,\tmi}\}_{i=1}^{2^m}
			\, ~;~
			1\le \alpha,\beta,\gamma\le d
		\right\}.
\end{align*}
First, we denote elements of
$\tilde{\mathcal{K}}^3_m$ by $K^m=\{K^m_{\tmim,\tmi}\}_{i=1}^{2^m}$.
We write $K^m_t=\sum_{i=1}^{\floor{2^mt}}K^m_{\tmim,\tmi}$
and $K^m_0=0$ and we denote the all $\{K^m_t\}_{t\in D_m}$
by $\mathcal{K}^3_m$.
We will show the next proposition.
\begin{proposition}\label{estimate of third order}
	Let $\frac{1}{3}<H^-<H$.
	Let $K^{m}\in \mathcal{K}^3_{m}$.
	Assume that for every $m$,
	a discrete process $\{F^m_t\}_{t\in D_m}$ satisfies
	$|F^m_0|+\|F^m\|_{H^-}\le C$, 
	where $C$ is a random variable independent of $m$.
	Let
	\begin{align*}
		I^m_t(F^m)
		=
			\sum_{i=1}^{\floor{2^mt}}F^m_{\tmim}K^{m}_{\tmim,\tmi}.
	\end{align*}
	Then it holds that 
	\begin{align*}
		\|(2^m)^{2H-\frac{1}{2}}I^m(F^m)\|_{2H^-}
		&\le
			C
			(2^{-m})^\vep
			G_\vep.
	\end{align*}
	Here $\vep$ is the positive number and $G_\vep\in L^{\infty-}$ is a random variable 
	specified in Lemma~\ref{estimate of K}.
\end{proposition}

\subsection{Proof of Proposition \ref{estimate of third order}}

Accepting Lemma~\ref{covariance estimate of Km} below for the moment,
we show Proposition \ref{estimate of third order}.
We will show Lemma~\ref{covariance estimate of Km} in the next subsection.
\begin{lemma}
	\label{covariance estimate of Km}
	Let $(K^m_t)\in \tilde{\mathcal{K}}^3_m$.
	Then there exists a $C$ is a positive constant such that 
	\begin{align*}
		\big|
				E
					\big[
						K^{m}_{\tmim,\tmj}
						K^{m}_{\tmjm,\tmj}
					\big]
		\big|
		\le
			C(2^{-m})^{6H}\sum_{k=1}^3|\rho_H(i-j)|^k
		\qquad
		\text{for all $s,t\in D_m$ with $s<t$ and $m$.}
	\end{align*}
\end{lemma}
The next lemma follows from Lemma~\ref{covariance estimate of Km}.
\begin{lemma}\label{estimate of K}
	Let $(K^m_t)\in \mathcal{K}^3_m$.
	The following hold.
	\begin{enumerate}
		\item 	Let $p\geq2$. 
				Then there exists a $C$ is a positive constant depending only on $p$ such that
				\begin{align*}
					E\big[\big|K^{m}_{s,t}\big|^p\big]
					\le
						C_p
						\left(2^{-m}\right)^{(3H-\frac{1}{2})p}
						(t-s)^{\frac{p}{2}}
					\qquad
					\text{for all $s,t\in D_m$ with $s<t$ and $m$.}
				\end{align*}
		\item	For any $\vep'>0$ and $m$, there exists a positive random variable $G_{m,\vep'}$
				such that
				\begin{align*}
					&\sup_m\|G_{m,\vep'}\|_{L^p}
					<
						\infty
					\qquad
					\text{ for all $p\ge 1$,}\\
					&|K^{m}_{s,t}|
					\le
						(2^{-m})^{3H-\frac{1}{2}}
						G_{m,\vep'}
						|t-s|^{\frac{1}{2}-\vep'}
					\qquad
					\text{for all $s,t\in D_m$ with $s<t$ and $m$.}
				\end{align*}
		\item	Let $\frac{1}{3}<H^-<H$.
                        There exists $\vep>0$ and a positive random variable
				$G_\vep\in L^{\infty-}$ such that
				\begin{align*}
					|(2^m)^{2H-\frac{1}{2}}K^{m}_{s,t}|
					\le
						(2^{-m})^\vep
						G_\vep
						|t-s|^{2H^{-}}
					\qquad
					\text{for all $s,t\in D_m$ with $s<t$ and $m$.}
				\end{align*}
	\end{enumerate}
\end{lemma}

\begin{proof}
	We show assertion (1) for $p=2$.
	From Lemma~\ref{covariance estimate of Km}, we obtain the following for $s=\tmk<\tml=t$
	\begin{align*}
		E[(K^m_{s,t})^2]
		\le
			C\sum_{i,j=k+1}^l(2^{-m})^{6H}|\rho_H(i-j)|
		\le
			C'(2^{-m})^{6H}(l-k).
	\end{align*}
	Noting that $(2^{-m})^{6H}(l-k)=(2^{-m})^{6H-1}\frac{l-k}{2^m}=(2^{-m})^{6H-1}(t-s)$,
	we see assertion (1) for $p=2$.
	Combining the hypercontractivity of the Ornstein-Uhlenbeck semigroup and the case $p=2$,
	we obtain the case $p>2$.
We can prove
	assertion (2) using the assertion (1) and the Garsia-Rodemich-Rumsey inequality
 in a similar way to
the proof of Corollary~\ref{cor to moment estimate}.
	Noting that $\frac{1}{2}+H>2H^-$, we can prove assertion~(3).
\end{proof}

\begin{proof}[Proof of Proposition $\ref{estimate of third order}$]
	By the assumption on the H\"older norm of
	$F^m$ and Lemma~\ref{estimate of K} (2) and 
	using the estimate of discrete Young integral, we see that the assertion holds.
\end{proof}

\subsection{Covariance of Wiener chaos of order $3$}

Next, we prove Lemma~\ref{covariance estimate of Km}.
Because $K^{m}_{\tmim,\tmi}$ belongs to the Wiener chaos of order less than or equal to $3$,
one can write
\begin{align*}
	K^{m}_{\tmim,\tmi}
	=
		I_3(\Gamma_{\tmim,\tmi})+I_1(l_{\tmim,\tmi}),
\end{align*}
where $\Gamma_{{\tmim,\tmi}}\in (\mathcal{H}^d)^{\odot 3}$
and $l_{\tmim,\tmi}\in \mathcal{H}^d$.
From this, we have
\begin{align*}
	E[K^{m}_{\tmim,\tmi}K^{m}_{\tmjm,\tmj}]
	&=
		E[I_3(\Gamma_{\tmim,\tmi})I_3(\Gamma_{\tmjm,\tmj})]
		+
		E[I_1(l_{\tmim,\tmi})I_1(l_{\tmjm,\tmj})]\\
	&=
	(\Gamma_{\tmim,\tmi},\Gamma_{\tmjm,\tmj})_{(\mathcal{H}^d)^{\otimes 3}}
		+
		(l_{\tmim,\tmi},l_{\tmjm,\tmj})_{\mathcal{H}^d}.
\end{align*}
For each $K^{m}_{\tmim,\tmi}$,
we will specify $\Gamma_{\tmim,\tmi}$ and $l_{\tmim,\tmi}$
and estimate the covariance.
Estimates for the covariances can be given using
their finite-dimensional approximations
$\Gamma_{\tmim,\tmi}(n)$ and $l_{\tmim,\tmi}(n)$ of
$\Gamma_{\tmim,\tmi}$ and $l_{\tmim,\tmi}$ in the sense that
\[
 \lim_{n\to\infty}\Gamma_{\tmim,\tmi}(n)=\Gamma_{\tmim,\tmi} \quad
\text{in $(\mathcal{H}^d)^{\otimes 3}$},\quad
\lim_{n\to\infty}l_{\tmim,\tmi}(n)=l_{\tmim,\tmi}\quad 
\text{in $\mathcal{H}^d$}.
\]
In what follows, we will use
\begin{align}\label{eq_498311}
	 E[|K^m_{s,t}|^2]\le 2E[|K^{m,1}_{s,t}|^2]+2E[|K^{m,1}_{s,t}|^2]
\end{align}
for $K^m_t=K^{m,1}_{t}+K^{m,2}_{t}$ $(t\in D_m)$.

In the calculation below, we use
\begin{align}
 \left(x_1\odot\cdots\odot x_p,
y_1\odot\cdots\odot y_p\right)_{(\mathcal{H}^d)^{\otimes p}}&=
\frac{1}{p!}\sum_{\sigma\in \mathfrak{G}_p}
\prod_{i=1}^p\left(x_i,y_{\sigma(i)}\right)_{\mathcal{H}^d},
\end{align}
where $x_i,y_j\in \mathcal{H}^d$
and
\begin{align}
 	(\psi^{\alpha}_{\tmim,\tmi},\psi^{\beta}_{\tmjm,\tmj})_{\mathcal{H}^d}
	=
		R([\tmim,\tmi]\times [\tmjm,\tmj])\delta_{\alpha,\beta},
\end{align}
and Proposition~\ref{product formula} (2) 
and Corollary~\ref{finite dimensional approximation}.

\subsubsection{
	The case
	$
		K^{m}_{\tmim,\tmi}
		=
			B^{\alpha}_{\tmim,\tmi}B^{\beta}_{\tmim,\tmi}B^{\gamma}_{\tmim,\tmi}
	$
}
For different integers $\alpha,\beta,\gamma$,
set
\begin{align*}
	\Gamma^{1,\alpha,\alpha,\alpha}_{\tmim,\tmi}
	&=
		(\psi^{\alpha}_{\tmim,\tmi})^{\otimes 3},
	&
	l^{1,\alpha,\alpha,\alpha}_{\tmim,\tmi}
	&=
		3\cdot 2^{-2mH}\psi^{\alpha}_{\tmim,\tmi},\\
	\Gamma^{1,\alpha,\beta,\gamma}_{\tmim,\tmi}
	&=
		\psi^{\alpha}_{\tmim,\tmi}
		\odot
		\psi^{\beta}_{\tmim,\tmi}
		\odot
		\psi^{\gamma}_{\tmim,\tmi},
	&
	l^{1,\alpha,\beta,\gamma}_{\tmim,\tmi}
	&=
		0,\\
	\Gamma^{1,\alpha,\alpha,\beta}_{\tmim,\tmi}
	&=
		(\psi^{\alpha}_{\tmim,\tmi})^{\otimes 2}
		\odot
		\psi^{\beta}_{\tmim,\tmi},
	&
	l^{1,\alpha,\alpha,\beta}_{\tmim,\tmi}
	&=
		2^{-2mH}
		\psi^{\beta}_{\tmim,\tmi}.
\end{align*}
Then, for every $\alpha,\beta,\gamma$, which may be the same, we have
\begin{align*}
	B^{\alpha}_{\tmim,\tmi}B^{\beta}_{\tmim,\tmi}B^{\gamma}_{\tmim,\tmi}
	=
		I_3(\Gamma^{1,\alpha,\beta,\gamma}_{\tmim,\tmi})
		+I_1(l^{1,\alpha,\beta,\gamma}_{\tmim,\tmi}).
\end{align*}
For example, we see
\begin{align*}
	(B^{\alpha}_{\tmim,\tmi})^3
	&=
		2^{-3mH}
		\{
			(2^{mH}B^{\alpha}_{\tmim,\tmi})^3
			-
			3
			\cdot
			2^{mH}
			B^{\alpha}_{\tmim,\tmi}
	   \}
	   +
	   3
	   \cdot
	   2^{-2mH}B^{\alpha}_{\tmim,\tmi}\\
	&=
		2^{-3mH}I_3\big(\big(2^{mH}\psi^{\alpha}_{\tmim,\tmi}\big)^{\odot 3}\big)
		+3\cdot 2^{-2mH}I(\psi^{\alpha}_{\tmim,\tmi})\\
	&=
		I_3\big(\Gamma^{1,\alpha,\alpha,\alpha}_{\tmim,\tmi}\big)
		+I(l^{1,\alpha,\alpha,\alpha}_{\tmim,\tmi}).
\end{align*}
Furthermore, we have
\begin{align*}
	(
		\Gamma^{1,\alpha,\beta,\gamma}_{\tmim,\tmi},
		\Gamma^{1,\alpha,\beta,\gamma}_{\tmjm,\tmj}
	)_{(\mathcal{H}^d)^{\otimes 3}}
	&=
		C_{1,\alpha,\beta,\gamma}
		R([\tmim,\tmi]\times[\tmjm,\tmj])^3,\\
	(
		l^{1,\alpha,\beta,\gamma}_{\tmim,\tmi},
		l^{1,\alpha,\beta,\gamma}_{\tmjm,\tmj}
	)_{\mathcal{H}^d}
	&=
		C_{2,\alpha,\beta,\gamma}
		2^{-4mH}
		R([\tmim,\tmi]\times[\tmjm,\tmj]),
\end{align*}
where $C_{1,\alpha,\beta,\gamma}$ and $C_{2,\alpha,\beta,\gamma}$
are constants depending only on $\alpha,\beta,\gamma$.

\subsubsection{
The case $K^{m}_{\tmim,\tmi}=B^{\alpha,\beta}_{\tmim,\tmi}B^{\gamma}_{\tmim,\tmi}$
}

Because the cases
\begin{align*}
	B^{\alpha,\alpha}_{\tmim,\tmi}B^{\alpha}_{\tmim,\tmi}
	&=
		\frac{1}{2}(B^{\alpha}_{\tmim,\tmi})^3,
	&
	B^{\alpha,\alpha}_{\tmim,\tmi}B^{\beta}_{\tmim,\tmi}
	&=
		\frac{1}{2}
		(B^{\alpha}_{\tmim,\tmi})^2
		B^{\beta}_{\tmim,\tmi}
\end{align*}
have been considered and the identity
\begin{align*}
	B^{\beta,\alpha}_{\tmim,\tmi}B^{\alpha}_{\tmim,\tmi}
	&=
		-B^{\alpha,\beta}_{\tmim,\tmi}B^{\alpha}_{\tmim,\tmi}
		+(B^\alpha_{\tmim,\tmi})^2B^\beta_{\tmim,\tmi}
\end{align*}
holds, we consider the case of
\begin{align*}
	&B^{\alpha,\beta}_{\tmim,\tmi}B^{\alpha}_{\tmim,\tmi}, &
	&B^{\alpha,\beta}_{\tmim,\tmi}B^{\gamma}_{\tmim,\tmi}
\end{align*}
for different $\alpha,\beta,\gamma$.
Set
$
	\Gamma^{\alpha,\beta}_{s,t}
	=
		\int_s^t
			\psi^\alpha_{s,u}
			\odot
			d\psi^{\beta}_u
$.
Furthermore
\begin{align*}
	\Gamma^{2,\alpha,\beta,\gamma}_{\tmim,\tmi}
	&=
		\Gamma^{\alpha,\beta}_{\tmim,\tmi}
		\odot
		\psi^\gamma_{\tmim,\tmi},
	&
	l^{2,\alpha,\beta,\gamma}_{\tmim,\tmi}
	&=
		0,\\
	\Gamma^{2,\alpha,\beta,\alpha}_{\tmim,\tmi}
	&=
		\Gamma^{\alpha,\beta}_{\tmim,\tmi}
		\odot
		\psi^{\alpha}_{\tmim,\tmi},
	&
	l^{2,\alpha,\beta,\alpha}_{\tmim,\tmi}
	&=
		\int_{\tmim}^{\tmi}
			R([\tmim,u]	\times [\tmim,\tmi])
			d\psi^\beta_u.
\end{align*}
Then we have
\begin{align*}
	B^{\alpha,\beta}_{\tmim,\tmi}B^{\gamma}_{\tmim,\tmi}
&=
 I_3(\Gamma^{2,\alpha,\beta,\gamma}_{\tmim,\tmi}),\\
	B^{\alpha,\beta}_{\tmim,\tmi}B^{\alpha}_{\tmim,\tmi}
&=
 I_3(\Gamma^{\alpha,\beta}_{\tmim,\tmi}\odot \psi^\alpha_{\tmim,\tmi})
 +I_1(2\Gamma^{\alpha,\beta}_{\tmim,\tmi}\odot_1 \psi^\alpha_{\tmim,\tmi})\\
&=
 I_3(\Gamma^{2,\alpha,\beta,\alpha}_{\tmim,\tmi})
 +I_1(l^{2,\alpha,\beta,\alpha}_{\tmim,\tmi})
\end{align*}
and
\begin{align*}
	\Gamma^{\alpha,\beta}_{\tmim,\tmi}(n)
	&=
		\sum_{k=1}^n\psi^{\alpha}_{\tmim,u_{k-1}}
		\odot \psi^{\beta}_{u_{k-1},u_k},
	&
	\Gamma^{2,\alpha,\beta,\gamma}_{\tmim,\tmi}(n)
	&=
		\Gamma^{\alpha,\beta}_{\tmim,\tmi}(n)
		\odot
		\psi^\gamma_{\tmim,\tmi},\\
	\Gamma^{2,\alpha,\beta,\alpha}_{\tmim,\tmi}(n)
	&=
		\Gamma^{\alpha,\beta}_{\tmim,\tmi}(n)
		\odot
		\psi^{\alpha}_{\tmim,\tmi},
	&
	l^{2,\alpha,\beta,\alpha}_{\tmim,\tmi}(n)
	&=
		\sum_{k=1}^n
			R([\tmim,u_{k-1}]	\times [\tmim,\tmi])
			\psi^\beta_{u_{k-1},u_k},
\end{align*}
where $\{u_k\}_{k=0}^n$ is a partition of $[\tmim,\tmi]$.
Furthermore, we have
\begin{align*}
	(
		\Gamma^{2,\alpha,\beta,\gamma}_{\tmim,\tmi}(n),
		\Gamma^{2,\alpha,\beta,\gamma}_{\tmjm,\tmj}(n)
	)_{(\mathcal{H}^d)^{\otimes 3}}
	&=
		S_{\tmim,\tmi}(n),\\
	(
		\Gamma^{2,\alpha,\beta,\alpha}_{\tmim,\tmi}(n),
		\Gamma^{2,\alpha,\beta,\alpha}_{\tmjm,\tmj}(n)
	)_{(\mathcal{H}^d)^{\otimes 3}}
	&=
		S_{\tmim,\tmi}(n)
		+
		T_{\tmim,\tmi}(n),\\
	(
		l^{2,\alpha,\beta,\alpha}_{\tmim,\tmi}(n),
		l^{2,\alpha,\beta,\alpha}_{\tmjm,\tmj}(n)
	)_{\mathcal{H}^d}
	&=
		U_{\tmim,\tmi}(n).
\end{align*}
Here by letting $\{v_{k'}\}$ be a partition of $[\tmjm,\tmj]$, we set
\begin{align*}
	S_{\tmim,\tmi}(n)
	&=
		\frac{1}{6}
		R([\tmim,\tmi]\times[\tmjm,\tmj])\\
	&
	\qquad
	\qquad
		\times
		\sum_{k,k'=1}^n
			R([\tmim,u_{k-1}]\times[\tmjm,v_{k'-1}])
			R([u_{k-1},u_k]\times[v_{k'-1},v_{k'}]),\\
	T_{\tmim,\tmi}(n)
	&=
		\frac{1}{6}
		\sum_{k,k'=1}^n
			R([\tmjm,\tmj]\times[\tmim,u_{k-1}])
			R([\tmim,\tmi]\times[\tmjm,v_{k'-1}])\\
	&
	\qquad
	\qquad
	\qquad
	\qquad
			\times
			R([u_{k-1},u_k]\times[v_{k'-1},v_{k'}]),\\
	U_{\tmim,\tmi}(n)
	&=
			\sum_{k,k'=1}^n
			R([\tmim,\tmi]\times[\tmim,u_{k-1}])
			R([\tmjm,\tmj]\times[\tmjm,v_{k'-1}])\\
	&
	\qquad
	\qquad
	\qquad
	\qquad
			\times
			R([u_{k-1},u_k]\times[v_{k'-1},v_{k'}]).
\end{align*}

\subsubsection{The case $K^{m}_{\tmim,\tmi}=B^{\alpha,\beta,\gamma}_{\tmim,\tmi}$}
Let $\alpha,\beta,\gamma$ be three different integers.
Because the case
$
	B^{\alpha,\alpha,\alpha}_{\tmim,\tmi}
	=
		\frac{1}{6}(B^{\alpha}_{\tmim,\tmi})^3
$
has been considered
and the identities
\begin{align*}
B^{\alpha,\beta,\alpha}_{\tmim,\tmi}
   &=
	   B^{\alpha,\beta}_{\tmim,\tmi}B^{\alpha}_{\tmim,\tmi}
	   -2B^{\alpha,\alpha,\beta}_{\tmim,\tmi},\\
	B^{\beta,\alpha,\alpha}_{\tmim,\tmi}
	&=
		\frac{1}{2}
		(B^{\alpha}_{\tmim,\tmi})^2
		B^{\beta}_{\tmim,\tmi}
		-B^{\alpha,\beta}_{\tmim,\tmi}B^{\alpha}_{\tmim,\tmi}
		+B^{\alpha,\alpha,\beta}_{\tmim,\tmi}
\end{align*}
hold (use \eqref{eq_498311}), we consider other cases.
First, set
\begin{align*}
	\Gamma^{3,\alpha,\beta,\gamma}_{\tmim,\tmi}
	&=
	\int_{\tmim}^{\tmi}
		\left(
			\int_{\tmim}^v
				\psi^{\alpha}_{\tmim,u}
				\odot
				d\psi^{\beta}_u
		\right)
		\odot
		d\psi^{\gamma}_v,
	&
	l^{3,\alpha,\beta,\gamma}_{\tmim,\tmi}
	&=
		0,\\
	\Gamma^{3,\alpha,\alpha,\beta}_{\tmim,\tmi}
	&=
	\int_{\tmim}^{\tmi}
		\psi^{\alpha}_{\tmim,u}
		\odot
		\psi^{\alpha}_{\tmim,u}
		\odot
		d\psi^{\beta}_u,
	&
	l^{3,\alpha,\alpha,\beta}_{\tmim,\tmi}
	&=
		\frac{1}{2}
		\int_{\tmim}^{\tmi}
			(u-\tmim)^{2H}
			d\psi^\beta_u.
\end{align*}
Then we have
\begin{align*}
 B^{\alpha,\beta,\gamma}_{\tmim,\tmi}&=
I_3(\Gamma^{3,\alpha,\beta,\gamma}_{\tmim,\tmi}), \quad\quad
B^{\alpha,\alpha,\beta}_{\tmim,\tmi}=
I_3(\Gamma^{3,\alpha,\alpha,\beta}_{\tmim,\tmi})
+I_1(l^{3,\alpha,\alpha,\beta}_{\tmim,\tmi})
\end{align*}
and
\begin{align*}
 \Gamma^{3,\alpha,\beta,\gamma}_{\tmim,\tmi}(n)&=
\sum_{1\le k<l\le n}\psi^{\alpha}_{\tmim, u_{k-1}}\odot
\psi^{\beta}_{u_{k-1},u_k}\odot \psi^{\gamma}_{u_{l-1},u_l},\\
\Gamma^{3,\alpha,\alpha,\beta}_{\tmim,\tmi}(n)&=
\sum_{1\le k\le n}
\psi^{\alpha}_{\tmim,u_{k-1}}\odot\psi^{\alpha}_{\tmim,u_{k-1}}
\odot \psi^{\beta}_{u_{k-1},u_k},\\
l^{3,\alpha,\alpha,\beta}_{\tmim,\tmi}(n)&=
\sum_{1\le k\le n}
(u_{k-1}-\tmim)^{2H}\psi^{\beta}_{u_{k-1},u_k}.
\end{align*}
We have
\begin{equation}
\begin{aligned}
&
	(
		\Gamma^{3,\alpha,\beta,\gamma}_{\tmim,\tmi}(n),
		\Gamma^{3,\alpha,\beta,\gamma}_{\tmjm,\tmj}(n)
	)_{(\mathcal{H}^d)^{\otimes 3}}\\
&\qquad\qquad
	=
		\frac{1}{6}
		\sum_{\substack{1\le k<l\le n,\\1\le k'<l'\le n}}
			R([\tmim,u_{k-1}]\times[\tmjm,v_{k'-1}])\\
&\qquad\qquad\qquad\qquad\qquad\qquad\qquad
			\times
			R([u_{k-1},u_{k}]\times[v_{k'-1},v_{k'}])
			R([u_{l-1},u_{l}]\times[v_{l'-1},v_{l'}]),\label{double integral}\\
&
	(
		\Gamma^{3,\alpha,\alpha,\beta}_{\tmim,\tmi}(n),
		\Gamma^{3,\alpha,\alpha,\beta}_{\tmjm,\tmj}(n)
	)_{(\mathcal{H}^d)^{\otimes 3}}\\
&\qquad\qquad
  	=
		\frac{1}{3}
		\sum_{k,k'=1}^n
			R([\tmim,u_{k-1}]\times[\tmjm,v_{k'-1}])^2
			R([u_{k-1},u_k]\times[v_{k'-1},v_{k'}]),\\
&
	(
		l^{3,\alpha,\alpha,\beta}_{\tmim,\tmi}(n), 
		l^{3,\alpha,\alpha,\beta}_{\tmjm,\tmj}(n)
	)_{\mathcal{H}^d}\\
&\qquad\qquad
	=
		\sum_{k,k'=1}^n
			(u_{k-1}-\tmim)^{2H}
			(v_{k'-1}-\tmjm)^{2H}
			R([u_{k-1},u_k]\times[v_{k'-1},v_{k'}]).
\end{aligned}
\end{equation}

It is necessary to clarify why we are able to obtain the expansion formula
for the iterated integral
$B^{\alpha,\alpha,\beta}_{s,t}$ above.
By definition, we have
\begin{align}
 B^{\alpha,\alpha,\beta}_{s,t}&=\int_s^t\frac{1}{2}(B^\alpha_{s,u})^2
dB^{\beta}_u
=\lim_{n\to\infty}\left(\sum_{i=1}^n\frac{1}{2}(B^\alpha_{s,t_{i-1}})^2
B^{\beta}_{t_{i-1},t_i}+
\sum_{i=1}^nB^{\alpha}_{s,t_{i-1}}B^{\alpha,\beta}_{t_{i-1},t_i}\right),
\label{Balphaalphabeta}
\end{align}
where $\{t_i\}$ is a partition of $[s,t]$.
Also, we have
\begin{multline*}
	E
		\Bigg[
			\Bigg(
				\sum_{i=1}^n
					B^{\alpha}_{s,t_{i-1}}
					B^{\alpha,\beta}_{t_{i-1},t_i}
			\Bigg)^2
		\Bigg]
	=
		\sum_{i,j=1}^n
			\int_{t_{i-1}\le u\le t_i, t_{j-1}\le v\le t_j}
				E
					[
						B^{\alpha}_{s,t_{i-1}}
						B^{\alpha}_{s,t_{j-1}}
						B^{\alpha}_{t_{i-1},u}
						B^{\alpha}_{t_{j-1},v}
					]\\
					\times
					dR([t_{i-1},u]\times [t_{j-1},v]).
\end{multline*}
Using this, Lemma~\ref{properties of R} (2), Wick's formula for 
the expectation of the 
product of Gaussian random variables and 
by a similar calculation to (\ref{L2limit}),
we can prove 
\begin{align*}
  \lim_{n\to\infty}\sum_{i=1}^n
  B^{\alpha}_{s,t_{i-1}}B^{\alpha,\beta}_{t_{i-1},t_i}=0
  \qquad\qquad
  \text{in $L^2$.}
\end{align*}
Therefore, we need to consider the first term only in
(\ref{Balphaalphabeta}).
This leads to the expansion formula above.

\begin{proof}[Proof of Lemma \textup{\ref{covariance estimate of Km}}]
By the identities obtained in this subsection,
using Theorem~\ref{Towghi}, Lemmas~\ref{cor to towghi} and~\ref{p-variation norm for product},
we see that \eqref{covariance estimate of Km} holds for any
$(K^{m}_t)\in \mathcal{K}^3_m$.
For example, we have
\begin{align*}
	|
		(
			\Gamma^{2,\alpha,\beta,\alpha}_{\tmim,\tmi}(n),
			\Gamma^{2,\alpha,\beta,\alpha}_{\tmjm,\tmj}(n)
		)_{(\mathcal{H}^d)^{\otimes 3}}
	|
	&\le
		C2^{-6mH}|\rho_H(i-j)|^3,\\
	|
		(
			\Gamma^{3,\alpha,\alpha,\beta}_{\tmim,\tmi}(n),
			\Gamma^{3,\alpha,\alpha,\beta}_{\tmjm,\tmj}(n)
		)_{(\mathcal{H}^d)^{\otimes 3}}
	|
	&\le
		C2^{-6mH}
		|\rho_H(i-j)|^3,\\
	|
		(
			\Gamma^{3,\alpha,\beta,\gamma}_{\tmim,\tmi}(n),
			\Gamma^{3,\alpha,\beta,\gamma}_{\tmjm,\tmj}(n)
		)_{(\mathcal{H}^d)^{\otimes 3}}
	|
	&\le
		C2^{-6mH}
		|\rho_H(i-j)|^3,\\
	|
		(
			l^{3,\alpha,\alpha,\beta}_{\tmim,\tmi}(n), 
			l^{3,\alpha,\alpha,\beta}_{\tmjm,\tmj}(n)
		)_{\mathcal{H}^d}
	|
	&\le
		C 2^{-6mH}|\rho_H(i-j)|.
\end{align*}
The estimates for other terms are similar to the above.
We should note that the sum appeared in
$
	(
		\Gamma^{3,\alpha,\beta,\gamma}_{\tmim,\tmi}(n),
		\Gamma^{3,\alpha,\beta,\gamma}_{\tmjm,\tmj}(n)
	)_{(\mathcal{H}^d)^{\otimes 3}}
$
is a double discrete Young integral
and different from other terms.
In the estimate of 
$
	(
		\Gamma^{3,\alpha,\alpha,\beta}_{\tmim,\tmi}(n),
		\Gamma^{3,\alpha,\alpha,\beta}_{\tmjm,\tmj}(n)
	)_{(\mathcal{H}^d)^{\otimes 3}}
$,
we apply Lemma~\ref{p-variation norm for product} (2).

 \end{proof}

\appendix

\section{Multidimensional Young integral}\label{appendix I}

First, we recall basic definitions and results concerning
multidimensional Young integrals.
We next explain some more auxiliary results for our study.

\subsection{Definitions and basic results}
Let $0\le s_r<t_r\le 1$ $(1\le r\le N)$ and
set $I=\prod_{r=1}^N[s_r,t_r]$.
We call $\cP=\cP_1\times\cdots\times \cP_N$ a grid-like partition of $I$,
where $\cP_r:s_r=t^r_{0}<\cdots<t^r_{m_r}=t_r$ is a partition of $[s_r,t_r]$
for every $1\le r\le N$.
We denote the all functions defined on the partition points
$(t^1_{i_1},\ldots,t^N_{i_N})$ of $\cP$ by $C(I_\cP)$.

Here we define notion for functions $f\in C(I_\cP)$.
For $u_i\in \cP_i$ ($1\le i\le k$) and
$u^1_i<u^2_i$ ($u^1_i,u^2_i\in \cP_i$, $k+1\le i\le N$),
we define
\begin{multline}
	\label{increment function}
	f
		\left(
			u_1,\ldots,u_k,
			[u^1_{k+1},u^2_{k+1}]\times\cdots\times[u^1_N,u^2_N]
		\right)\\
	=
		\sum_{\sigma_j=1,2,k+1\le j\le N}
			(-1)^{\sum_{j={k+1}}^N\sigma_j}
			f
				\left(
					u_1,\ldots,u_k,u^{\sigma_{k+1}}_{k+1},\ldots,u^{\sigma_N}_N
				\right).
\end{multline}
Let $\cP'_r$ be a partition of
$I$ whose all partition points are included in the partition points
of $\cP_r$.
We call the grid-like partition defined by
$\cP'=\cP_1'\times\cdots\times \cP_N'$ a sub-partition of $\cP$.
Note $f\in C(I_\cP)$ implies $f|_{I_{\cP'}}\in C(I_{\cP'})$.
For a grid-like partition $\cP=\cP_1\times\cdots\times \cP_N$, $p\ge 1$ and
$f\in C(I_{\cP})$,
we define
\begin{align*}
	\tilde{V}_p(f ; I_\cP)
	&=
		\left\{
			\sum_{i_1=1}^{m_1}
			\cdots
			\sum_{i_N=1}^{m_N}
				|f([t^1_{i_1-1},t^1_{i_1}]\times\cdots\times[t^N_{i_N-1},t^N_{i_N}])|^p
		\right\}^{1/p},\\
	V_p(f ; I_{\cP})
 	&=
		\max
			\big\{
				\tilde{V}_p(f|_{I_{\cP'}}; I_{\cP'})~\big|~
				\text{$\cP'$ moves in the set of all sub-partitions of $\cP$}
			\big\}.
\end{align*}
Let $A=\{n_1,\ldots,n_l\}$ ($n_1<\cdots<n_l$)
be a non-empty subset of $\{1,\ldots,N\}$.
Let us define a function $f(s_r;r\in A^c)$ on
$(\prod_{r\in A}[s_r,t_r])_{\prod_{r\in A}\cP_r}$ 
which is a product space $\prod_{r\in A}[s_r,t_r]$ with the partition
$\prod_{r\in A}\cP_r$ by
\begin{align}
	\label{eq489108901}
	f(s_r;r\in A^c)(u_{n_1},\ldots,u_{n_l})
	=
		f(u_1,\ldots,u_N)|_{u_r=s_r,\, r\in A^c}.	
\end{align}
We may write
\begin{align}
	\label{eq4789317891}
	f(s_r;r\in A^c)(u_{n_1},\ldots,u_{n_l})
	=
		f(s_r,u_a;r\in A^c,a\in A).
\end{align}
When $s_r=0$ for all $r\in A^c$, we write $f(0_r;r\in A^c)$.
When $N=4$ and $A=\{1,3\}$, we have
$f(s_r;r\in \{2,4\})(u_1,u_3)=f(u_1,s_2,u_3,s_4)$ for $(u_1,u_3)\in \cP_1\times\cP_3$.
We define
\begin{align*}
	\bar{V}_p(f; I_\cP)
	&=
		\sum_{A\subset \{1,\ldots,N\}}
			V_p\left(f(s_a;a\in A^c) ; \left(\prod_{r\in A}[s_r,t_r]\right)_{\prod_{r\in A}\cP_r}\right)
		+
		|f(s_1,\ldots,s_N)|.
\end{align*}

Next we define notion for continuous functions $f\in C(I)$.
For $u_i\in [s_i,t_i]$ ($1\le i\le k$) and
$u^1_i<u^2_i$ ($u^1_i,u^2_i\in[s_i,t_i]$, $k+1\le i\le N$),
we define
$
	f
	(
		u_1,\ldots,u_k,
		[u^1_{k+1},u^2_{k+1}]\times\cdots\times[u^1_N,u^2_N]
	)
$
similarly to \eqref{increment function}.
For a continuous function $f\in C(I)$,
the $p$-variation norm on $I$ is defined by
\begin{align*}
	V_p(f ; I)
	&=
		\sup
			\big\{
				\tilde{V}_p(f|_{I_\cP}; I_\cP)~\big|~
				\text{$\cP$ moves all grid-like partition of $I$}
			\big\},\\
	\bar{V}_p(f; I)
	&=
		\sum_{A\subset \{1,\ldots,N\}}
			V_p\left(f(s_a;a\in A^c) ; \prod_{r\in A}[s_r,t_r]\right)
		+
		|f(s_1,\ldots,s_N)|.
\end{align*}
Unlike the one-dimensional case, the functional
$I\mapsto V_p(f ; I)^p$ is not superadditive generally.
The controlled $p$-variation norm satisfies such a satisfactory property.
The controlled $p$-variation norm $\|f\|_{p\hyp var, I}$
of the continuous function $f$ on
$I$ is defined as follows.
\begin{align*}
 \|f\|_{p\hyp var, I}&=\sup\Biggl\{
\left(\sum_{k=1}^K|f(I_k)|^p\right)^{1/p}~\Bigg |~
\text{$I=\cup_{k=1}^KI_k$, where $I_k=\prod_{r=1}^N[s^k_r,t^k_r]\subset I$ 
and}\\
&\qquad\text{
$I_k\cap I_l$ $(k\ne l)$ is included in their boundaries and
$1\le K<\infty$}
\Biggr\}.
\end{align*}

The following theorem is important for clarifying the relation between 
the two norms above.
See Friz-Victoir~(\cite{friz-victoir2011}).

\begin{theorem}\label{FV}
 Let $I$ be a rectangle in $[0,1]^N$.
Then for any $p\ge 1$ and $\vep>0$, there exists $C_{\vep,p}$
such that
\begin{align*}
 C_{\vep,p}\|f\|_{(p+\vep)\hyp var, I}\le V_p(f ; I)
\le \|f\|_{p\hyp var}.
\end{align*}
\end{theorem}
For $f, g\in C(I_\cP)$, we define
\begin{align*}
	\int_{I_\cP}f(u_1,\ldots,u_n)dg(u_1,\ldots,u_n)
	&=
		\sum_{i_1=1}^{m_1}
		\cdots
		\sum_{i_N=1}^{m_N}
			f(t^1_{i_1-1},\ldots,t^N_{i_N-1})
			g
				\left(
					[t^1_{i_1-1},t^1_{i_1}]
					\times
					\cdots
					\times
					[t^N_{i_N-1},t^N_{i_N}]
				\right).
\end{align*}

The following theorem is due to Towghi (\cite{towghi}).

\begin{theorem}\label{Towghi}
Let $p,q$ be positive numbers satisfying $\frac{1}{p}+\frac{1}{q}>1$.
Let $\cP=\cP_1\times\cdots\times \cP_N$ be a grid-like partition of $I$.
 Let $f\in C(I_{\cP})$ and $g\in C(I_{\cP})$.
The following constants $C$ depend only on $p,q$ and $N$.
\begin{enumerate}
 \item[$(1)$] It holds that
\begin{align*}
\left|\int_{I_\cP}
f(u_1,\ldots,u_N)dg(u_1,\ldots,u_N)\right|
&\le C\bar{V}_p(f ; I_\cP)
V_q(g ; I_\cP).
\end{align*}
\item[$(2)$] If $f(\cdots,s_r,\cdots)=0$ for all $1\le r\le N$,
then 
\begin{align*}
\left|\int_{I_\cP}
f(u_1,\ldots,u_N)dg(u_1,\ldots,u_N)\right|
&\le C V_p(f ; I_\cP)
V_q(g ; I_\cP).
\end{align*}
\end{enumerate}
\end{theorem}

\begin{remark}
	By applying the theorem presented above,
we see that for any $f$ and $g$ which satisfy
$\bar{V}_p(f ; I)<\infty$ and $V_q(g ; I)<\infty$ with
$\frac{1}{p}+\frac{1}{q}>1$
the limit
\begin{align*}
 \lim_{|\cP|\to 0}\int_{I_\cP}f(u_1,\ldots,u_n)dg(u_1,\ldots,u_n)
\end{align*}
exists and the limit is called the Young integral
of $f$ against $g$ and we denote the limit by
\begin{align*}
	\int_{I}f(u_1,\ldots,u_n)dg(u_1,\ldots,u_n).
\end{align*}
\end{remark}

\subsection{Auxiliary results}

Next, we collect necessary results used in this paper.
We apply the following lemma to estimate the sum of
(\ref{double integral}).

\begin{lemma}\label{cor to towghi}
	Let $I=[s_1,t_1]\times [s_2,t_2]\subset [0,1]^2$ and
	let $\cP=\cP_1\times \cP_2$ be a grid-like partition,
	where $\cP_1:s_1=u_0<\cdots<u_n=t_1$ and $\cP_2:s_2=v_0<\cdots<v_{m}=t_2$.
	We write $I_{i,j}=[u_{i-1},u_i]\times [v_{j-1},v_j]$ 
	for every $1\le i\le n$ and $1\le j\le m$.

Let $p,q,q'$ be non-negative numbers satisfying
$p>1$, $q'>q>1$ and
$\frac{1}{p}+\frac{1}{q'}>1$.
Let $f\in C(I_{\cP})$.
Let $g\in C(I)$ and suppose $V_q(g ; I)<\infty$.
We define $h\in C(I_\cP)$ by
$
	h(s_1,\cdot)
	=
		h(\cdot,s_2)
	=
		0
$
and
\begin{align*}
	h(u_i,v_j)
	=
		\sum_{k=1}^i
		\sum_{l=1}^j
			f(u_{k-1},v_{l-1})
			g(I_{k,l})
	\qquad
	\qquad
	\text{for \quad $1\le i\le n$\quad and\quad $1\le j\le m$}.
\end{align*}
Then we have
$
	V_{q'}(h ; I_{\cP})
	\le
		C \bar{V}_p(f; I_\cP)V_q(g; I)
$.
\end{lemma}

\begin{proof}
	Let $0\le a<a'\le n$ and $0\le b<b'\le m$.
	Write $J=[u_a,u_{a'}]\times[v_b,v_{b'}]$.
	Consider
	\begin{align*}
		h(J)
		=
			\sum_{k=a+1}^{a'}
			\sum_{l=b+1}^{b'}
				f(u_{k-1},v_{l-1})
				g(I_{k,l}).
	\end{align*}
	Then the right-hand side is the discrete Young integral,
	and Theorem~\ref{Towghi} (1) implies
	\begin{align*}
        |h(J)|
        &\le
            C
            \bar{V}_p(f;J_\cP)
            V_{q'}(g;J_\cP).
	\end{align*}
	Using
	$
		\bar{V}_p(f;J_\cP)
		\le
		C
			\bar{V}_p(f;I_\cP)
	$, which we will show after,
	and 
	$
		V_{q'}(g;J_\cP)
		\leq
			V_{q'}(g;J)
		\leq
		\|g\|_{q'\hyp var,J}
	$, which follows from Theorem~\ref{FV},
	we have
	\begin{align*}
		|h(J)|
        &\le
            C
            \bar{V}_p(f;I_\cP)
            \|g\|_{q'\hyp var,J}.
	\end{align*}
	Using this inequality and by applying Theorem~\ref{FV},
	we arrived at the assertion.

	Here we show that
	$
		\bar{V}_p(f;J_\cP)
		\le
		C
			\bar{V}_p(f;I_\cP)
	$.
	Let $b=\beta_0<\dots<\beta_N=b'$
	and consider a partition $v_b=v_{\beta_0}<\cdots<v_{\beta_N}=v_{b'}$ of 
	$[v_b,v_{b'}]$.
	Because
	\begin{align*}
		|f(u_a, [v_{\beta_{i-1}},v_{\beta_i}])|^p
		&=
			|f([u_0,u_a]\times[v_{\beta_{i-1}},v_{\beta_i}])+f(s_1,[v_{\beta_{i-1}},v_{\beta_i}])|^p\\
		&\leq
			2^{p-1}
			\{|f([u_0,u_a]\times[v_{\beta_{i-1}},v_{\beta_i}])|^p+|f(s_1,[v_{\beta_{i-1}},v_{\beta_i}])|^p\},
	\end{align*}
	we have
	\begin{align*}
		\sum_{i=1}^N
			|f(u_a, [v_{\beta_{i-1}},v_{\beta_i}])|^p
		\leq
			2^{p-1}
			\{
				V_p(f;I_\cP)^p
				+
				V_p(f(s_1,\cdot);[s_2,t_2]|_{\cP_2})^p
			\}
		\leq
			C\bar{V}_p(f;I_\cP)^p,
	\end{align*}
	which implies
	\begin{align*}
		V_p(f(u_a,\bullet);[v_b,v_{b'}]|_{\cP_2})
		\leq
			C\bar{V}_p(f;I_\cP)^p.
	\end{align*}
	Since $|f(u_a,v_b)|$, 
	$V_p(f(\bullet,v_b);[u_a,u_{a'}]|_{\cP_1})$,
	and $V_p(f;J|_{\cP})$
	have similar bounds, we see the assertion.
    This completes the proof.
\end{proof}

The following lemma is used in Lemma~\ref{psi^{(l)}}.
\begin{lemma}\label{for psi^{(l)}}
	Let $I=[s_1,t_1]\times [s_2,t_2]\subset [0,1]^2$ and
	let $\cP=\cP_1\times \cP_2$ be a grid-like partition,
	where $\cP_1:s_1=u_0<\cdots<u_n=t_1$ and $\cP_2:s_2=v_0<\cdots<v_{m}=t_2$.
	We write $I_{i,j}=[u_{i-1},u_i]\times [v_{j-1},v_j]$ 
	for every $1\le i\le n$ and $1\le j\le m$.

	Let $p,q$ be non-negative numbers satisfying $p>1$, $q>1$ and $\frac{1}{p}+\frac{1}{q}>1$.
	Let $f,g\in C(I)$ satisfy
$V_p(f; I)<\infty$,
$f(s_1,\cdot)=f(\cdot,s_2)=0$ and $V_q(g ; I)<\infty$.
Let $\tf\in C(I_{\cP})$ with
$\tf(s_1,\cdot)=\tf(\cdot,s_2)=0$.
We define $h,\tilde{h}\in C(I_{\cP})$ by
$
	h(s_1,\cdot)
	=
		h(\cdot,s_2)
	=
		\tilde{h}(s_1,\cdot)
	=
		\tilde{h}(\cdot,s_2)
	=
		0
$
and 
\begin{align*}
	h(u_i,v_j)
	&=
		\sum_{k=1}^i
		\sum_{l=1}^j
			f(u_{k-1},v_{l-1})
			g(I_{k,l}),
	&
	\tilde{h}(u_i,v_{j})
	&=
		\sum_{k=1}^i
		\sum_{l=1}^j
			\tf(u_{k-1},v_{l-1})
			g(I_{k,l})
\end{align*}
for $1\le i\le n$ and $1\le j\le m$.

Suppose $p'>p$ and $q'>q$ satisfy $\frac{1}{p'}+\frac{1}{q'}>1$.
Set $\theta'=\frac{1}{p'}+\frac{1}{q'}$.
Then, for $q<q''<q'$, we have
\begin{align}
	\label{h-youngintegral}
	V_{q'}\left(h-\int_{[s_1,\cdot]\times [s_2,\cdot]}f(u,v)dg(u,v) ;I_{\cP}\right)
	\le
		C
		\mathfrak{V}_{p'}(f,I_{\cP})^{1-\frac{1}{\theta'}}
		V_p(f ; I)^{\frac{1}{\theta'}}
		V_{q''}(g ; I),
\end{align}
and
\begin{align}
	\label{h-tildeh}
	V_{q'}(h-\tilde{h} ; I_{\cP})
	\le
		C
		V_p(f-\tf ; I_{\cP})
		V_{q''}(g ; I).
\end{align}
Here
\begin{align*}
	\mathfrak{V}_{p'}(f,I_{\cP})
	=
		\max_{\substack{1\leq k\leq n,\\1\leq l\leq m}}
			\big\{
				V_{p'}(f; I_{k,l})
				+V_{p'}(f(\cdot, v_{l-1}); [u_{k-1},u_k])
				+V_{p'}(f(u_{k-1},\cdot); [v_{l-1},v_l])
			\big\}
\end{align*}

\end{lemma}

\begin{proof}
	We prove (\ref{h-youngintegral}).
	Let $0\le a<a'\le n$ and $0\le b<b'\le m$.
	Write $J=[u_a,u_{a'}]\times[v_b,v_{b'}]$.
	Set
	\begin{align*}
		F(J)
		=
			\int_J f(u,v)dg(u,v)
			-
			h(J)
		=
			\sum_{k=a+1}^{a'}
			\sum_{l=b+1}^{b'}
				\int_{I_{k,l}}
					\{f(u,v)-f(u_{k-1},v_{l-1})\}
					dg(u,v).
	\end{align*}
	In the following, we will show
	\begin{align}
		\label{eq391091381}
		|F(J)|
		\leq
			C
			\mathfrak{V}_{p'}(f,I_{\cP})^{1-\frac{1}{\theta'}}
			\|f\|_{p'\hyp var,I}^{\frac{1}{\theta'}}
			\|g\|_{q'\hyp var;J}.
	\end{align}
	Using this inequality and by applying Theorem~\ref{FV},
	we obtain \eqref{h-youngintegral}.

	First, Theorem~\ref{Towghi} (1) implies that
	\begin{align*}
		|F(J)|
		&\leq
			C 
			\sum_{k=a+1}^{a'}
			\sum_{l=b+1}^{b'}
				\{
					V_{p'}(f;I_{k,l})
					+V_{p'}(f(u_{k-1},\cdot);[v_{l-1,l},v_l])\\
		&
		\qquad
		\qquad
		\qquad
		\qquad
		\qquad
		\qquad
		\qquad
					+V_{p'}(f(\cdot,v_{l-1});[u_{k-1,l},u_k])
				\}
				V_{q'}(g; I_{k,l}).
	\end{align*}
	Next we give estimates of the summations.
	The H{\"o}lder inequality for the summation with respect to $k$ and $l$ implies
	\begin{multline*}
		\sum_{k=a+1}^{a'}
		\sum_{l=b+1}^{b'}
			V_{p'}(f;I_{k,l})
			V_{q'}(g; I_{k,l})
		\leq
			\mathfrak{V}_{p'}(f,I_{\cP})^{1-\frac{1}{\theta'}}
			\sum_{k=a+1}^{a'}
			\sum_{l=b+1}^{b'}
				V_{p'}(f;I_{k,l})^{\frac{1}{\theta'}}
				V_{q'}(g; I_{k,l})\\
		\begin{aligned}
			&\leq
				\mathfrak{V}_{p'}(f,I_{\cP})^{1-\frac{1}{\theta'}}
				\Bigg\{
					\sum_{k=a+1}^{a'}
					\sum_{l=b+1}^{b'}
						\|f\|_{p'\hyp var,I_{k,l}}^{p'}
				\Bigg\}^{\frac{1}{p\theta'}}
				\Bigg\{
					\sum_{k=a+1}^{a'}
					\sum_{l=b+1}^{b'}
						\|g\|_{q'\hyp var; I_{k,l}}^{q'\theta'}
				\Bigg\}^{\frac{1}{q'\theta'}}\\
			&\leq
				\mathfrak{V}_{p'}(f,I_{\cP})^{1-\frac{1}{\theta'}}
				\|f\|_{p'\hyp var,I}^{\frac{1}{\theta'}}
				\|g\|_{q'\hyp var,J}.
		\end{aligned}
	\end{multline*}
	Here, we should note
	\begin{align*}
		V_{p'}(f(u_{k-1},\cdot);[v_{l-1},v_l])
		\leq
			V_{p'}(f,[s_1,t_1]\times [v_{l-1},v_l])
		\leq
			\|f\|_{p'\hyp var,[s_1,t_1]\times [v_{l-1},v_l]},
	\end{align*}
	which follows from
	\begin{align*}
		\sum_{j=1}^M
			|f(u_{k-1},[\eta_{j-1},\eta_j])|^{p'}
		=
			\sum_{j=1}^M
				|f([s_1,u_{k-1}]\times[\eta_{j-1},\eta_j])|^{p'}
		\leq
			V_{p'}(f,[s_1,u_{k-1}]\times [v_{l-1},v_l])
	\end{align*}
	for $v_{l-1}=\eta_0<\dots<\eta_M=v_l$.
	The H{\"o}lder inequality for the summation with respect to $l$ implies
	\begin{multline*}
		\sum_{k=a+1}^{a'}
		\sum_{l=b+1}^{b'}
			V_{p'}(f(u_{k-1},\cdot);[v_{l-1},v_l])
			V_{q'}(g; I_{k,l})\\
		\begin{aligned}
			&\leq
				\mathfrak{V}_{p'}(f,I_{\cP})^{1-\frac{1}{\theta'}}
				\sum_{k=a+1}^{a'}
					\Bigg\{
						\sum_{l=b+1}^{b'}
							\|f\|_{p'\hyp var,[s_1,t_1]\times [v_{l-1},v_l]}^{p'}
					\Bigg\}^{\frac{1}{p\theta'}}
					\Bigg\{
						\sum_{l=b+1}^{b'}
							\|g\|_{q'\hyp var; I_{k,l}}^{q'\theta'}
					\Bigg\}^{\frac{1}{q'\theta'}}\\
			&\leq
				\mathfrak{V}_{p'}(f,I_{\cP})^{1-\frac{1}{\theta'}}
				\sum_{k=a+1}^{a'}
					\|f\|_{p'\hyp var,I}^{\frac{1}{\theta'}}
					\|g\|_{q'\hyp var;[u_{k-1},u_k]\times[v_b,v_{b'}]}\\
			&\leq
				\mathfrak{V}_{p'}(f,I_{\cP})^{1-\frac{1}{\theta'}}
				\|f\|_{p'\hyp var,I}^{\frac{1}{\theta'}}
				\|g\|_{q'\hyp var;J}
		\end{aligned}
	\end{multline*}
	The summand of $V_{p'}(f(\cdot,v_{l-1});[u_{k-1,l},u_k])V_{q'}(g; I_{k,l})$ has the same bound.
	Therefore \eqref{eq391091381} is shown.

	We see \eqref{h-tildeh} follows from Theorem~\ref{Towghi} (2) and Theorem~\ref{FV}.
\end{proof}

Next, we prepare some more Propositions.
We apply Proposition~\ref{a multidimensional young integral} and 
Proposition~\ref{estimate of fgh} to the estimate of the
Malliavin derivatives of
the functional of $Y_t,J_t,J^{-1}_t$ in 
Section~\ref{weighted hermite variation}.

\begin{proposition}\label{a multidimensional young integral}
Let $w=w(s,t)$ $0\le s\le t\le 1$ be a control function.
Let $p,q$ be positive numbers satisfying 
$\theta:=\frac{1}{p}+\frac{1}{q}>1.$
Let $I=\prod_{r=1}^N [s_r,t_r]\times \prod_{r=1}^N [s_r,t_r]\subset [0,1]^{2N}$ and
$\cP=\cP_1\times\cdots\times \cP_N\times\cP_1\times\cdots\times \cP_N$ be a grid-like partition of $I$,
where $\cP_r\colon s_r=t^r_0<\dots<t^r_{m_r}=t_r$.
Furthermore, assume $\phi\in C(I_\cP)$ satisfies that there exists a positive constant
$C$ such that the following condition holds:
	\begin{align*}
		\left|
			\phi
				\left(
					\prod_{r=1}^N
						[u_r,u'_r]
					\times
					\prod_{r=1}^N
						[v_r,v'_r]
				\right)
		\right|
		\leq
			C
			\prod_{r=1}^N
				\{
					w(u_r,u'_r)^{\frac{1}{p}}
					w(v_r,v'_r)^\frac{1}{q}
				\}
	\end{align*}
	for all $u_r, u'_r, v_r,v'_r\in \cP_r$ with
$u_r<u'_r$ and $v_r<v'_r$ $(1\leq r\leq N).$

Then we have
	\begin{align}
		\label{eq8940181}
		\left|
			\sum_{i_1=1}^{m_1}
			\dots
			\sum_{i_N=1}^{m_N}
				\phi
					\left(
						\prod_{r=1}^N
							[s_r,t^r_{i_r-1}]
						\times
						\prod_{r=1}^N
							[t^r_{i_r-1},t^r_{i_r}]
					\right)
		\right|
		\leq
			C
			\zeta(\theta)^N
			\prod_{r=1}^N
				w(s_r,t_r)^\theta,
	\end{align}
	where $C$ is the same constant as the one appearing in the assumption on
$\phi$,
	$\zeta$ is the zeta function.

\end{proposition}

\begin{proof}
We first prove the case where $N=1$.
We write
\begin{align*}
 \cP&=\{s_1=t^1_0<\cdots<t^1_{m_1}=t_1\},\\
\cP\setminus\{i\}&=\{s_1=t_0<\cdots<t^1_{i-1}<t^1_{i+1}<
\cdots<t^1_{m_1}=t_r\},\quad 1\le i\le m_1-1,
\end{align*}
and set
\begin{align*}
 I(\cP)&=\sum_{i=1}^{m_1}\phi\left([s_1,t^1_{i-1}]
\times [t^1_{i-1},t^1_i]\right).
\end{align*}
Note that $I(\{s,t\})=0$ by definition because $[s_1,t^1_0]=\{s_1\}$.
Then using the assumption on $\phi$,
\begin{align*}
	|I(\cP)-I(\cP\setminus\{i\})|
	&=
		|
			\phi([s_1,t^1_{i-1}]\times [t^1_{i-1},t^1_i])
			+\phi([s_1,t^1_{i}]\times [t^1_{i},t^1_{i+1}])
			-\phi([s_1,t^1_{i-1}]\times [t^1_{i-1},t^1_{i+1}])
		|\\
	&=
		|\phi([t^1_{i-1},t^1_i]\times [t^1_{i},t^1_{i+1}])|\\
	&\le
		Cw(t^1_{i-1},t^1_{i})^{1/p}w(t^1_i,t^1_{i+1})^{1/q}.
\end{align*}
Therefore, using the H\"older inequality and the property of 
the control function, we obtain
\begin{align*}
 \sum_{i=1}^{m_1-1}\left|I(\cP)-I(\cP\setminus\{i\})\right|^{1/\theta}
&\le C^{1/\theta}\sum_{i=1}^{m_1-1}
w(t^1_{i-1},t^1_i)^{1/p\theta}w(t^1_i,t^1_{i+1})^{1/q\theta}\\
&\le C^{1/\theta}\left(\sum_{i=1}^{m_1-1}
w(t^1_{i-1},t^1_i)\right)^{1/p\theta}
\left(\sum_{i=1}^{m_1-1}w(t^1_i,t^1_{i+1})\right)^{1/q\theta}\\
&\le C^{1/\theta} w(s_1,t_1).
\end{align*}
It is readily apparent that there exists $i$ such that
\begin{align*}
 \left|I(\cP)-I(\cP\setminus\{i\})\right|&\le
C\left(\frac{1}{m_1-1}\right)^{\theta}w(s_1,t_1)^{\theta}.
\end{align*}
Repeating this procedure, we arrive at
\[
 |I(\cP)|\le C\sum_{k=1}^{m_1-1}\frac{1}{k^{\theta}}
w(s_1,t_1)^{\theta}\le
C\zeta(\theta)w(s_1,t_1)^{\theta},
\]
which implies the desired estimate.

Next, we prove the general case.
Let $K$ be a natural number such that $0\le K\le N$.
Also, let
\begin{multline*}
	\tilde{\phi}_K(u_{K+1},\ldots,u_N,v_{K+1},\ldots, v_N)\\
	=
		\sum_{i_1=1}^{m_1}
		\cdots
		\sum_{i_K=1}^{m_K}
			\phi
				\left(
					\prod_{r=1}^K
						[s_r,t^r_{i_r-1}],
					u_{K+1},\ldots,u_N,
					\prod_{r=1}^K [t^r_{i_r-1},t^r_{i_r}],
					v_{K+1},\ldots,v_N
				\right).
\end{multline*}
Note $\tilde{\phi}_0=\phi$ and
$\tilde{\phi}_N$ is a constant function and equal to the left-hand side of \eqref{eq8940181}.
Here, we prove that, 
for all $0\le K\le N$ and $u_k,u'_k, v_k, v_k'\in\cP_k$ with $u_k<u'_k$ and $v_k<v_k'$
($K+1\leq k\leq N$),
\begin{align*}
	\left|
		\tilde{\phi}_K
			\left(
				\prod_{k=K+1}^N[u_{k},u'_{k}]\times \prod_{k=K+1}^N[v_k,v'_k]
			\right)
	\right|
	\le
		C\zeta(\theta)^K
		\prod_{r=K+1}^N
			\{
				w(u_r,u'_r)^{\frac{1}{p}}
				w(v_r,v'_r)^\frac{1}{q}
			\}
		\prod_{r=1}^Kw(s_r,t_r)^\theta,
\end{align*}
by induction on $K$.
The case $K=N$ is our conclusion.
The case $K=0$ is the assumption.
Next, we assume the case of $K$ and show the case of $K+1$.
In this case, we consider the function
\begin{align*}
	(u_{K+1},v_{K+1})
	\mapsto
		\varphi(u_{K+1},v_{K+1})
		=
			\tilde{\phi}_K
				\left(
					u_{K+1},
					\prod_{r=K+2}^N[u_{r},u'_{r}],\,
					v_{K+1},
					\prod_{r=K+2}^N[v_{r},v_{r}']
				\right).
\end{align*}
Then by the assumption of the induction, we have
\begin{multline*}
	\left|
		\varphi
		\left(
			[u_{K+1},u'_{K+1}]\times [v_{K+1},v'_{K+1}]
		\right)
	\right|
	\le
		C\zeta(\theta)^K
		\left(
			\prod_{r=K+2}^N
			\{
				w(u_r,u'_r)^{\frac{1}{p}}
				w(v_r,v'_r)^\frac{1}{q}
			\}
			\prod_{r=1}^Kw(s_r,t_r)^\theta
		\right)\\
		\times
		w(u_{K+1},u'_{K+1})^{1/p}w(v_{K+1},v'_{K+1})^{1/q}.
\end{multline*}
Hence, by using the case $N=1$,
we can complete the proof of the case $K$.
\end{proof}

We use the following Proposition in the proof of
Lemma~\ref{estimate of integration by parts}.

\begin{proposition}\label{estimate of fgh}
	Assume that $1<p,q<\infty$ satisfy
$\frac{1}{p}+\frac{1}{q}>1$ and set $\frac{1}{p}+\frac{1}{q}=\theta$.
	Let $N\in\mathbb{N}$.
	Let $K, L$ be non-negative integers such that
\[
 0\leq K, L\leq N,\quad 0\le M\le \min(K,L),\quad L-M\le N-K.
\]
Let $f\colon [0,1]^{N-K} \to \RR$
and $g\colon [0,1]^{N-L}\to \RR$ be continuous functions
satisfying 
\[
  V_p(f ; [0,1]^{N-K})<\infty,\quad
V_q(g; [0,1]^{N-L})<\infty.
\]
Furthermore, let
 $\phi\colon [0,1]^{K+L}\to\RR$ 
be continuous functions satisfying
that there exists a positive constant
$C$ such that the following condition holds:
there exist a positive constant $C$ and	a control function $w$ such that
	\begin{align*}
		\left|
			\phi
			\left(
				\prod_{k=1}^K
					[u_k,u'_k]
				\times
				\prod_{l=1}^L
					[v_l,v'_l]
			\right)
		\right|
		\leq
			 C
\prod_{k=1}^Kw(u_k,u'_k)^{\frac{1}{p}}
\prod_{l=1}^Lw(v_l,v'_l)^{\frac{1}{q}}
	\end{align*}
for all $0\le u_r<v_r\le 1, 0\le u'_r<v'_r\le 1$
$(1\leq r\leq K)$.
$0\leq v_l<v'_l\leq 1$ $(l\in L)$.
Let 
	\begin{align*}
\Phi(u_1,\dots,u_N,v_1,\dots,v_N)
		&=
f(u_{M+1},\dots,u_L, u_{L+1},\ldots, u_{M+N-K})\nonumber\\
&\quad \times \phi(u_{1},\ldots,u_M,u_{M+N-K+1},\ldots,u_{N},
v_{1},\dots, v_M, v_{M+1},\ldots, v_L)\nonumber\\
&\quad \times g(v_{L+1},\dots, v_{M+N-K},v_{M+N-K+1},\cdots,v_{N}).
	\end{align*}
	For each $0\leq r\leq N$, let us consider
a partition 
\[
 \cP_r : 0\le s_r=t^r_0<t^r_1<\dots<t^r_{m_r}=t_r\leq 1.
\]
	Then we have
\begin{multline*}
	\left|
		\sum_{i_1=1}^{m_1}
		\dots
		\sum_{i_N=1}^{m_N}
			\Phi
			\left(
				\prod_{r=1}^N
					[s_r,t^r_{i_r-1}]
			\times	\prod_{r=1}^N
					[t^r_{i_r-1},t^r_{i_r}]
			\right)
	\right|
	\le	
		C_{N,p,q}
		C
		V_p\left(f ; \prod_{r=M+1}^{M+N-K}[s_{r},t_r]\right)\\
	\times
		V_q\left(g ; \prod_{r=L+1}^N[s_r,t_r]\right)
		\prod_{r=1}^M w(s_r,t_r)^{\theta}
		\prod_{r=M+N-K+1}^N w(s_r,t_r)^{\frac{1}{p}}
		\prod_{r=M+1}^L w(s_r,t_r)^{\frac{1}{q}},
\end{multline*}
where $C_{N,p,q}$ is a positive constant which depends only on $N, p, q$.
\end{proposition}

\begin{proof}
We restrict the variables take values in the partition points as follows:
\[
  u_i\in \cP_i,\quad v_i\in \cP_i, \quad 1\le i\le N.
\]

For each $1\le i\le M$, let
\begin{multline*}
	\hat{\phi}(u_{M+N-K+1},\ldots,u_{N},v_{M+1},\ldots, v_L)
	=
		\sum_{\substack{1\le i_r\le m_r,\\1\le r\le M}}
			\phi
				\Bigl(
					t^1_{i_1-1},\ldots,t^M_{i_M-1},
					u_{M+N-K+1},\ldots,u_N,\\
					[t^1_{i_1-1},t^1_{i_1}]
					\times
					\cdots
					\times
					[t^M_{i_M-1},t^M_{i_M}],
					v_{M+1},\ldots,v_L
				\Bigr).
\end{multline*}
By 
Proposition~\ref{a multidimensional young integral} (1),
we have
\begin{multline}
	\label{estimate of phihat}
		\left|
			\hat{\phi}
				\left(
					\prod_{r=M+N-K+1}^N[u_r,u'_r]
					\times
					\prod_{r=M+1}^L[v_r,v'_r]
				\right)
		\right| \\
		\le
			C\zeta(\theta)^M
			\prod_{r=1}^M w(s_r,t_r)^{\theta}
			\prod_{r=M+N-K+1}^N w(u_r,u'_r)^{\frac{1}{p}}
			\prod_{r=M+1}^L w(v_r,v'_r)^{\frac{1}{q}}.
\end{multline}
Next, we set 
\begin{align*}
 	\hat{\Phi}(u_{M+1},\dots,u_N,v_{M+1},\dots,v_N)
	&=
		f(u_{M+1},\dots,u_L, u_{L+1},\ldots, u_{M+N-K})\\
	&\phantom{=}\qquad
		\times
		\hat{\phi}(u_{M+N-K+1},\ldots,u_{N},v_{M+1},\ldots, v_L)
		g(v_{L+1},\dots, v_{N}).
\end{align*}
Therefore, it holds that
\begin{multline}
	\sum_{i_1=1}^{m_1}
	\dots
	\sum_{i_N=1}^{m_N}
		\Phi
			\left(
				\prod_{r=1}^N
					[s_r,t^r_{i_r-1}]
			\times	\prod_{r=1}^N
					[t^r_{i_r-1},t^r_{i_r}]
			\right)\\
	=
		\sum_{\substack{1\leq i_r\leq m_r,\\ M+1\le r\le N}}
			\hat{\Phi}
			\left(
				\prod_{r=M+1}^N
					[s_r,t^r_{i_r-1}]
			\times	\prod_{r=M+1}^N
					[t^r_{i_r-1},t^r_{i_r}]
			\right).
\label{identity for phi}
\end{multline}
It is sufficient to estimate the quantity on the right-hand side in the equation
above.
Let
\begin{align*}
\Psi(u_{L+1},\ldots,u_N)&=
 \sum_{\substack{1\le i_r\le m_r,\\M+1\le r \le L}}
f\left([s_{M+1},t^{M+1}_{i_{M+1}-1}]\times\cdots\times [s_{L},t^L_{i_L-1}]
,u_{L+1},\ldots,u_{M+N-K}\right)\\
&\quad\quad \times
\hat{\phi}\left(u_{M+N-K+1},\ldots,u_N,[t^{M+1}_{i_{M+1}-1},t^M_{i_M}]
\times\cdots\times [t^L_{i_L-1},t^L_{i_L}]\right).
\end{align*}
Then we can rewrite 
\begin{align*}
& \text{The right-hand side of (\ref{identity for phi})}\\
&=
\sum_{\substack{1\le i_{r}\le m_r,\\L+1\le r\le N}}
\Psi\left([s_{L+1},t^{L+1}_{i_{L+1}-1}]\times\cdots\times
[s_N,t^N_{i_N-1}]\right)
g\left(
[t^{L+1}_{i_{L+1}-1},t^{L+1}_{i_L}]\times\cdots\times
[t^N_{i_N-1},t^N_{i_N}]
\right).
\end{align*}
Let $\vep$ be a positive number such that
$\frac{1}{p+\vep}+\frac{1}{q}>1$.
By Theorem~\ref{Towghi},
\begin{align*}
 \left|\text{The right-hand side of (\ref{identity for phi})}\right|
&\le
CV_{p+\vep}\Bigg(\Psi ; \left(\prod_{r=L+1}^N[s_r,t_r]\right)_{\cP}
\Bigg)
V_q\left(g ; \prod_{r=L+1}^N[s_r,t_r]\right).
\end{align*}
Therefore, we estimate of the norm of $\Psi$.
For $u_r<u'_r, u_r\in \cP_r$ $(L+1\le r\le N)$,
we have
\begin{align*}
& \left|\Psi\left([u_{L+1},u'_{L+1}]\times
\cdots\times[u_N,u_N']\right)\right|\\
&\le
C V_{p+\vep}\left(f\left(\cdots,
[u_{L+1},u'_{L+1}]\times\cdots\times
[u_{M+N-K},u_{M+N-K}']
\right); \prod_{r=M+1}^L[s_r,t_r]\right)\\
&\quad\times V_q\Bigg(\hat{\phi}
\left([u_{M+N-K+1},u_{M+N-K+1}']\times
\cdots\times[u_N,u_N'],\cdots\right);
\left(\prod_{r=M+1}^L[s_r,t_r]\right)_{\cP}\Bigg)
\end{align*}
and
\begin{multline*}
	V_q
		\Bigg(
			\hat{\phi}
				(
					[u_{M+N-K+1},u_{M+N-K+1}']
					\times
					\cdots
					\times
					[u_N,u_N'],
					\dots
				);
			\left(\prod_{r=M+1}^L[s_r,t_r]\right)_{\cP}
		\Bigg)\\
\le
C\zeta(\theta)^M
\prod_{r=1}^M w(s_r,t_r)^{\theta}
\prod_{r=M+N-K+1}^N w(u_r,u'_r)^{\frac{1}{p}}
\prod_{r=M+1}^L w(s_r,t_r)^{\frac{1}{q}},
\end{multline*}
where we have used (\ref{estimate of phihat}).
Therefore, we obtain
\begin{multline*}
	V_{p+\vep}\left(\varphi ; \prod_{r=L+1}^N[s_r,t_r]\right)
	\le
		C\zeta(\theta)^M
		\|f\|_{(p+\vep)\hyp var; \prod_{r=M+1}^{M+N-K}[s_r,t_r]}
		\prod_{r=1}^M w(s_r,t_r)^{\theta} \\
		\times
		\prod_{r=M+N-K+1}^N w(s_r,t_r)^{\frac{1}{p}}
		\prod_{r=M+1}^L w(s_r,t_r)^{\theta}
		\prod_{r=M+N-K+1}^N w(s_r,t_r)^{\frac{1}{p}}
		\prod_{r=M+1}^L w(s_r,t_r)^{\frac{1}{q}},
\end{multline*}
where we have used Theorem~\ref{FV} in the first inequality
and we complete the proof.
\end{proof}

The following lemma is used in the proof of Lemma \ref{covariance estimate of Km}.
The beginning of this section presents the notation used in this lemma and its proof, 
particularly \eqref{eq489108901} and \eqref{eq4789317891}.
\begin{lemma}\label{p-variation norm for product}
Let $q>p\ge 1$. Let $f$ and $g$ be real-valued continuous
functions on $[0,1]^n$.
Then the following estimates hold.
\begin{enumerate}
	\item	If there are no common variables of $f$ and $g$,
			for instance, $f=f(t_1,\ldots,t_k)$ and $g=g(t_{k+1},\ldots,t_n)$ $(1\le k\le n)$,
			then it holds that
			\begin{align}
				V_p(fg ; [0,1]^n)
				&\le
					V_p(f ; [0,1]^k)
					V_p(g ; [0,1]^{n-k}).
				\label{p-variation of fg1}
			\end{align}
	\item	In general, we have
			\begin{align}
				\label{p-variation of fg2}
				V_q(fg ; [0,1]^n)
				&\le
					C
					\bar{V}_p(f,[0,1]^n)
					\bar{V}_p(g,[0,1]^n),\\
				\label{p-variation of fg3}
				\bar{V}_q(fg ; [0,1]^n)
				&\le
					C
					\bar{V}_p(f;[0,1]^n)
					\bar{V}_p(g;[0,1]^n).
			\end{align}
\end{enumerate}
\end{lemma}
\begin{proof}
	The estimate of (\ref{p-variation of fg1}) is trivial by definition.
	We prove (\ref{p-variation of fg2}).
	(\ref{p-variation of fg3}) follows from (\ref{p-variation of fg2}).
	Let $0\le t_i^1<t_i^2\le 1$ ($1\le i\le n$).
	Noting that
	\begin{multline*}
		(fg)(t_1^{a_1},\ldots, t_{n-1}^{a_{n-1}}, t_n^2)
		-
		(fg)(t_1^{a_1},\ldots, t_{n-1}^{a_{n-1}}, t_n^1)\\
		=
			f(t_1^{a_1},\ldots, t_{n-1}^{a_{n-1}},t^2_n)
			g(t_1^{a_1},\ldots, t_{n-1}^{a_{n-1}},[t_n^1,t_n^2])
			+
			f(t_1^{a_1},\ldots, t_{n-1}^{a_{n-1}},[t^1_n,t^2_n])
			g(t_1^{a_1},\ldots, t_{n-1}^{a_{n-1}}, t_n^1),
	\end{multline*}
	we have
\begin{multline*}
	(fg)\Bigg(\prod_{i=1}^n[t_i^{1},t_i^{2}]\Bigg)
	=
		\sum_{a_i=1,2 ; 1\le i\le n}
			(-1)^{\sum_{i=1}^na_i}(fg)(t_1^{a_1},\ldots,t_n^{a_n})\\
	\begin{aligned}
		&=
			\sum_{a_i=1,2; 1\le i\le n-1}
				(-1)^{\sum_{i=1}^{n-1}a_i}
					\left(
						(fg)(t_1^{a_1},\ldots, t_{n-1}^{a_{n-1}}, t_n^2)
						-
						(fg)(t_1^{a_1},\ldots, t_{n-1}^{a_{n-1}}, t_n^1)
					\right)\\
		&=
			\sum_{a_i=1,2; 1\le i\le n-1}
					(-1)^{\sum_{i=1}^{n-1}a_i}
					f(t_1^{a_1},\ldots, t_{n-1}^{a_{n-1}},t^2_n)
					g(t_1^{a_1},\ldots, t_{n-1}^{a_{n-1}},[t_n^1,t_n^2]) \\
		&
		\qquad
		\qquad
		 \qquad
			+
					\sum_{a_i=1,2; 1\le i\le n-1}
					(-1)^{\sum_{i=1}^{n-1}a_i}
					f(t_1^{a_1},\ldots, t_{n-1}^{a_{n-1}},[t^1_n,t^2_n])
					g(t_1^{a_1},\ldots, t_{n-1}^{a_{n-1}}, t_n^1)\\
		&=
			(f(\cdot,t^2_n)g(\cdot,[t_n^1,t_n^2]))
				\Bigg(\prod_{i=1}^{n-1}[t_i^{1},t_i^{2}]\Bigg)
			+
			(f(\cdot,[t^1_n,t^2_n])g(\cdot,t_n^1))
				\Bigg(\prod_{i=1}^{n-1}[t_i^{1},t_i^{2}]\Bigg).
	\end{aligned}
\end{multline*}
Then, iterating this calculation, we have
\begin{align*}
	(fg)\Bigg(\prod_{i=1}^n[t_i^{1},t_i^{2}]\Bigg)
	&=
		\sum_{A\subset \{1,\ldots,n\}}
			f(t^2_a, [t^1_r,t^2_r] ; a\in A, r\in A^c)
			g(t^1_r, [t^1_a,t^2_a] ; r\in A^c, a\in A).
\end{align*}
Here and hereafter, $A$ can be an empty set or $\{1,\ldots,n\}$.
This implies 
\begin{align*}
	\left|
		(fg)\Bigg(\prod_{i=1}^n[t_i^{1},t_i^{2}]\Bigg)
	\right|^q
	&\leq
		C
		\sum_{A\subset \{1,\ldots,n\}}
			|f(t^2_a, [t^1_r,t^2_r] ; a\in A, r\in A^c)|^q
			|g(t^1_r, [t^1_a,t^2_a] ; r\in A^c, a\in A)|^q.
\end{align*}

Here we have
\begin{align*}
	|f(t^2_a, [t^1_r,t^2_r];a\in A, r\in A^c)|
	&\leq
		\sum_{B\subset A}
			\|f(0_a; a\in A\setminus B)\|_{q\hyp var;[0,1]^{|B|}\times \prod_{r\in A^c}[t^1_r,t^2_r]},
\end{align*}
where $f(0_a; a\in A\setminus B)$ is a $(|A^c|+|B|)$-variables function defined by
\begin{align*}
	f(0_a; a\in A\setminus B)(u_b,v_r;b\in B,r\in A^c)
	&=
		f(0_a,u_b, v_r;a\in A\setminus B, b\in B,r\in A^c ),
\end{align*}
which is similar notation with \eqref{eq489108901}.
The estimate above follows from
\begin{align*}
	f(t^2_a, [t^1_r,t^2_r];a\in A, r\in A^c)
	&=
		\sum_{B\subset A}
			f(0_a, [0,t^2_b], [t^1_r,t^2_r] ; a\in A\setminus B, b\in B, r\in A^c)\\
	&=
		\sum_{B\subset A}
			f(0_a ; a\in A\setminus B)([0,t^2_b],[t^1_r,t^2_r];b\in B, r\in A^c).
\end{align*}

Let us consider a grid-like partition $\cP=\cP_1\times\dots\times\cP_n$,
where $\cP_i\colon 0=t_i^0<\cdots<t_i^{m_i}=1$ ($1\le i\le n$).
By the above and the definition, for $A\subset\{1,\dots,n\}$, we have
\begin{multline*}
	\sum_{\substack{1\le k_i\le m_i,\\1\le i\le n}}
		|f(t^{k_a}_a, [t^{k_r-1}_r,t^{k_r}_r] ; a\in A, r\in A^c)|^q
		|g(t^{k_r}_r, [t^{k_a-1}_a,t^{k_a-1}_a] ; r\in A^c, a\in A)|^q\\
	\begin{aligned}
		&\leq
			\sum_{\substack{1\le k_i\le m_i,\\1\le i\le n}}
				\sum_{B\subset A}
					\|f(0_a; a\in A\setminus B)\|_{q\hyp var;[0,1]^{|B|}\times \prod_{r\in A^c}[t^{k_r-1}_r,t^{k_r}_r]}^q\\
		&\phantom{=}
				\qquad
				\qquad
				\qquad
				\qquad
				\qquad
		\times
				\sum_{C\subset A^c}
					\|g(0_r; r\in A^c\setminus C)\|_{q\hyp var;[0,1]^{|C|}\times \prod_{a\in A}[t^{k_a-1}_a,t^{k_a}_a]}^q\\
		&\leq
			\sum_{B\subset A}
				\|f(0_a; a\in A\setminus B)\|_{q\hyp var;[0,1]^{|B|}\times[0,1]^{|A^c|}}^q
			\sum_{C\subset A^c}
				\|g(0_r; r\in A^c\setminus C)\|_{q\hyp var;[0,1]^{|C|}\times[0,1]^{|A|}}^q\\
		&\leq
			C
			\bar{V}_p(f,[0,1]^n)^q
			\bar{V}_p(g,[0,1]^n)^q.
	\end{aligned}
\end{multline*}
Here noting $q>p$, we used Theorem~\ref{FV}.
This implies
\begin{align*}
	V_q((fg)|_{\cP};[0,1]^n_{\cP})^q
		&\leq
			C
			\sum_{A\subset \{1,\ldots,n\}}
			\sum_{\substack{1\le k_i\le m_i,\\1\le i\le n}}
				|f(t^{k_a}_a, [t^{k_r-1}_r,t^{k_r}_r] ; a\in A, r\in A^c)|^q\\
		&\phantom{=}
				\qquad
				\qquad
				\qquad
				\qquad
				\qquad
				\times
				|g(t^{k_r}_r, [t^{k_a-1}_a,t^{k_a-1}_a] ; r\in A^c, a\in A)|^q\\
		&\le
			C
			\bar{V}_p(f,[0,1]^n)^q
			\bar{V}_p(g,[0,1]^n)^q.
\end{align*}
The proof is completed.
\end{proof}

\address{
Shigeki Aida\\
Graduate School of Mathematical Sciences,\\
The University of Tokyo,\\
Meguro-ku, Tokyo, 153-8914, Japan}
{aida@ms.u-tokyo.ac.jp}

\address{
Nobuaki Naganuma\\
Faculty of Advanced Science and Technology,\\
Kumamoto University,\\
Kumamoto city, Kumamoto, 860-8555, Japan}
{naganuma@kumamoto-u.ac.jp}


\begin{thebibliography}{99}
	%
   \bibitem{aida-naganuma2020}
	Aida, S., Naganuma, N.:
	 Error analysis for approximations to
	one-dimensional SDEs via the perturbation method.
	 Osaka J. Math. 57(2), 381--424 (2020).
   
   
   \bibitem{aida-naganuma2023approach}
	Aida, S., Naganuma, N.:
   An interpolation of discrete rough differential
   equations
   and its applications to
   analysis of error distributions. arXiv:2302.03912.
   
   \bibitem{BreuerMajor1983}
	   Breuer, P., Major P.:
	   Central limit theorems for nonlinear functionals of Gaussian fields.
	   J. Multivariate Anal. 13(3), 425--441 (1983).
	%
   \bibitem{BurdzySwanson2010}
   Burdzy, K., Swanson, J.:
   A change of variable formula with It\^o correction term.
   Ann. Probab. 38(5), 1817--1869 (2010).
   %
   \bibitem{cll2013}
   Cass, T., Litterer, C., Lyons,T.:
   Integrability and tail estimates for Gaussian rough differential equations.
   Ann. Probab. 41(4), 3026--3050 (2013).
   %
   %
   \bibitem{friz-hairer}
   Friz, P., Hairer, M.:
	 A Course on Rough Paths\,
	 With an Introduction to Regularity Structures,
   Second edition.
	 Universitext, Springer, Cham, (2020).
   %
   \bibitem{friz-victoir}
   Friz, P., Victoir, N.:
	 Multidimensional Stochastic Processes as Rough Paths\,
	 {\small Theory and Applications}.
	 Cambridge Studies in Advanced Mathematics, 120,
	 Cambridge University Press (2010).
   %
   \bibitem{friz-victoir2011}
	Friz, P., Victoir, N.:
	 A note on higher dimensional $p$-variation.
	Electron. J. Probab. 16(68), 1880--1899 (2011).
   
   \bibitem{GradinaruNourdin2009}
	Gradinaru, M., Nourdin, I.:
	 Milstein's type schemes for fractional SDEs.
	Ann. Inst. Henri Poincar\'e Probab. Stat. 45(4), 1085--1098 (2009).
   %
   \bibitem{gubinelli}
   Gubinelli, M.:
	 Controlling rough paths.
	J. Funct. Anal. 216(1), 86--140 (2004).
	%
   \bibitem{hairer-pillai}
	Hairer, M., Pillai, N.:
	 Regularity of laws and ergodicity of hypoelliptic SDEs
	driven by rough paths. Ann. Probab. 41(4),
	 2544--2598 (2013).
   %
   \bibitem{hu-liu-nualart2016}
	Hu, Y., Liu, Y., Nualart, D.:
	 Rate of convergence and asymptotic error distribution of Euler 
	approximation schemes for fractional diffusions. 
	Ann. Appl. Probab. 26(2), 1147--1207 (2016).
	%
   \bibitem{hu-liu-nualart2021}
	Hu, Y., Liu, Y., Nualart, D.:
	 Crank--Nicolson scheme for stochastic differential equations driven 
	by fractional Brownian motions. 
	Ann. Appl. Probab. 31(1), 39--83 (2021).
   %
   \bibitem{inahama}
   Inahama, Y.:
	 Malliavin differentiability of solutions of rough differential
	equations.
	J. Funct. Anal. 267(5), 1566--1584 (2014).
   %
   \bibitem{liu-tindel}
	Liu, Y., Tindel, S.:
	 First-order Euler scheme for SDEs driven by fractional Brownian motions:
	the rough case. Ann. Appl. Probab. 29(2), 758--826 (2019).
   %
   \bibitem{liu-tindel2020}
   Liu, Y., Tindel, S.:
	 Discrete rough paths and limit theorems. 
	Ann. Inst. Henri Poincar\'e Probab. Stat. 56(3), 1730--1774 (2020).
   %
   \bibitem{lyons98}
	Lyons, T.:
	 Differential equations driven by rough signals.
	 Rev. Mat. Iberoamer. 14, 215--310 (1998).
   %
\bibitem{lq}
  T.~Lyons and Z.~Qian,
  System control and rough paths.
  Oxford Math. Monogr.
	Oxford Sci. Publ.
	Oxford University Press, Oxford, (2002).
	%
   \bibitem{naganuma2015}
   Naganuma, N.:
	 Asymptotic error distributions of the Crank--Nicholson scheme for
	SDEs driven by fractional Brownian motion.
	J. Theoret. Probab. 28(3), 1082--1124 (2015).
   %
   \bibitem{neuenkirch-tindel-unterberger2010}
	Neuenkirch, A., Tindel, S., Unterberger, J.:
	 Discretizing the fractional L\'evy area. Stochastic Process. Appl. 120(2),
   223--254 (2010).
   %
   \bibitem{NourdinNualartTudor2010}
   Nourdin, I., Nualart D., Tudor, C.A.: 
	 Central and non-central limit theorems for weighted 
	power variations of fractional Brownian motion.
	 Ann. Inst. Henri Poincar\'e Probab. Stat. 46, 1055--1079 (2010).
   %
   \bibitem{nourdin-peccati}
	Nourdin, I., Peccati, G.:
	 Normal approximations with Malliavin calculus.
	From Stein's method to universality.
	Cambridge Tracts in Mathematics, 192.
	Cambridge University Press, Cambridge (2012).
   %
   \bibitem{nualart}
   Nualart, D.:
	 The Malliavin calculus and related topics. Second edition. 
	Probability and its Applications (New York). Springer-Verlag, Berlin (2006). 
   %
   \bibitem{NualartPeccati2005}
   Nualart D., Peccati, G.:
   Central limit theorems for sequences of multiple stochastic integrals.
   Ann. Probab. 33(1), 177--193 (2005).
   %
   \bibitem{PeccatiTudor2005}
	Peccati, G., Tudor, C.A.:
	Gaussian limits for vector-valued multiple stochastic integrals.
	In {\em S{\'e}minaire de {P}robabilit{\'e}s {XXXVIII}}, vol. 1857
	   of {\em Lecture Notes in Math.}, 247--262. Springer, Berlin (2005).
	%
   \bibitem{shigekawa}
	Shigekawa, I.:
	 Stochastic analysis,
   Translations of Mathematical Monographs.
	224. Iwanami Series in Modern Mathematics. 
	American Mathematical Society, Providence, RI, (2004).
   %
   \bibitem{towghi}
	Towghi, N.:
	 Multidimensional extension of L. C. Young's inequality.
	JIPAM. J. Inequal. Pure Appl. Math. 3(2), Article 22, 13pp, (2002).
   \end{thebibliography}
\end{document}